\documentclass[11pt,dvipsnames]{amsart}
\usepackage{amscd,amssymb,longtable, rotating, lscape}
\usepackage[matrix,arrow,curve]{xy}
\usepackage{supertabular}
\usepackage{setspace}
\usepackage{hyperref}
\usepackage{ifsym}

\usepackage{alltt}

\usepackage{float}
\usepackage{pgf,tikz}
\usepackage{mathrsfs}
\usetikzlibrary{arrows}

\usepackage{fancyhdr}
\usepackage{supertabular}
\usepackage{setspace}
\usepackage{enumerate}
\usepackage{amsfonts}
\usepackage{cite}
\usepackage{pifont}

\usepackage{soul}

\theoremstyle{definition}
\newtheorem{theorem}[equation]{Theorem}
\newtheorem*{theorem*}{Theorem}
\newtheorem{lemma}[equation]{Lemma}
\newtheorem{corollary}[equation]{Corollary}
\newtheorem{question}[equation]{Question}

\newtheorem{example}[equation]{Example}
\newtheorem{definition}[equation]{Definition}
\newtheorem{proposition}[equation]{Proposition}
\newtheorem*{definition*}{Definition}

\theoremstyle{remark}
\newtheorem{remark}[equation]{Remark}

\makeatletter\@addtoreset{equation}{section}
\makeatletter\@addtoreset{section}{part}

\makeatother

\newcommand{\mumu}{\boldsymbol{\mu}}

\def \Aut {\mathrm{Aut}}
\def \Cl {\mathrm{Cl}\,}
\def \Pic {\mathrm{Pic}\,}

\title{Double Veronese cones with 28 nodes}

\author{Hamid Ahmadinezhad, Ivan Cheltsov, Jihun Park, Constantin Shramov}

\begin{document}

\begin{abstract}
We study nodal del Pezzo 3-folds of degree $1$
(also known as double Veronese cones)
with~$28$ singularities, which is the maximal possible number
of singularities for such varieties. We show that they are in one-to-one
correspondence with smooth plane quartics and use this correspondence to study their automorphism groups.
As an application, we find all $G$-birationally rigid varieties of this kind, and
construct an infinite number of non-conjugate embeddings of the group $\mathfrak{S}_4$ into the space Cremona group.
\end{abstract}

\address{\emph{Hamid Ahmadinezhad}
\newline
\textnormal{Department of Mathematical Sciences, Loughborough University
\newline
Loughborough LE11 3TU, UK
\newline
 \texttt{h.ahmadinezhad@lboro.ac.uk}}}

\address{ \emph{Ivan Cheltsov}\newline \textnormal{School of Mathematics, The
University of Edinburgh
\newline \medskip  Edinburgh EH9 3JZ, UK
\newline
National Research University Higher School of Economics,
Laboratory of Algebraic Geometry,
\newline
6 Usacheva street, Moscow, 117312, Russia
\newline
\texttt{I.Cheltsov@ed.ac.uk}}}

\address{ \emph{Jihun Park}\newline \textnormal{Center for Geometry and Physics, Institute for Basic Science
\newline \medskip 77 Cheongam-ro, Nam-gu, Pohang, Gyeongbuk, 37673, Korea \newline
Department of
Mathematics, POSTECH
\newline
77 Cheongam-ro, Nam-gu, Pohang, Gyeongbuk,  37673, Korea \newline
\texttt{wlog@postech.ac.kr}}}

\address{\emph{Constantin Shramov}\newline \textnormal{Steklov Mathematical Institute of Russian Academy of Sciences
\newline
\medskip 8 Gubkina street, Moscow, 119991, Russian
\newline
National Research University Higher School of Economics, Laboratory of Algebraic Geometry,
\newline
6 Usacheva street, Moscow, 117312, Russia
\newline \texttt{costya.shramov@gmail.com }}}

\maketitle

\tableofcontents

\section{Introduction}
\label{section:intro}

It is a common experience that varieties with the maximal possible
number of singularities in a given class  have fascinating  geometric properties.
Classical examples of such varieties include the Burkhardt quartic (see~\cite{JSBV},
\cite{CheltsovPrzyjalkowskiShramov-BarBur}), the Segre cubic, and many others.
The Segre cubic, which is the unique three-dimensional cubic hypersurface with $10$ isolated
singular points, can be thought of as a \emph{del Pezzo 3-fold}. Recall that a del Pezzo 3-fold is a
Fano 3-fold with at most terminal
Gorenstein singularities
whose canonical class is divisible by $2$ in the Picard group; in particular, the singularities of such 3-folds
are always isolated.

Del Pezzo 3-folds are classified according to their \emph{degrees}, which are defined as the cubes of the half-anticanonical classes.
Such varieties with a maximal possible number of singularities were
described in~\cite[Theorem~7.1]{Prokhorov}.
The smaller degrees such del Pezzo 3-folds have, the more intriguing their geometry is.
Del Pezzo 3-folds of degree $4$ can have at most six singularities,
and there exists a unique such variety with exactly six singular points.
It has been studied in \cite{Avilov-IntTwoQuadrics}.
There are many works devoted to the geometry of the Segre cubic (see~\mbox{\cite[\S2]{Dolgachev-Segre}}, \cite{Finkelnberg},
\cite[\S5]{CheltsovShramov-A6}, \cite{Avilov-Segre}, and \cite{Avilov-RealSegre}).
Del Pezzo 3-folds of degree~$2$ with~$16$ singular points, which is the maximal possible
number of singularities in this case, are
quartic double solids branched along Kummer quartic surfaces.
Thus their properties are closely related to those of
the corresponding abelian surfaces. In particular, such del Pezzo 3-folds are in one-to-one correspondence with smooth curves of genus~$2$
(see for instance~\mbox{\cite[Propositions~1.2(i) and~1.3]{Keum}}).

The current paper  studies the geometry
of del Pezzo 3-folds of degree~$1$ with $28$ singular points,
which is the maximal possible number of singularities for such varieties
(see~\mbox{\cite[Theorem~1.7]{Prokhorov}} and~\mbox{\cite[Remark~1.8]{Prokhorov}}).
Each del Pezzo 3-fold of degree~$1$ is a hypersurface of degree six in the weighted projective space~\mbox{$\mathbb{P}(1,1,1,2,3)$} (see \cite[Theorem~3.2.5(i)]{Isk-Prokh}). This implies that it is a double cover of the cone over the Veronese surface in~$\mathbb{P}^5$.
For this reason,  del Pezzo 3-folds of degree $1$ are often called \emph{double Veronese cones}.

Let $V$ be a double Veronese cone with $28$ singular points.
Then $-K_V\sim 2H$, where $H$ is an ample Cartier divisor on $V$ such that $H^3=1$.
The existence of the following diagram was proven in  \cite{Prokhorov}.
\begin{equation}
\label{equation:Prokhorov-diagram}
\xymatrix{
\widehat{\mathbb{P}^3}\ar@{->}[d]_{\pi}\ar@{->}[rr]^{\phi}&&V\ar@{-->}[d]^{\kappa}\\%
\mathbb{P}^3\ar@{-->}[rr]&&\mathbb{P}^2}
\end{equation}
In this diagram, $\phi$ is a small resolution of all singular points of the 3-fold $V$,
the morphism $\pi$ is the blow up of  $\mathbb{P}^3$ at seven distinct points $P_1,\ldots,P_7$  that satisfy certain explicit
generality conditions, the rational map $\kappa$ is given by the linear system $|H|$,
and the rational map from $\mathbb{P}^3$ to $\mathbb{P}^2$ is given by the linear system (net) of quadrics passing through $P_1,\ldots,P_7$.
Furthermore, Prokhorov implicitly verified
in~\cite{Prokhorov}
that each singular point of $V$ is an isolated ordinary double point (a node).

Note that $V$ is not $\mathbb{Q}$-factorial
since $\mathrm{Cl}(V)\cong\mathbb{Z}^8$ and $\mathrm{Pic}(V)\cong\mathbb{Z}$.
Nevertheless, the class group~$\mathrm{Cl}(V)$ can be naturally equipped with the intersection form:
\begin{equation}
\label{equation:E7}
\big(D_1,D_2\big)=D_1\cdot D_2\cdot H
\end{equation}
for any two Weil divisors $D_1$ and $D_2$ on the 3-fold $V$. Furthermore,
Prokhorov proved in~\mbox{\cite[Theorem~1.7]{Prokhorov}} (cf.~\mbox{\cite[Remark~1.8]{Prokhorov}}) that
the orthogonal complement to the canonical divisor~$K_V$ in~$\mathrm{Cl}(V)$ is the lattice~$\mathbf{E}_7$ (cf.~\mbox{\cite[p.~107]{DolgachevOrtland}}).
The same lattice appears in the Picard group of the double cover of $\mathbb{P}^2$ branched along a smooth plane quartic curve.
This suggests a possible relation between double Veronese cones with $28$ nodes and smooth plane quartic curves.
As far as we know, such a relation was first pointed out in \cite[\S\,IX.6]{DolgachevOrtland}.   A goal  in the present
paper is to make this
relation  precise and detailed.

The net of quadrics in \eqref{equation:Prokhorov-diagram}
has exactly eight distinct base points,
including the points~\mbox{$P_1,\ldots,P_7$}; such a collection of eight points is
called a \emph{regular Cayley octad} (see  Definition~\ref{definition:Aronhold-heptad} for details).
Moreover, the Hessian curve of this net is a smooth plane quartic curve, which we denote by $C$.
One can show that the isomorphism class of the curve $C$ is independent  of  the choice of the commutative diagram~\eqref{equation:Prokhorov-diagram}.
In fact, the curve $C$ is projectively dual
to the discriminant curve of the rational elliptic fibration~$\kappa$
(see~\S\ref{section:1-1} for details).

The first main result of the present paper is

\begin{theorem}\
\label{theorem:one-to-one}
Assigning the curve $C$ to the 3-fold $V$ gives
a one-to-one correspondence between the isomorphism classes of smooth plane quartic curves and the isomorphism classes of $28$-nodal double Veronese cones.
\end{theorem}

Two alternative proofs of Theorem~\ref{theorem:one-to-one} will be presented in this paper.
The first proof starts with a plane quartic curve $C$ and produces an explicit equation that defines the corresponding 3-fold $V$ in the weighted projective space  $\mathbb{P}(1,1,1,2,3)$ (see equation~\eqref{eq:V1} and~\S\ref{section:equations} for details). Another proof uses  machinery from the theory of nets of quadrics to recover a
commutative diagram~\eqref{equation:Prokhorov-diagram}
and check that the resulting 3-fold does not depend on the choices we have to make on the way.

Observe
that for a given $28$-nodal double Veronese cone
the commutative diagram~\eqref{equation:Prokhorov-diagram} is not unique.
However, we can refine Theorem~\ref{theorem:one-to-one} to characterize all such diagrams
in terms of the corresponding smooth plane quartic curve.
Essentially, the diagram~\eqref{equation:Prokhorov-diagram} depends only on the choice of seven distinct points $P_1,\ldots,P_7$ in $\mathbb{P}^3$ up to projective transformations. We will show that there are exactly~$288$  diagrams~\eqref{equation:Prokhorov-diagram} up to projective transformations  of $\mathbb{P}^3$ for a given general $28$-nodal double Veronese cone $V$,
and will also give a way to compute the number of such diagrams for an arbitrary~$V$ (see Corollaries~\ref{corollary:Aronhold-sets} and~\ref{corollary:288-diagram}). Here, the number $288=36\cdot 8$
equals
the number  of even theta characteristics of $C$ times the number of choices for a point in the corresponding regular Cayley octad.
In other words, it is the number of Aronhold systems
on $C$. We refer the reader to~\S\ref{section:theta} for details.

Note that there are other classical constructions that associate certain 3-folds to smooth plane quartics.
We recall some of them  in \S\ref{section:question}.
It would be interesting to find the connections between these varieties and
$28$-nodal double Veronese cones.

The second result of the present paper is the following.

\begin{theorem}
\label{theorem:Aut}
One has
$$
\mathrm{Aut}(V)\cong\mumu_2\times\mathrm{Aut}(C),
$$
where the generator of the cyclic group $\mumu_2$ of order $2$
acts on $V$ by the Galois involution of the double cover of the Veronese cone.
\end{theorem}

We point out that Theorem~\ref{theorem:Aut} can be obtained from \cite[\S\,IX.6]{DolgachevOrtland} and Torelli theorem (see for instance~\cite{Weil}). 
However, we provide a more elementary proof. Another advantage of our proof is its explicit nature, enabling an easier access to details of the group actions on the $3$-folds, which will be useful later on when we study their birational properties.

Note that the commutative diagram \eqref{equation:Prokhorov-diagram} is not necessarily $\mathrm{Aut}(V)$-equivariant.
However, the rational map $\kappa$ in \eqref{equation:Prokhorov-diagram} always is.
This gives a natural action of the group $\mathrm{Aut}(V)$ on the plane $\mathbb{P}^2$.
This action agrees with the direct product structure in Theorem~\ref{theorem:Aut}
(see the proof of Theorem~\ref{theorem:Aut} in~\S\ref{section:1-1}).

The automorphism groups of smooth plane quartic curves are completely classified (see~\mbox{\cite[Theorem~6.5.2]{Dolgachev}}, \cite{Henn}, or \cite{KuribayashiKomiya}).
This gives us the full  list of  the automorphism groups of $28$-nodal double Veronese cones.

\begin{example}
\label{example:V1-S4}
Let $C$ be the plane quartic curve given by
\begin{equation}\label{eq:S4-quartic-general}
x^4+y^4+z^4+\lambda(y^2z^2+x^2z^2+x^2y^2)=0,
\end{equation}
where $\lambda$ is a constant different from  $-2$, $2$, and $-1$ (in these three cases, the curves are singular).
The corresponding $28$-nodal double Veronese cone $V$ is the hypersurface of degree $6$ in~\mbox{$\mathbb{P}(1,1,1,2,3)$}
that is given by
$$
w^2=v^3-g_4(s,t,u)v+g_6(s,t,u),
$$
where $s$, $t$, $u$, $v$ and $w$ are weighted homogeneous coordinates on $\mathbb{P}(1,1,1,2,3)$ whose weights
are equal to $1$, $1$, $1$, $2$ and $3$, respectively. From \S\ref{section:equations} and Appendix~\ref{section:appendix} we obtain
$$
g_4(s,t,u)=\frac{(\lambda^2 + 12)}{3}(s^4 +t^4+u^4)+  \frac{2(\lambda^2+6\lambda)}{3}(t^2u^2+s^2u^2+s^2t^2)
$$
and
\[\begin{split}
g_6(s,t,u)&=\frac{2( - \lambda^3 + 12\lambda^2 + 12\lambda)}{9}(t^4u^2+t^2u^4+s^4u^2+s^2u^4+s^4t^2 + s^2t^4) +\\
&+\frac{2(-\lambda^3 + 36\lambda)}{27}(s^6+t^6+u^6)+\frac{4(8\lambda^3 - 9\lambda^2 + 36 )}{9}s^2t^2u^2.
\end{split}\]
Let $\mathfrak{G}$ be the subgroup in $\mathrm{Aut}(V)$ that is generated by
\[\begin{split}
&[s:t:u:v:w]\mapsto [-s:-t:u:v:w],\\
&[s:t:u:v:w]\mapsto [u:s:t:v:w],\\
&[s:t:u:v:w]\mapsto [t:s:u:v:-w].\\
\end{split}\]
Then $\mathfrak{G}\cong\mathfrak{S}_4$. Let $\tau\in\mathrm{Aut}(V)$ be the Galois involution of the double cover of the Veronese cone.
If~$\lambda\ne 0$ and $\lambda^2+3\lambda+18\ne 0$, then $\tau$ and
$\mathfrak{G}$ generate the group
$$
\mathrm{Aut}(V)\cong\mumu_2\times\mathfrak{S}_4
$$
(see  \cite[Theorem~6.5.2]{Dolgachev}, \cite{Henn}, or \cite{KuribayashiKomiya}).
If~\mbox{$\lambda=0$}, then $C$ is the Fermat quartic, so that
$$
\mathrm{Aut}(V)\cong\mumu_2\times(\mumu_4^2\rtimes\mathfrak{S}_3),
$$
where $\mumu_n$ denotes the cyclic group of order $n$.
Finally, if $\lambda^2+3\lambda+18=0$, then $C$ is isomorphic to the Klein quartic, so that
$$
\mathrm{Aut}(V)\cong\mumu_2\times\mathrm{PSL}_2(\mathbb{F}_7).
$$
\end{example}

It follows from \eqref{equation:Prokhorov-diagram} that $28$-nodal double Veronese cones are rational.
Thus, their automorphism groups are embedded into the space Cremona group, that is the group of birational selfmaps of $\mathbb{P}^3$.

For a given subgroup $G\subset\mathrm{Aut}(V)$,
to study its conjugacy class in the space Cremona group,
one should understand $G$-equivariant birational geometry of the 3-fold $V$.
From this point of view, the most interesting groups $G$ are such that
$\mathrm{rk}\,\mathrm{Cl}(V)^G=1$, and $V$ is $G$-birationally (super-)rigid (see \cite[Definition~3.1.1]{CheltsovShramov-V5}).
In this paper, we prove

\begin{theorem}
\label{theorem:rigid}
Let $V$ be a $28$-nodal double Veronese cone and $G$ be a subgroup of $\mathrm{Aut}(V)$.
Then the  following four conditions are equivalent:
\begin{enumerate}
\item $\mathrm{rk}\,\mathrm{Cl}(V)^G=1$ and $V$ is $G$-birationally rigid;

\item $\mathrm{rk}\,\mathrm{Cl}(V)^G=1$ and $V$ is $G$-birationally super-rigid;

\item $\mathrm{rk}\,\mathrm{Cl}(V)^G=1$ and $G$ does not fix any point in $\mathbb{P}^2$;

\item  $V$ is a double Veronese cone from Example~\ref{example:V1-S4} and $G$ contains the subgroup $\mathfrak{G}$.
\end{enumerate}
\end{theorem}

It is known that the plane Cremona group,
that is the group of birational selfmaps of~$\mathbb{P}^2$, contains
an infinite number of non-conjugate subgroups isomorphic to
the symmetric group~$\mathfrak{S}_4$
(see~\mbox{\cite[\S8]{DolgachevIskovskikh}},
and cf.~\mbox{\cite[Theorem~1.5]{Lucas}}).
Theorem~\ref{theorem:rigid} implies the following result.

\begin{corollary}
\label{corollary:S4}
The space Cremona group  contains an infinite number of non-conjugate subgroups isomorphic to $\mathfrak{S}_4$.
\end{corollary}

It is known from \cite{CheltsovShramov-PSL} and \cite[Theorem~2.2]{Ahmadinezhad} that
the space Cremona group contains at least four non-conjugate subgroups isomorphic to $\mathrm{PSL}_2(\mathbb{F}_7)$.
Applying Theorem~\ref{theorem:rigid} to the case of the Klein quartic curve, we obtain

\begin{corollary}
\label{corollary:Klein}
The space Cremona group contains at least five non-conjugate subgroups isomorphic to $\mathrm{PSL}_2(\mathbb{F}_7)$.
\end{corollary}

Let $V$ be the 3-fold from Example~\ref{example:V1-S4}.
Then $\mathrm{Aut}(V)$ contains a subgroup $\mathfrak{G}^\prime\cong\mathfrak{S}_4$ that is not conjugate to the subgroup $\mathfrak{G}$.
The subgroup $\mathfrak{G}^\prime$ is generated by
\[\begin{split}
&[s:t:u:v:w]\mapsto [-s:-t:u:v:w],\\
&[s:t:u:v:w]\mapsto [u:s:t:v:w],\\
&[s:t:u:v:w]\mapsto [t:s:u:v:w].\\
\end{split}\]
In \S\ref{section:S4-Veronese}, we show that there exists a $\mathfrak{G}^\prime$-equivariant birational map $V\dasharrow X$,
where $X$ is a del Pezzo 3-fold of degree $2$ with $16$ singular points such that
$\mathrm{rk}\,\mathrm{Cl}(X)^{\mathfrak{G}^\prime}=1$.
It would be interesting to determine whether $X$ is $\mathfrak{G}^\prime$-birationally rigid or not (cf. \cite{Avilov-15points}).

We have established a relation between $28$-nodal double Veronese cones and smooth plane quartics.
Despite this establishment, the connection between the $\mathbf{E}_7$ latices appearing in their class groups and the Picard groups
of the double covers of $\mathbb{P}^2$ branched in the corresponding smooth quartic curves does not seem to be  straightforward.
We discuss this phenomenon in \S\ref{section:question}.

\medskip
The plan of the paper is as follows.
In~\S\ref{section:preliminaries}
we collect several auxiliary results
that will be used later in our proofs.
In~\S\ref{section:nets-of-quadrics} we recall some  assertions concerning nets of quadrics in~$\mathbb{P}^3$. The base points of the nets of quadrics described therein will be utilized in order to construct $28$-nodal  double Veronese cones in~\S\ref{section:construction}. This construction is originally due to~\cite{Prokhorov}.
In~\S\ref{section:equations} we construct a $28$-nodal double Veronese cone starting from a smooth plane quartic curve.
We also provide an explicit equation~\eqref{eq:V1} for such a 3-fold in terms of covariants
of the quartic. In~\S\ref{section:1-1} we put the previous results together to establish a one-to-one correspondence between
$28$-nodal double Veronese cones and smooth plane quartics, which proves Theorem~\ref{theorem:one-to-one}.
Also, in~\S\ref{section:1-1} we
study the automorphism groups of $28$-nodal double Veronese cones and  prove Theorem~\ref{theorem:Aut}.
In~\S\ref{section:theta} we
give an alternative proof of Theorem~\ref{theorem:one-to-one}.
Also, in~\S\ref{section:theta} we prove Corollary~\ref{corollary:Aronhold-sets}
which relates commutative diagrams~\eqref{equation:Prokhorov-diagram} with even theta characteristics on the smooth plane quartic.
In~\S\ref{section:S4-Veronese}
we investigate the $\mathfrak{S}_4$-equivariant birational geometry of
double Veronese cones introduced in Example~\ref{example:V1-S4},
and, in particular, describe the action of the group~$\mathfrak{S}_4$ on their Weil divisor class groups.
In~\S\ref{section:rigidity}
we study  $G$-birationally rigid nodal double Veronese cones and
prove Theorem~\ref{theorem:rigid}.
In~\S\ref{section:question} we discuss some open questions
concerning $28$-nodal double Veronese cones.
Finally, in Appendix~\ref{section:appendix} we collect
some formulae allowing to construct explicit equations
of $28$-nodal  double Veronese cones from the equations of plane quartics.

\medskip
\textbf{Acknowledgements.}
The authors are grateful to   Fabrizio Catanese, Igor Dolgachev,
Alexander Kuznetsov, Marco Pacini, and Yuri Prokhorov
for inspiring  discussions.
Part of this work was done during their stay at the Erwin Schr\"odinger International Institute for Mathematics and Physics in Vienna in August 2018.
They are grateful to the institute for excellent working conditions.
Jihun Park would also like to thank Adrien Dubouloz for his hospitality at Institut de Math\'ematiques de Bourgogne in August 2019, where he worked on finalizing the paper.
Ivan Cheltsov was supported by the Royal Society 
grant No.~\mbox{IES\textbackslash R1\textbackslash 180205} 
and the 
HSE University Basic Research Program,
Russian Academic Excellence Project~\mbox{5-100}.
Jihun Park was supported by IBS-R003-D1, Institute for Basic Science in Korea.
Constantin Shramov was supported by
the HSE University Basic Research Program, 
Russian Academic Excellence Project~\mbox{5-100},
by the Young Russian Mathematics award, and by the Foundation for the
Advancement of Theoretical Physics and Mathematics ``BASIS''.

\section{Preliminaries}
\label{section:preliminaries}

In this section we collect several auxiliary results
that will be used later in our proofs.

The following fact is well known to experts. It was pointed out to us by
Alexander Kuznetsov. This will be utilized in \S\ref{section:construction}.

\begin{lemma}[{cf.~\cite[Remark~4.2.7]{CheltsovShramov-V5}}]
\label{lemma:normal-bundle-blow-up}
Let $Y$ be a smooth variety, $C$ be a smooth curve on~$Y$, and~$P$ be a point on $C$.
Let $\pi\colon\widehat{Y}\to Y$ be the blow up of $Y$ at $P$, and let $\widehat{C}$ be the proper
transform of $C$ on $\widehat{Y}$.
Then
$$
\mathcal{N}_{\widehat{C}/\widehat{Y}}\cong \mathcal{N}_{C/Y}(-P).
$$
\end{lemma}

\begin{proof}
The differential of the morphism $\pi$ provides the
morphisms~\mbox{$d\pi\colon T_{\widehat{C}/\widehat{Y}}\to \pi^*T_{C/Y}$}
and~\mbox{$d\pi\colon \mathcal{N}_{\widehat{C}/\widehat{Y}}\to \pi^*\mathcal{N}_{C/Y}$}.
They can be included into the following commutative diagram:
$$
\xymatrix{
0\ar@{->}[r]& T_{\widehat{C}}\ar@{->}[r]\ar@{->}[d]_{\cong}&
T_{\widehat{C}/\widehat{Y}}\ar@{->}[r]\ar@{->}[d]^{d\pi}& \mathcal{N}_{\widehat{C}/\widehat{Y}}\ar@{->}[r]\ar@{->}[d]^{d\pi} &0\phantom{.}\\
0\ar@{->}[r]& \pi^*T_C\ar@{->}[r]&  \pi^*T_{C/Y}\ar@{->}[r]& \pi^*\mathcal{N}_{C/Y}\ar@{->}[r] &0.
}
$$

Let $E$ be the exceptional divisor of $\pi$.
Note that the fiber of $\mathcal{N}_{\widehat{C}/\widehat{Y}}$ over the point~\mbox{$\widehat{P}=\widehat{C}\cap E$}
is identified with the tangent space to $E$ at
$\widehat{P}$. Since the differential $d\pi$ vanishes
on the latter space, we conclude that the
morphism~\mbox{$d\pi\colon \mathcal{N}_{\widehat{C}/\widehat{Y}}\to \pi^*\mathcal{N}_{C/Y}$}
factors through the morphism
$$
\theta\colon \mathcal{N}_{\widehat{C}/\widehat{Y}}\to \pi^*\mathcal{N}_{C/Y}(-k\widehat{P})
$$
for some
positive integer $k$.
Comparing the degrees of these vector bundles, we conclude that~\mbox{$k=1$}. Now
$\theta$ is a morphism of vector bundles of the same rank and degree, and it is an isomorphism
over a general point of $\widehat{C}$. This implies that $\theta$ is an isomorphism.
\end{proof}

\begin{corollary}
\label{corollary:normal-bundle-1}
Let $L$ be a line in $\mathbb{P}^3$, and let
$P_1, P_2$ be two points on $L$. Let $\pi\colon \widehat{\mathbb{P}}^3\to\mathbb{P}^3$
be the blow up of $\mathbb{P}^3$ at $P_1$ and $P_2$, and denote by $\widehat{L}$ the proper transform
of $L$ on $\widehat{\mathbb{P}}^3$. Then
$$
\mathcal{N}_{\widehat{L}/\widehat{\mathbb{P}}^3}\cong\mathcal{O}_{\mathbb{P}^1}(-1)\oplus\mathcal{O}_{\mathbb{P}^1}(-1).
$$
\end{corollary}

\begin{corollary}\label{corollary:normal-bundle-2}
Let $C$ be a twisted cubic curve in $\mathbb{P}^3$, and let
$P_1, \ldots, P_6$ be six distinct points on~$C$. Let $\pi\colon \widehat{\mathbb{P}}^3\to\mathbb{P}^3$
be the blow up at  the points $P_1,\ldots,P_6$, and denote by $\widehat{C}$ the proper transform
of $C$ on $\widehat{\mathbb{P}}^3$. Then
$$
\mathcal{N}_{\widehat{C}/\widehat{\mathbb{P}}^3}\cong\mathcal{O}_{\mathbb{P}^1}(-1)\oplus\mathcal{O}_{\mathbb{P}^1}(-1).
$$
\end{corollary}
\begin{proof}
Note that
$$
\mathcal{N}_{C/\mathbb{P}^3}\cong\mathcal{O}_{\mathbb{P}^1}(5)\oplus\mathcal{O}_{\mathbb{P}^1}(5),
$$
(see~\cite[Proposition~6]{EisenbudVanDeVen}).
It remains to apply Lemma~\ref{lemma:normal-bundle-blow-up}.
\end{proof}

In \S\ref{section:rigidity}
we will need the following auxiliary result that
is a special equivariant version of a result of Kawakita (see \cite[Theorem~1.1]{Kawakita}).

\begin{lemma}
\label{lemma:mult-2}
Let $Y$ be a 3-fold, and let
$P$ be its smooth point. Let $G$
be a
subgroup of~\mbox{$\mathrm{Aut}(Y)$}
that fixes the point~$P$.
Let $\mathcal{M}_Y$ be a $G$-invariant mobile linear system on $Y$.
Suppose that~$P$ is a center of non-canonical singularities of the
 log pair $(Y,\lambda\mathcal{M}_Y)$ for some rational number $\lambda$.
If the Zariski tangent space $T_{P,Y}$ is an irreducible
representation of the group~$G$,
then
$$
\mathrm{mult}_{P}\big(\mathcal{M}_Y\big)>\frac{2}{\lambda}.
$$
\end{lemma}

\begin{proof}
Suppose that $\mathrm{mult}_{P}(\mathcal{M}_Y)\leqslant\frac{2}{\lambda}$.
Let us show that $T_{P,Y}$ is a reducible representation of the group~$G$.
Let $\sigma\colon \hat{Y}\to Y$ be the blow up at the point $P$, let $E$ be the $\sigma$-exceptional surface,
and let~$\mathcal{M}_{\hat{Y}}$ be the proper transform of the linear system $\mathcal{M}_Y$ via $\sigma$.
Then
$$
K_{\hat{Y}}+\lambda\mathcal{M}_{\hat{Y}}+\Big(\lambda\mathrm{mult}_{P}(\mathcal{M}_Y)-2\Big)E\sim_{\mathbb{Q}}\sigma^*\big(K_Y+\lambda\mathcal{M}_Y\big),
$$
so that  $E$ contains a non-canonical center of the log pair $(\hat{Y},\lambda\mathcal{M}_{\hat{Y}})$.
Then $E$ contains a non-log canonical center of the log pair $(\hat{Y},E+\lambda\mathcal{M}_{\hat{Y}})$.
It then follows from the inversion of adjunction (for instance, see \cite[Theorem~5.50]{KollarMori98})
that $(E,\lambda\mathcal{M}_{\hat{Y}}\vert_{E})$ is not log canonical.

Let $\mu$ be the positive rational number such that $(E,\mu\mathcal{M}_{\hat{Y}}\vert_{E})$ is strictly log canonical.
Then~\mbox{$\mu<\lambda$}, so that $\mu\mathrm{mult}_{P}(\mathcal{M}_Y)<2$.
On the other hand, we have
\begin{equation}
\label{equation:P2-equivalence}
\mu\mathcal{M}_{\hat{Y}}\vert_{E}\sim_{\mathbb{Q}}\mu\mathrm{mult}_{P}\big(\mathcal{M}_Y\big)L,
\end{equation}
where $L$ is a line in $E\cong\mathbb{P}^2$.

Let $Z$ be a log canonical center of the pair $(E,\mu\mathcal{M}_{\hat{Y}}\vert_{E})$.
If $Z$ is a curve, then the multiplicity of the restriction $\mu\mathcal{M}_{\hat{Y}}\vert_{E}$ at this curve is at least $1$ by \cite[Exercise~6.18]{CoKoSm03},
so~that \eqref{equation:P2-equivalence} implies that~$Z$ is a line.
The same arguments imply that other curves in $E$ cannot
be log canonical centers of the
pair~\mbox{$(E,\mu\mathcal{M}_{\hat{Y}}\vert_{E})$}
simultaneously with $Z$ in this case,
so that $Z$ is $G$-invariant.
This implies that $T_{P,Y}$ is a reducible representation of the group $G$.

To complete the proof, we may assume that the locus of log canonical singularities
of the log pair $(E,\mu\mathcal{M}_{\hat{Y}}\vert_{E})$ consists of finitely many points.
On the other hand, it follows from \eqref{equation:P2-equivalence} that the divisor $-(K_E+\mu\mathcal{M}_{\hat{Y}}\vert_{E})$ is ample.
Thus, using Koll\'ar--Shokurov connectedness principle (for example, see \cite[Corollary~5.49]{KollarMori98}),
we see that this locus is connected. Then it consists of one point.
This again means that $T_{P,Y}$ is a reducible representation of the group $G$.
\end{proof}

The following two lemmas will also be used in~\S\ref{section:rigidity}.

\begin{lemma}\label{lemma:4-general-points}
Let $G$ be a finite group acting on $\mathbb{P}^2$ without fixed points.
If a $G$-orbit  on $\mathbb{P}^2$ has at least four points, then it
contains four points such
that no three of them are collinear.
\end{lemma}

\begin{proof}
Let $\Sigma$ be a $G$-orbit with at least four points. The $G$-action has no invariant line, otherwise it would have a fixed point outside the invariant line.
Thus $\Sigma$ contains
at least three non-collinear points, say, $P_1$, $P_2$, and $P_3$.
For each $i=1,2,3$ denote by $L_i$ the line determined by  the two points~${\{P_1, P_2, P_3\}\setminus\{P_i\}}$.
If $\Sigma$ contains a point outside $L_1\cup L_2\cup L_3$, then
we are done. Suppose that this is not a case.
Let~$R_1$ be a point of $\Sigma$ different from $P_1$, $P_2$, and $P_3$.
We may assume that $R_1\in L_1$. Since the line $L_1$ is not $G$-invariant, there is a line $L_1'$ different from $L_1$ and  containing three points of $\Sigma$. We are then able to choose four distinct points of $\Sigma\setminus \left(L_1\cap L_1'\right)$ that lie on $L_1$ or $L_1'$. Such four points satisfy the desired property.
\end{proof}

\begin{lemma}\label{lemma:must-contain-S4}
Let $G$ be a finite group acting on $\mathbb{P}^2$ without fixed points.
Suppose that there is a $G$-invariant smooth
quartic curve $C$ in $\mathbb{P}^2$. Then
\begin{itemize}
\item the curve $C$ is given
by equation~\eqref{eq:S4-quartic-general};
\item the group
$G$ contains a subgroup isomorphic to $\mathfrak{S}_4$;
\item the projective plane $\mathbb{P}^2$ can be identified
with a projectivization of an irreducible three-dimensional
representation of $\mathfrak{S}_4$.
\end{itemize}
\end{lemma}

\begin{proof}
The list of all groups preserving plane quartics
together with the equations
of the corresponding quartics can be found
in~\mbox{\cite[Theorem~6.5.2]{Dolgachev}}, \cite{Henn}, or \cite{KuribayashiKomiya}. Thus, it remains to check
which of them have fixed points on $\mathbb{P}^2$.
\end{proof}

\section{From Aronhold heptads to smooth plane quartics}
\label{section:nets-of-quadrics}

In this section we present some auxiliary assertions concerning nets of quadrics in $\mathbb{P}^3$.
The theory of nets of quadrics is a classical subject. We refer the reader to \cite{Beauville-Prym},
\cite{Edge1}, \cite{Edge2}, \cite{Edge3}, \cite{Edge4}, \cite{Edge5}, \cite{Wall1}, \cite{Wall2}, \cite{DolgachevOrtland},
\cite[\S6.3.2]{Dolgachev}, \cite{Tyurin}, \cite{Reid}, and references therein for many aspects of this theory.

For seven distinct points $P_1,\ldots, P_7$ in $\mathbb{P}^3$ denote by $\mathcal{L}(P_1, \ldots, P_7)$  the linear system of the quadrics passing through the points $P_1,\ldots, P_7$  in $\mathbb{P}^3$.

Before we proceed, let us employ a simple assertion that tells us the dimension of $\mathcal{L}(P_1, \ldots, P_7)$.

\begin{lemma}\label{lemma:CB}
Let $P_1,\ldots, P_r$ be points in a zero-dimensional complete intersection of three quadric surfaces in $\mathbb{P}^3$.
\begin{itemize}
\item If $r\leqslant 7$, then the $r$ points $P_1,\ldots, P_r$ impose independent conditions on quadric surfaces.
\item  If $r=8$, then they do not impose independent conditions on quadric surfaces.
\end{itemize}
\end{lemma}
\begin{proof}
The first assertion immediately follows from \cite[Conjecture~CB11]{EGH}.
The conjecture is partially
proven in \cite[2.2]{EGH}, which allows us to apply it to our case.
The second assertion is obvious since the eight points form a complete intersection of three quadric surfaces.
\end{proof}

The following assertion is a key observation that connects 28-nodal double Veronese cones and nets of quadrics in~$\mathbb{P}^3$.

\begin{lemma}[{cf. \cite[Lemma~IX.5]{DolgachevOrtland}}]
\label{lemma:Aronhold-condition}
For seven distinct points $P_1,\ldots,P_7$ in $\mathbb{P}^3$,
the two conditions
\begin{itemize}
\item[(A)] every element of $\mathcal{L}(P_1, \ldots, P_7)$ is
irreducible;

\item[(B)] the base locus of $\mathcal{L}(P_1, \ldots, P_7)$ consists
of eight distinct points,
\end{itemize}
are satisfied if and only if
the following three conditions hold:
\begin{itemize}
\item[(A${}^\prime$)] no four points of $P_1,\ldots,P_7$ are coplanar
(and in particular no three are collinear);

\item[(B${}^\prime$)] all points $P_1,\ldots,P_7$
are not contained in a single  twisted cubic;

\item[(C${}^\prime$)] for each $i$, the twisted cubic passing through the points of $\{P_1,\ldots,P_7\}\setminus \{P_i\}$ and the line passing through the  point $P_i$ and  one point in $\{P_1,\ldots,P_7\}\setminus \{P_i\}$  meet neither twice nor tangentially.
\end{itemize}
\end{lemma}
\begin{proof}
First, suppose that conditions (A) and (B) hold.

If a plane $\Pi_1$ contains  four points in $\{P_1,\ldots, P_7  \}$, then we have another plane $\Pi_2$ that contains the remaining three points.  The divisor $\Pi_1+\Pi_2$ is a reducible member in $\mathcal{L}(P_1, \ldots, P_7)$, which
is a contradiction to condition~(A).
Therefore, condition~(A${}^\prime$) holds.

If the seven points $P_1,\ldots,P_7$ are contained in a twisted cubic, then every quadric surface passing through $P_1,\ldots,P_7$
must contain this twisted cubic. Thus,
the base locus of $\mathcal{L}(P_1, \ldots, P_7)$ is one-dimensional,
which contradicts condition~(B).
Therefore, condition~(B${}^\prime$) must hold.

Suppose that for some $i, j$, say $i=1$, $j=2$, the twisted cubic curve $C$ passing through the six points $P_2,\ldots, P_7$ and the line $L$ passing through $P_1$ and $P_2$  meet at $P_2$ and at some point~$P$.
If they meet tangentially at $P_2$, then we put $P=P_2$.
Since the linear system of quadric surfaces containing~$C$ is two-dimensional, there are two distinct quadric surfaces $Q_1$ and $Q_2$ that contain the curve $C$ and the point $P_1$. Note that $Q_1$ and $Q_2$  belong to the linear system~\mbox{$\mathcal{L}(P_1, \ldots, P_7)$}.
The intersection of~$Q_1$ and $Q_2$ is a curve of degree $4$ that contains the curve $C$.
Since $L$ passes through three points~${P_1,\, P_2}$, and $P$ of $Q_i$, $i=1,2$, it must be contained  in~$Q_i$, $i=1,2$.
Therefore, the intersection of $Q_1$ and~$Q_2$ consists of $C$ and~$L$.
This implies that for an element~$Q$ in $\mathcal{L}(P_1, \ldots, P_7)$,
the intersection of the quadrics $Q_1$, $Q_2$ and $Q$
either is a curve or consists of~$7$ points (i.e., $P_1+2P_2+P_3+\ldots+P_7$),
which gives a contradiction to condition~(B).
Therefore, condition~(C${}^\prime$) holds.

Now assume that conditions (A${}^\prime$), (B${}^\prime$),
and (C${}^\prime$) hold.

Suppose that there is a reducible member in $\mathcal{L}(P_1, \ldots, P_7)$. Then it consists of two planes containing the seven points $P_1,\ldots, P_7$.
Then one of the planes must contain four of the seven points,
which is a contradiction to condition~(A${}^\prime$).
Therefore, condition~(A) must be satisfied.

Suppose that condition (B) does not hold, i.e.,
the support of the base locus of $\mathcal{L}(P_1, \ldots, P_7)$ either consists of only seven points or contains a curve.
Note that $\dim \mathcal{L}(P_1, \ldots, P_7)\geqslant 2$.

Let~$C$ be the twisted cubic curve passing through the six points $P_2,\ldots, P_7$. Then there are two distinct quadric surfaces $Q_1$ and $Q_2$ that contain the curve $C$ and the point $P_1$.  They are elements
of~\mbox{$\mathcal{L}(P_1, \ldots, P_7)$}.
The intersection of $Q_1$ and $Q_2$ consists of the twisted cubic curve~$C$ and a line~$L$. Since $C$ does not contain the point $P_1$
by condition (B${}^\prime$),  the line $L$ must contain~$P_1$. Note that $L$ intersects $C$ twice or tangentially.

If the base locus of $\mathcal{L}(P_1, \ldots, P_7)$ is zero-dimensional,
then we choose an element $Q_3$ of~\mbox{$\mathcal{L}(P_1, \ldots, P_7)$}
that contains neither $C$ nor $L$.
The intersection of the quadrics $Q_1$, $Q_2$ and~$Q_3$
then consists of the seven points $P_1, \ldots, P_7$ set-theoretically. Otherwise it would consists of eight points including $P_1, \ldots, P_7$, and every element in $\mathcal{L}(P_1, \ldots, P_7)$ would pass through the eight intersection points by Lemma~\ref{lemma:CB}.
Since the three quadric surfaces meet at eight points counted with multiplicities, the line $L$ and $C$ must meet at one of the six points $P_2,\ldots, P_7$;  otherwise~$Q_3$ would contain either $C$ or $L$. This is a contradiction
to condition (C${}^\prime$) since $L$ meets~$C$ twice or tangentially.

If the base locus of $\mathcal{L}(P_1, \ldots, P_7)$ contains a curve, then for an arbitrary  member $Q$
in~\mbox{$\mathcal{L}(P_1, \ldots, P_7)$} the intersection of the quadrics $Q_1$, $Q_2$ and $Q$ contains either $C$ or $L$. Since the dimension of the linear system of quadrics containing both $C$ and $P_1$ is $1$, the base locus
of~\mbox{$\mathcal{L}(P_1, \ldots, P_7)$} cannot contain~$C$.
Choose two general quadric surfaces that contain both the point~$P_7$ and the twisted cubic curve passing through $P_1, \ldots, P_6$. They are elements in~\mbox{$\mathcal{L}(P_1, \ldots, P_7)$}. Therefore, their intersection consists of the twisted cubic curve passing through the points $P_1, \ldots, P_6$, and the line $L$. Therefore, $L$ must meet $C$ at $P_7$. This contradicts condition (C${}^\prime$)
since $L$ and $C$ meet twice or tangentially.
Therefore, condition~(B) holds.
\end{proof}

\begin{definition}\label{definition:Aronhold-heptad}
A set of seven distinct points $P_1,\ldots, P_7$
in $\mathbb{P}^3$ is called an \emph{Aronhold heptad} if every element of $\mathcal{L}(P_1, \ldots, P_7)$ is   irreducible and
the base locus of $\mathcal{L}(P_1, \ldots, P_7)$ consists of eight distinct points.
The set of such eight distinct points is  called a \emph{regular Cayley octad}.  To be precise, the seven points  $P_1,\ldots, P_7$  are called an Aronhold heptad of the regular Cayley octad.
\end{definition}

For an  Aronhold heptad $P_1,\ldots, P_7$ in $\mathbb{P}^3$,  Lemma~\ref{lemma:CB}  immediately implies that there exists a unique  regular Cayley octad that contains $P_1,\ldots, P_7$. It also shows  that
the linear system~\mbox{$\mathcal{L}(P_1, \ldots, P_7)$} is a net.
Conversely,  seven points of a regular Cayley octad
always form an Aronhold heptad.

\begin{lemma}\label{lemma:pencil-base-locus-ij}
Let $P_1,\ldots,P_8$ be a regular Cayley octad in $\mathbb{P}^3$.
For all $1\leqslant i<j\leqslant 8$, denote by~$L_{ij}$
the line passing through the points $P_i$ and $P_j$, and
denote by $T_{ij}$ the twisted cubic passing through the six points of
$\{P_1,\ldots,P_8\}\setminus\{P_i,P_j\}$. Let $\mathcal{L}_{ij}$ be the linear subsystem
of~\mbox{$\mathcal{L}(P_1,\ldots,P_7)$} that consists of the quadrics passing through $T_{ij}$.
Then $\mathcal{L}_{ij}$ is a pencil, and the base locus
of~$\mathcal{L}_{ij}$ is the union $T_{ij}\cup L_{ij}$.
\end{lemma}
\begin{proof}
The first assertion follows from Lemma~\ref{lemma:CB}.
To prove the second assertion, note that the base locus of $\mathcal{L}_{ij}$ is a union of
$T_{ij}$ and some line. On the other hand, the base locus must contain all the points
$P_1,\ldots,P_8$.  Since the twisted cubic $T_{ij}$
passes neither through $P_i$ nor
through $P_j$, the line in the base locus
of~$\mathcal{L}_{ij}$ must pass through both of these
points, i.e., this  line must be~$L_{ij}$.
\end{proof}

Suppose that we are given seven points $P_1,\ldots, P_7$ in $\mathbb{P}^3$ such that $\dim \mathcal{L}(P_1, \ldots, P_7)=2$. Then there are three linearly independent quadric homogeneous polynomials $F_0$, $F_1$, $F_2$ over $\mathbb{P}^3$ that generate the linear system $\mathcal{L}(P_1, \ldots, P_7)$, i.e., an element
in~\mbox{$\mathcal{L}(P_1, \ldots, P_7)$}
is defined by the quadric homogeneous equation
\begin{equation}
\label{equation:matrix}xF_0+yF_1+zF_2=0
\end{equation}
for some $[x:y:z]\in \mathbb{P}^2$. We can express~\eqref{equation:matrix} as a $4\times 4$ symmetric matrix $M(P_1, \ldots, P_7)$ with entries of linear forms in $x, y,z$. Then $\det \left(M(P_1, \ldots, P_7)\right)=0$ defines a plane quartic curve $H(P_1, \ldots, P_7)$ in~$\mathbb{P}^2$.
This plane quartic curve is called the \emph{Hessian quartic}
of the net~\mbox{$\mathcal{L}(P_1, \ldots, P_7)$}.

\begin{remark}\label{remark:determinantal-hypersurface}
The hypersurface  $\mathfrak{D}$ in $\mathbb{P}\left(\mathrm{Sym}(4, \mathbb{C})\right)\cong\mathbb{P}^9$
defined by the determinant polynomial  is singular exactly at non-zero $4\times 4$ symmetric matrices of coranks at least~2.
The three  linearly independent quartic homogeneous polynomials $F_0$, $F_1$, $F_2$  determine a plane $\Pi( F_0, F_1, F_2)$
in~\mbox{$\mathbb{P}\left(\mathrm{Sym}(4, \mathbb{C})\right)$}.
Therefore, the plane quartic curve $H(P_1, \ldots, P_7)$ is singular if and only if  the plane $\Pi( F_0, F_1, F_2)$  either passes through a point corresponding to a non-zero $4\times 4$ symmetric matrix of corank at least~2, or tangentially intersects the hypersurface $\mathfrak{D}$ at a point corresponding to a $4\times 4$ symmetric matrix of corank~$1$.
\end{remark}

\begin{lemma}\label{lemma:smooth-quartic}
Seven distinct points $P_1,\ldots, P_7$ in $\mathbb{P}^3$
form an Aronhold heptad
if and only if the linear system $\mathcal{L}(P_1, \ldots, P_7)$ is a net and
the plane quartic curve $H(P_1, \ldots, P_7)$ is smooth.
Moreover, every smooth plane quartic curve can be obtained in this way.
\end{lemma}

\begin{proof}
Assume that the plane quartic curve $H(P_1, \ldots, P_7)$ is singular.
Since the net~\mbox{$\mathcal{L}(P_1, \ldots, P_7)$}
does not contain any reducible member, it follows from Remark~\ref{remark:determinantal-hypersurface} that
the plane $\Pi(F_0, F_1, F_2)$ tangentially intersects the hypersurface $\mathfrak{D}$ at a point corresponding to a~\mbox{$4\times 4$} symmetric matrix of corank $1$
in~\mbox{$\mathbb{P}\left(\mathrm{Sym}(4, \mathbb{C})\right)$}.
We may assume that $F_0$ corresponds to such a $4\times 4$ symmetric
matrix of corank~$1$.
Denote by $A_0, A_1, A_2$  the $4\times 4$ symmetric matrices   corresponding to $F_0$, $F_1$, $F_2$, respectively.
We may also assume that the matrix $A_0$ is a diagonal matrix with the last diagonal entry $0$. Since the lines $xA_0+yA_1$ and $xA_0+zA_2$ in $\mathbb{P}\left(\mathrm{Sym}(4, \mathbb{C})\right)$ tangentially intersect the surface $\mathfrak{D}$  at $A_0$, both the entry of $A_1$ at the 4th row and the 4th column  and the entry of~$A_2$ at the 4th row and the 4th column are  zero. This means that the point $[0:0:0:1]$ is a base point of~\mbox{$\mathcal{L}(P_1, \ldots, P_7)$}.
However, $F_0$ is singular at $[0:0:0:1]$, and hence the base locus of~\mbox{$\mathcal{L}(P_1, \ldots, P_7)$} cannot consist of eight distinct points. This means that the points $P_1,\ldots, P_7$ do not
form an Aronhold heptad.

Now assume that the curve $H(P_1, \ldots, P_7)$ is smooth.

Suppose that the net $\mathcal{L}(P_1, \ldots, P_7)$ contains a reducible member. Then the plane~\mbox{$\Pi( F_0, F_1, F_2)$}   passes through a point corresponding to a non-zero $4\times 4$ symmetric matrix of corank at least~2, and hence $H(P_1, \ldots, P_7)$
is singular by Remark~\ref{remark:determinantal-hypersurface}, which gives a contradiction.

Suppose that the base locus of $\mathcal{L}(P_1, \ldots, P_7)$ is not zero-dimensional. It then must be one-dimensional. Furthermore, each irreducible curve $B$ of the base locus must be of degree at most~$3$ since $\mathcal{L}(P_1, \ldots, P_7)$ is a net.

If $B$ is a line, then there is a point $P_i$ outside $B$. Let $\Pi$ be the plane determined by $B$ and $P_i$. Each quadric from
$\mathcal{L}(P_1, \ldots, P_7)$ cuts out in $\Pi$ the line $B$ and a line  passing though $P_i$. Such lines form at most a pencil. Therefore, there is a member in $\mathcal{L}(P_1, \ldots, P_7)$ containing $\Pi$, which leads to a contradiction.

If $B$ is a conic, two members in $\mathcal{L}(P_1, \ldots, P_7)$ intersect along $B$ and another conic $B'$. Then either~$B$ or~$B'$ must contain
four points of $P_1, \ldots, P_7$.
This again implies that there is a reducible member
in~\mbox{$\mathcal{L}(P_1, \ldots, P_7)$} and leads to a contradiction.

If $B$ is a cubic curve, two members in $\mathcal{L}(P_1, \ldots, P_7)$
intersect along $B$ and a line. Since a line cannot
contain more than two points of $P_1, \ldots, P_7$,
the curve $B$ must contain at least~$5$ points
of~\mbox{$P_1, \ldots, P_7$}. If $B$ is singular, then $B$ is planar, and hence $\mathcal{L}(P_1, \ldots, P_7)$ contains a reducible member.  Therefore,~$\mathcal{L}(P_1, \ldots, P_7)$ must be the net determined by the twisted cubic~$B$.  In such a case, we can directly show that $H(P_1, \ldots, P_7)$ is singular.

Consequently, the base locus of $\mathcal{L}(P_1, \ldots, P_7)$ is zero-dimensional. Suppose that the base locus of~$\mathcal{L}(P_1, \ldots, P_7)$ does not consist of eight distinct points. Then there is a point $P_i$, say~$P_1$,
with the following property: three general  elements in $\mathcal{L}(P_1, \ldots, P_7)$
meet at $P_1$ so that their local intersection index at~$P_1$
is at least $2$.
This implies that there is an element in~\mbox{$\mathcal{L}(P_1, \ldots, P_7)$} singular at $P_1$,
and hence~\mbox{$H(P_1, \ldots, P_7)$} cannot be smooth.
Indeed, we may assume that the singular quadric is defined
by a~\mbox{$4\times 4$} diagonal matrix $A_0$
with the $4$th diagonal entry equal to zero.
The singular point~$P_1$ is located
at $[0:0:0:1]$. Then the net $\mathcal{L}(P_1, \ldots, P_7)$ can be generated by the net of~\mbox{$4\times 4$} symmetric matrices
\[xA_0+yA_1+zA_2,\]
where $A_1$ and $A_2$ are $4\times 4$ symmetric matrices each of which has $0$ for the entry at the 4th row and the 4th column.
Then
$$
\det \left(xA_0+yA_1+zA_2\right)=0
$$
defines a quartic curve singular at~$[1:0:0]$. The obtained contradiction
shows that the base locus of~$\mathcal{L}(P_1, \ldots, P_7)$ consists of eight distinct points,
so that $P_1,\ldots, P_7$ form an Aronhold heptad.

A given smooth quartic curve can be defined by the determinant of a $4\times 4$ symmetric linear matrix.
This fact together with the first assertion of the lemma implies the second assertion.
\end{proof}

\begin{remark}
The representation of a smooth plane quartic as the determinant
of a symmetric linear matrix that we used in the proof of Lemma~\ref{lemma:smooth-quartic}
goes back to~1855 (see~\cite{Hesse}) and~1902 (see~\cite{Dixon}).
For a contemporary proof, see \cite[Proposition~4.2]{Beauville-determinantal}.
\end{remark}

The following well-known fact concerning intersections of quadrics in $\mathbb{P}^3$ will be necessary for the present paper.  It should be remarked here that they have  been extensively researched in \cite{Reid} and \cite{Tyurin} in much wider settings.

\begin{lemma}
\label{lemma:Q1Q2}
Let $Q_1$  and $Q_2$ be two distinct
quadrics in $\mathbb{P}^3$ and $E$ be the intersection of $Q_1$ and $Q_2$. Then the following are equivalent.
\begin{itemize}
\item the intersection $E$ is smooth and of codimension $2$ in $\mathbb{P}^3$.
\item The line $\mathcal{L}$ determined by $Q_1$ and $Q_2$ in $\mathbb{P}\left(\mathrm{Sym}(4, \mathbb{C})\right)\cong\mathbb{P}^9$ intersects the determinant hypersurface  $\mathfrak{D}$ (see Remark~\ref{remark:determinantal-hypersurface}) at four distinct points.
\item  After a suitable linear coordinate change, $E$ can be defined by the equations
\[
x_0^2+x_1^2+x_2^2+x_3^2=\lambda_0 x_0^2+\lambda_1x_1^2+\lambda_2x_2^2+\lambda_3x_3^2=0
\]
in $\mathbb{P}^3$, where $\lambda_i$'s are four distinct constants.
\end{itemize}
If these equivalent  conditions hold, then
the intersection $E$ is isomorphic to the double cover of the line $\mathcal{L}$ in $\mathbb{P}\left(\mathrm{Sym}(4, \mathbb{C})\right)$ branched exactly at the four intersection points  $\mathcal{L}\cap\mathfrak{D}$.
\end{lemma}
\begin{proof}
The equivalence immediately follows from \cite[Proposition~2.1]{Reid}.
As a matter of fact, the last statement is also instantly implied by the result in \cite[Theorem~4.8]{Reid} that deals with much more general situation.
However, in $\mathbb{P}^3$, the conditions for \cite[Theorem~4.8]{Reid} turn into a tangible state, so that an elementary and short proof could be presented as follows.

Suppose that $E$ is defined  in $\mathbb{P}^3$ by the equations
\[
x_0^2+x_1^2+x_2^2+x_3^2=\lambda_0 x_0^2+\lambda_1x_1^2+\lambda_2x_2^2+\lambda_3x_3^2=0
\]
for some four distinct  constants $\lambda_0, \ldots, \lambda_3$. To complete the proof, it is enough to show that $E$ is isomorphic to a double cover of $\mathbb{P}^1$ branched at four distinct points whose cross-ratio is
\begin{equation}\label{eq:cross-ratio-elliptic}
\frac{(\lambda_1-\lambda_3)(\lambda_2-\lambda_0)}{(\lambda_1-\lambda_0)(\lambda_2-\lambda_3)}
\end{equation}
modulo  permutation.
Note that $\lambda_0\neq\lambda_i$ for $1\leqslant i\leqslant 3$ because $E$ is smooth. Rewrite the equations of~$E$ as
\[x_0^2+x_1^2+x_2^2+x_3^2=(\lambda_0-\lambda_1)x_1^2+(\lambda_0-\lambda_2)x_2^2+(\lambda_0-\lambda_3)x_3^2=0.\]
We first consider the projection of $E$ to the plane $\Pi$ defined by $x_0=0$ centered at the point~\mbox{$[1:0:0:0]$}.  The image is the conic $R$ defined by
$$
(\lambda_0-\lambda_1)x_1^2+(\lambda_0-\lambda_2)x_2^2+(\lambda_0-\lambda_3)x_3^2=0
$$
on~$\Pi$. Furthermore, $E$ is a double cover of $R$. Let $[0:\alpha:\beta:\gamma]$ be a point on this conic. The equation
\[\mu^2+\nu^2(\alpha^2+\beta^2+\gamma^2)=0\]
in $[\mu:\nu]\in\mathbb{P}^1$ determines the points
$[\mu:\nu\alpha:\nu\beta:\nu\gamma]$ of $E$ over the
point~\mbox{$[0:\alpha:\beta:\gamma]$}. Therefore, the double cover is
branched at~\mbox{$[0:\alpha:\beta:\gamma]$} if and only if
$$
\alpha^2+\beta^2+\gamma^2=0.
$$
Consequently, the four branch points on $R$ are
the intersection points of the conics
\[
\aligned
&x_1^2+x_2^2+x_3^2=0,\\
&\lambda_1x_1^2+\lambda_2x_2^2+\lambda_3x_3^2=0
\endaligned
\]
on $\Pi$. These four points are  $\left[0:\sqrt{\lambda_3-\lambda_2}:\pm\sqrt{\lambda_1-\lambda_3}:\pm\sqrt{\lambda_2-\lambda_1}\right]$.

We now consider the projection of $R$ centered at the point
$$
[0:\sqrt{\lambda_0-\lambda_2}: \sqrt{\lambda_1-\lambda_0}:0]
$$
to the line $L$ defined by $x_1=0$  on $\Pi$. This projection defines an isomorphism of $R$ onto $L$. Therefore,~$E$ is the double cover of $L$ branched at the points
\[
\left[0:0:\sqrt{\lambda_3-\lambda_2}\sqrt{\lambda_1-\lambda_0}\pm\sqrt{\lambda_0-\lambda_2}\sqrt{\lambda_1-\lambda_3}: \pm\sqrt{\lambda_0-\lambda_2}\sqrt{\lambda_2-\lambda_1}\right].
\]
It is easy to check that the cross-ratio of the above four points
\[\frac{\pm\sqrt{\lambda_3-\lambda_2}\sqrt{\lambda_1-\lambda_0}\pm\sqrt{\lambda_0-\lambda_2}\sqrt{\lambda_1-\lambda_3}}{\sqrt{\lambda_0-\lambda_2}\sqrt{\lambda_2-\lambda_1}}\]
is equal to the desired number~\eqref{eq:cross-ratio-elliptic}
modulo permutation.
This completes the proof.
\end{proof}

For an Aronhold heptad $P_1,\ldots,P_7$, the
projective plane $\check{\mathbb{P}}^2$ projectively dual to the net of quadrics~\mbox{$\mathcal{L}(P_1,\ldots,P_7)\cong\mathbb{P}^2$}
can be identified with the base of the elliptic fibration $\varkappa$
that is obtained from $\mathbb{P}^3$ by blowing up the points of the corresponding regular Cayley octad.
It  is well-known that the projectively dual curve
of the Hessian quartic curve can be interpreted in terms of singular fibers of~$\varkappa$.
Namely, we have the following

\begin{lemma}\label{lemma:Hessian-dual}
Let $P_1,\ldots,P_7$ be an Aronhold heptad, and let $\check{H}\subset\check{\mathbb{P}}^2$ be the curve
parameterizing singular fibers of~$\varkappa$.
Then $\check{H}$ is projectively dual to $H(P_1,\ldots,P_7)$.
\end{lemma}
\begin{proof}
Choose a fiber $\widetilde{E}$ of the elliptic fibration $\varkappa$.
It is the proper transform
of an intersection curve~$E$ of two quadrics in
$\mathcal{L}(P_1,\ldots,P_7)$. Furthermore,
the fiber $\widetilde{E}$ is smooth if and only if the intersection curve $E$ is smooth  because it
cannot be singular at the points of the regular Cayley octad.
On the other hand, Lemma~\ref{lemma:Q1Q2} implies that the intersection curve $E$ is smooth if and only if
the corresponding pencil of quadrics
contains exactly four singular quadrics.
This verifies  that
the projectively dual curve of~$\check{H}$ coincides with
the Hessian curve of the net $\mathcal{L}(P_1,\ldots,P_7)$.
\end{proof}

\section{From Aronhold heptads to double Veronese cones}
\label{section:construction}

In this section we review the birational construction of $28$-nodal double Veronese cones
due to~\cite{Prokhorov}.

The following result is mainly a part of \cite[Theorem~7.1]{Prokhorov}. We provide its proof for the reader's  convenience and for clarification.

\begin{proposition}
\label{proposition:Prokhorov-construction}
Let $P_1,\ldots,P_7$ be seven points in $\mathbb{P}^3$ that form an Aronhold heptad.
Let~\mbox{$\pi\colon\widehat{\mathbb{P}}^3\to\mathbb{P}^3$} be the blow up of $\mathbb{P}^3$ at the points $P_1,\ldots,P_7$,
and let~\mbox{$\phi\colon \widehat{\mathbb{P}}^3\dasharrow V$}
be the map given by the linear system $|-2K_{\widehat{\mathbb{P}}^3}|$.
Then
\begin{itemize}
\item the map $\phi$ is a birational morphism;
\item   the exceptional locus of $\phi$ is a disjoint union of the proper transforms of the lines
passing through pairs of the points $P_i$ and  the twisted cubics passing through  six-tuples
of the points $P_i$;
\item the variety $V$ is a $28$-nodal double Veronese cone.
\end{itemize}
\end{proposition}
\begin{proof}
Note that by Lemma~\ref{lemma:Aronhold-condition} there are $21$ lines passing through pairs of the points $P_i$,
and $7$ twisted cubic curves passing through  six-tuples
of the points $P_i$. Furthermore, due to Lemma~\ref{lemma:Aronhold-condition} their proper transforms on $\widehat{\mathbb{P}}^3$ are disjoint. It is easy to see that these proper transforms intersect the anticanonical class of~$\widehat{\mathbb{P}}^3$  trivially.

We claim that any other irreducible curve  intersects the anticanonical class neither trivially nor negatively. To see this, let $L$ be an irreducible curve on $\mathbb{P}^3$ whose proper transform on $\widehat{\mathbb{P}}^3$ non-positively  intersects the anticanonical class of $\widehat{\mathbb{P}}^3$. It means
\begin{equation}\label{eq:sum-mult-2deg}
\sum_{i=1}^{7}\mathrm{mult}_{P_i}L\geqslant 2\deg (L).
\end{equation}
Choose a point $P$ on $L$ other than the base point of $\mathcal{L}(P_1,\ldots,P_7)$. Then the quadric surfaces in~\mbox{$\mathcal{L}(P_1,\ldots,P_7)$} passing through the extra point $P$ form a pencil. It follows from~\eqref{eq:sum-mult-2deg} that every quadric surface in this pencil contains the curve~$L$. If $\deg L=4$, then we see from~\eqref{eq:sum-mult-2deg} that~$L$ must be singular
at some of the points $P_i$. This means that the intersection of three quadrics froms~ $\mathcal{L}(P_1,\ldots,P_7)$
is not reduced at that point, i.e., the base locus of $\mathcal{L}(P_1,\ldots,P_7)$ consists of less than $8$ points, which
is a contradiction. If $\deg L=2$, then $L$ is a conic, and by~\eqref{eq:sum-mult-2deg}
it contains four of the points $P_i$. This is impossible by Lemma~\ref{lemma:Aronhold-condition}. Therefore,
one has either $\deg L=3$ or $\deg L=1$.
If $L$ is a singular cubic curve, then it is planar.  However,~\eqref{eq:sum-mult-2deg} implies that
it contains at least four of the points $P_i$. This is impossible by Lemma~\ref{lemma:Aronhold-condition}.
Therefore, $L$ is either a twisted cubic or a line.
In both cases we see from~\eqref{eq:sum-mult-2deg} that $L$ is one of the above $28$ curves.

Let $\tilde{\phi}$ be the morphism defined by the linear system $|-nK_{\widehat{\mathbb{P}}^3}|$ for large enough $n$.
Now Corollaries~\ref{corollary:normal-bundle-1}~and~\ref{corollary:normal-bundle-2} imply that $\tilde{\phi}$ is a flopping contraction, which implies that $\tilde{\phi}$ satisfies the properties of  the first and the second assertions. They also verify that the $28$ proper transforms are contracted to~$28$ nodes on~$V$.
Since $\widehat{\mathbb{P}}^3$ is a smooth weak Fano 3-fold and $\tilde{\phi}$ contracts all the curves that trivially intersect the anticanonical class of $\widehat{\mathbb{P}}^3$, the  3-fold $V$ is a Fano 3-fold with exactly $28$ singular points.

Let $Q$ be a quadric surface passing through the Aronhold heptad, then
\[
-K_{\widehat{\mathbb{P}}^3}\sim 2\left(\pi^*(Q)-(F_1+\ldots +F_7)\right),
\]
where $F_i$'s are the exceptional surfaces of $\pi$. This shows that $-K_V$ is divisible by $2$.
Since
$$
\left(-K_V\right)^3=\left(-K_{\widehat{\mathbb{P}}^3}\right)^3=8,
$$
the variety $V$ is a del Pezzo 3-fold of degree $1$.
Now it is easy to check that the divisor $-2K_V$ is very ample, which implies that the
morphism $\tilde{\phi}$ coincides with the map $\phi$ defined by the linear system~$|-2K_{\widehat{\mathbb{P}}^3}|$.
This completes the proof.
\end{proof}

The following observation was explained to us by Yuri Prokhorov.

\begin{example}
Let $C$ be a twisted cubic in $\mathbb{P}^3$. Choose six distinct points $P_1,\ldots,P_6$ on $C$.
Let~$P$ be a point on $C$ different from $P_1,\ldots,P_6$, and let $L$
be the line passing through $P_6$ and~$P$. Choose a point $P_7\in L\setminus\{P_6, P\}$.
Then such seven points cannot be an Aronhold  heptad since they violate the conditions in Lemma~\ref{lemma:Aronhold-condition}.
Furthermore, if we define the 3-folds $\widehat{\mathbb{P}}^3$ and~$V$
and the map $\phi$ as in Proposition~\ref{proposition:Prokhorov-construction},
then $\phi$ contracts (in particular) the proper transforms of~$C$ and~$L$ on~$\widehat{\mathbb{P}}^3$.
These curves are not disjoint, which implies that the 3-fold $V$ is not nodal.
The same holds if $L$ is chosen to be a tangent line at the point $P_6$ to~$C$.
\end{example}

The following result is implicitly contained (but not clearly stated)
in \cite{Prokhorov}.
It is implied by the results of~\cite{Prokhorov}
together with a simple additional observation.

\begin{theorem}\label{theorem:28-can-be-constructed}
Let $V$ be a double Veronese cone with $28$ singular points. Then $V$ can be constructed from
some Aronhold heptad as in Proposition~\ref{proposition:Prokhorov-construction}.
\end{theorem}
\begin{proof}
It first follows from  \cite[Theorem~1.7]{Prokhorov} and~\cite[Remark~1.8]{Prokhorov}
that the rank of the divisor class group of $V$ is $8$.
Then \cite[Theorem~7.1(i)]{Prokhorov}  implies that $V$ can be obtained in the same way as in Proposition~\ref{proposition:Prokhorov-construction} with seven points of $\mathbb{P}^3$ such that
no four of them are coplanar.
It remains to verify that the seven points form a Aronhold heptad.
However,  \cite[Theorem~7.1(ii)]{Prokhorov} shows that having~$28$ singular points is equivalent to ($\mathrm{A}'$), ($\mathrm{B}'$), and ($\mathrm{C}'$) in Lemma~\ref{lemma:Aronhold-condition}. Therefore, the seven points in the construction must form an Aronhold heptad by Lemma~\ref{lemma:Aronhold-condition}.
\end{proof}

Proposition~\ref{proposition:Prokhorov-construction} and Theorem~\ref{theorem:28-can-be-constructed}
immediately imply the following

\begin{corollary}\label{corollary:28-node}
If a double Veronese cone has $28$ singular points, then they are all nodes.
\end{corollary}

We conclude this section with another result that is a part of \cite[Theorem~7.1]{Prokhorov}. Its proof
is similar to that of Proposition~\ref{proposition:Prokhorov-construction}.

\begin{proposition}
\label{proposition:Prokhorov-construction-2}
Let $P_1,\ldots,P_6$ be six points in $\mathbb{P}^3$ such that no four of them are coplanar.
Let~\mbox{$\pi\colon\widehat{\mathbb{P}}^3\to\mathbb{P}^3$} be the blow up of $\mathbb{P}^3$ at the points $P_1,\ldots,P_6$,
and let~\mbox{$\phi\colon \widehat{\mathbb{P}}^3\dasharrow W$}
be the map given by the linear system $|-K_{\widehat{\mathbb{P}}^3}|$.
Then
\begin{itemize}
\item the map $\phi$ is a birational morphism;
\item   the exceptional locus of $\phi$ is a disjoint union of the proper transforms of the lines
passing through pairs of the points $P_i$ and  the twisted cubic passing through all the points~$P_i$;
\item the variety $W$ is a $16$-nodal double cover of $\mathbb{P}^3$ branched over a quartic surface.
\end{itemize}
\end{proposition}

\section{From smooth plane quartics to
double Veronese cones}
\label{section:equations}

In this section we construct a $28$-nodal double Veronese cone starting from a smooth plane quartic curve.
We also provide an explicit equation~\eqref{eq:V1} for such a 3-fold in terms of covariants
of the plane quartic curve.

Let $\check{\mathbb{P}}^2$ be the dual of the projective plane $\mathbb{P}^2$.
We use a homogeneous coordinate system~\mbox{$[x:y:z]$} for $\mathbb{P}^2$ and $[s:t:u]$ for~$\check{\mathbb{P}}^2$.

We start with a smooth quartic curve $C$ in the projective plane $\mathbb{P}^2$ given by an equation
\begin{equation}\label{equation:quartic}
H(x,y,z)=\sum_{i+j+k=4}a_{ijk}x^iy^jz^k=0.
\end{equation}

We regard $[s:t:u]$ as a general point in $\check{\mathbb{P}}^2$. Then the corresponding line on $\mathbb{P}^2$ is a general line~$L_{s,t,u}$ given by
\[sx+ty+uz=0.\]
The line $L_{s,t,u}$ hits the quartic $C$ at four distinct points $x_1$, $x_2$, $x_3$, $x_4$ lying on $L_{s,t,u}\setminus\{z=0\}$. We may regard these four points as points on the affine line, so that we could define their cross-ratio as follows:
\[\lambda(x_1, x_2, x_3, x_4)=\frac{(x_1-x_3)(x_4-x_2)}{(x_1-x_2)(x_4-x_3)}.\]
This has six different values according to the order of the four points.  However, the following $j$-function is invariant with respect to the reordering $x_1$, $x_2$, $x_3$, $x_4$.

\begin{equation}\label{equation:j-function}%
\begin{split}%
j(x_1, x_2, x_3, x_4)&=256\frac{\left(1-\lambda(x_1, x_2, x_3, x_4)\left(1-\lambda(x_1, x_2, x_3, x_4)\right)\right)^3}%
{\lambda(x_1, x_2, x_3, x_4)^2\left(1-\lambda(x_1, x_2, x_3, x_4)\right)^2}\\%
&=2^8\frac{\left((x_1-x_2)^2(x_4-x_3)^2-(x_1-x_3)(x_4-x_2)(x_4-x_1)(x_3-x_2)\right)^3}%
{(x_1-x_2)^2(x_1-x_3)^2(x_1-x_4)^2(x_2-x_3)^2(x_2-x_4)^2(x_3-x_4)^2}.%
\end{split}
\end{equation}

By plugging $z=-\frac{sx+ty}{u}$ into \eqref{equation:quartic}, we obtain
\begin{equation}\label{equation:quartic2}
\begin{split}
u^4H\left(x,y,-\frac{sx+ty}{u}\right)&=\sum_{i+j+k=4}a_{ijk}(-u)^{4-k}\left(sx+ty\right)^kx^iy^j\\
&=\sum_{r=0}^4b_{4-r}x^ry^{4-r},
\end{split}
\end{equation}
where
\[b_r=\sum_{j=0}^{r}\sum_{i+k=4-j}(-1)^ka_{ijk}\binom{k}{k+j-r}s^{k+j-r}t^{r-j}u^{4-k}.\]
Then we have the following identities for elementary symmetric functions of $x_1$, $x_2$, $x_3$, $x_4$:
\[\aligned
& x_1+x_2+x_3+x_4=-\frac{b_1}{b_0};\\
& x_1x_2+x_1x_3+x_1x_4+x_2x_3+x_2x_4+x_3x_4=\frac{b_2}{b_0};\\
& x_2x_3x_4+x_1x_2x_4+x_1x_2x_4+x_1x_2x_3=-\frac{b_3}{b_0};\\
& x_1x_2x_3x_4=\frac{b_4}{b_0}.
\endaligned\]
Since  the denominator and the numerator of the $j$-function in \eqref{equation:j-function} are symmetric polynomials in~$x_1, x_2, x_3, x_4$,
 the $j$-function in \eqref{equation:j-function}
 may be regarded as a rational function in~$b_0, b_1, b_2, b_3, b_4$.
 Indeed, one has
\[
\begin{split}
j(b_0, b_1, b_2, b_3, b_4)&=1728\frac{4h_2(b_0, b_1, b_2, b_3, b_4)^3}{4h_2(b_0, b_1, b_2, b_3, b_4)^3-27h_{3}(b_0, b_1, b_2, b_3, b_4)^2},\\
\end{split}
\]
where
\[h_2(b_0, b_1, b_2, b_3, b_4)=\frac{1}{3}\Big(-3b_1b_3 +12 b_0b_4+b_2^2\Big);\]
\[h_{3}(b_0, b_1, b_2, b_3, b_4)=\frac{1}{27}\Big(72b_0b_2b_4 - 27b_0b_3^2 - 27b_1^2b_4 + 9b_1b_2b_3 - 2b_2^3\Big)\]
(see Appendix~\ref{section:appendix} for a more detailed
computation).

Regarding $h_2 (b_0, b_1, b_2, b_3, b_4)$ and $h_3 (b_0, b_1, b_2, b_3, b_4)$ as polynomials in $s,t,u$, we see from Appendix~\ref{section:appendix}
that
 \[h_2 (b_0, b_1, b_2, b_3, b_4)=u^4g_4(s,t,u),\]
 \[h_3 (b_0, b_1, b_2, b_3, b_4)=u^6g_6(s,t,u),\]
where $g_4(s,t,u)$ and $g_6(s,t,u)$ are homogeneous polynomials of degrees $4$ and $6$, respectively, in $s, t, u$.
Consequently, the rational function $j(b_0, b_1, b_2, b_3, b_4)$ may be regarded as a rational function $j_C$ in~$s, t, u$, so that
it is a rational function on $\check{\mathbb{P}}^2$. More precisely, one has
\begin{equation}\label{equation:j-function2}%
j_C(s,t,u)=1728\frac{4g_4(s,t,u)^3}{4g_4(s,t,u)^3-27g_{6}(s,t,u)^2}.
\end{equation}

\begin{remark}\label{remark:harmonic}
The equation $g_4(s,t,u)=0$ of degree $4$ describes the points of $\check{\mathbb{P}}^2$
corresponding to lines that intersect~$C$ by equianharmonic quadruples of points (see for instance
\cite[\S2.3.4]{Dolgachev} for terminology). In other words, these lines with the quadruples of points define elliptic curves of $j$-invariant $0$. Elliptic curves of $j$-invariant $0$ are isomorphic to the Fermat plane cubic curve.
Similarly, the equation $g_6(s,t,u)=0$ of degree $6$ describes the points of $\check{\mathbb{P}}^2$
corresponding to lines that intersect~$C$ by harmonic quadruples of points. In this case, lines with the quadruples of points define elliptic curves of $j$-invariant $1728$.  Note that neither the cross-ratio nor the $j$-invariant is defined for the lines corresponding to the $24$ intersection points (counted with multiplicities) of $g_4(s,t,u)=0$ and~$g_6(s,t,u)=0$.
\end{remark}

\begin{lemma}\label{lemma:j-not-constant}
The $j$-function $j_C(s,t,u)$ is not constant. In particular, $g_4(s,t,u)$ cannot be a zero polynomial.
\end{lemma}
\begin{proof}
Among the lines on $\mathbb{P}^2$, there are both tangent lines to $C$ and non-tangent lines to $C$.
This implies that the $j$-function, which is defined by the cross-ratios of the  intersection points, is not constant.
\end{proof}

We can use the polynomials $g_4(s,t,u)$ and $g_6(s,t,u)$ to write down the equation
of the projectively dual curve of the plane quartic~$C$.

\begin{proposition}
\label{proposition:discriminant}
The equation
\begin{equation}\label{equation:discriminant}
4g_4(s,t,u)^3-27g_6(s,t,u)^2=0
\end{equation}
of degree $12$ defines the projectively dual curve $\check{C}$
of the quartic curve $C$ in $\check{\mathbb{P}}^2$.
\end{proposition}
\begin{proof}
The curve $\check{C}$ has degree $12$.
By construction the rational function $j_C(s,t,u)$ is well-defined at least outside the projectively  dual curve $\check{C}$.  On the other hand,
we know from Lemma~\ref{lemma:j-not-constant} that~\mbox{$g_4(s,t,u)$} is not a zero polynomial.
Thus it follows from \eqref{equation:j-function2} that $j_C(s,t,u)$ has poles along the curves defined by some of
the factors of~\mbox{$4g_4(s,t,u)^3-27g_6(s,t,u)^2$}. These poles must be contained in the projectively  dual curve~$\check{C}$. Since $\check{C}$ is an irreducible curve of degree $12$,
it  must be  defined by the homogeneous equation~\eqref{equation:discriminant}
of degree $12$.
\end{proof}

Recall that a
bitangent line to the quartic curve $C$, which is not a tangent line at a hyper-inflection point of $C$, yields an ordinary double point on the  projectively dual curve $\check{C}$. Meanwhile, the tangent line to $C$ at an ordinary inflection point of $C$ produces an ordinary cusp in  $\check{C}$ and the tangent line to~$C$ at a hyper-inflection point of $C$ generates a triple point  of $\check{C}$ that is analytically isomorphic to a singularity defined by equation~\mbox{$s^3=t^4$}.
These are all the possible types of singularities of~$\check{C}$.

Let $C_4$ be the quartic curve defined by $g_4(s,t,u)=0$, and let $C_6$ be the sextic curve defined by~$g_6(s,t,u)=0$ in $\check{\mathbb{P}}^2$.

\begin{lemma}\label{lemma:C4-meets-C6}
Let $P$ be an intersection point of the curves $C_4$ and $C_6$.
Then one of the following possibilities occur.
\begin{itemize}
\item The curves $C_4$ and $C_6$ are smooth at $P$, and meet
transversally at $P$; in this case
$P$ is an ordinary cusp of the curve $\check{C}$.

\item  The curve $C_6$ has a double point at $P$
and  the local intersection index of $C_4$ and $C_6$ at $P$ is~$2$;
in this case $P$
is a triple point  of $\check{C}$ analytically  isomorphic to~\mbox{$s^3=t^4$}.
\end{itemize}
\end{lemma}
\begin{proof}
The point $P$ is a singular point of $\check{C}$. The defining equation \eqref{equation:discriminant} of $\check{C}$ shows that $C_4$ is smooth at $P$ and that
$P$ cannot be an ordinary double point of  $\check{C}$.
Thus $\check{C}$ has  either an ordinary cusp
or a triple point analytically
isomorphic to  $s^3=t^4$ at $P$. Note that $P$ is a triple point of $\check{C}$ if and only if~$C_6$ is singular at $P$. Furthermore, in such a case, $P$ must be a double point of $C_6$.

The defining equation~\eqref{equation:discriminant} tells that
$P$ is an ordinary cusp of $\check{C}$  if $C_4$ and $C_6$ are transversal at $P$.
On the other hand, if $C_6$ is smooth and tangent to $C_4$ at $P$, then $\check{C}$ has a double point worse than a simple cusp at~$P$, in the sense that its log canonical threshold is smaller than~$5/6$.

We now suppose that $P$ is a double point of $C_6$.  Then, \eqref{equation:discriminant} shows that $\check{C}$ has a triple point at~$P$.
Therefore, to complete the proof, it is enough to show that  the singularity of $\check{C}$ at $P$ cannot be  a triple point analytically
isomorphic to  $s^3=t^4$ if the tangent line to $C_4$ at $P$ is one of the Zariski tangent lines of $C_6$ at $P$. This can be easily checked by  blowing up. Indeed, if the tangent line to $C_4$ at $P$ is one of the Zariski tangent lines of $C_6$ at $P$, then the proper transform of $\check{C}$ via the blow up of $\check{\mathbb{P}}^2$ at~$P$ is still singular at some
point over~$P$, which is not the case for the singularity~\mbox{$s^3=t^4$}.
\end{proof}

We now consider the  hypersurface $V$ of degree $6$ in $\mathbb{P}(1,1,1,2,3)$ given by an equation
\begin{equation}\label{eq:V1}
-w^2+v^3-g_4(s,t,u)v+g_6(s,t,u)=0,
\end{equation}
where $\mathrm{wt}(w)=3, \mathrm{wt}(v)=2$.
We denote by $F(s,t,u,v,w)$ the left hand side of~\eqref{eq:V1},
and set
$$
G(s,t,u)=4g_4(s,t,u)^3-27g_6(s,t,u)^2.
$$
Then $G(s,t,u)=0$ defines the projectively  dual curve $\check{C}$ of the smooth quartic curve $C$ in $\check{\mathbb{P}}^2$
by Proposition~\ref{proposition:discriminant}.

Denote by $V_{v}^{sing}$ the set of singular points of $V$ in the section by $v=0$ and by $V_{o}^{sing}$ the set of singular points of $V$ outside  the section by $v=0$. On the other hand, denote by $\check{C}_v^{sing}$ the set of singular points of~$\check{C}$ at which~$C_6$ is singular, and denote by~$\check{C}_o^{sing}$ the set of singular points of~$\check{C}$ outside~$C_6$.

\begin{lemma}\label{lemma:correspondence}
Define the maps as follows:
\[\begin{split}\check{C}_o^{sing} &\longrightarrow V_{o}^{sing},\\
[a_0:a_1:a_2]&\mapsto \left[a_0:a_1:a_2:\frac{3g_6(a_0,a_1,a_2)}{2g_4(a_0,a_1,a_2)} :0\right];\\
&\\
\check{C}_v^{sing} &\longrightarrow V_{v}^{sing},\\
[a_0:a_1:a_2]&\mapsto [a_0:a_1:a_2:0:0].\\
\end{split}\]
Then both are one-to-one correspondences.
\end{lemma}

\begin{proof}
Let $[a_0:a_1:a_2]$ be a point in the set $\check{C}_o^{sing}$.
Observe that
$g_6(a_0,a_1,a_2)\ne 0$,
and hence~\mbox{$g_4(a_0,a_1,a_2)\ne 0$}.
Since $[a_0:a_1:a_2]\in\check{C}_o^{sing}$, we have
$$
\frac{\partial{F}}{\partial v}\left(a_0,a_1,a_2,\frac{3g_6(a_0,a_1,a_2)}{2g_4(a_0,a_1,a_2)},0\right)=\frac{27g_6(a_0,a_1,a_2)^2-4g_4(a_0,a_1,a_2)^3}{4g_4(a_0,a_1,a_2)^2}=0.
$$
Also, since
$$
12g_4(a_0,a_1,a_2)^2\frac{\partial g_4}{\partial s}(a_0,a_1,a_2)-54g_6(a_0,a_1,a_2)\frac{\partial g_6}{\partial s}(a_0,a_1,a_2)=0,
$$
we have
\[\begin{split}
\frac{\partial{F}}{\partial s}\left(a_0,a_1,a_2,\frac{3g_6(a_0,a_1,a_2)}{2g_4(a_0,a_1,a_2)},0\right)&=-\frac{3g_6(a_0,a_1,a_2)}{2g_4(a_0,a_1,a_2)}\frac{\partial g_4}{\partial s}(a_0,a_1,a_2)+\frac{\partial g_6}{\partial s}(a_0,a_1,a_2)\\
&= -\frac{27g_6(a_0,a_1,a_2)^2}{4g_4(a_0,a_1,a_2)^3}\frac{\partial g_6}{\partial s}(a_0,a_1,a_2)+\frac{\partial g_6}{\partial s}(a_0,a_1,a_2)=0.
\end{split}
\]
Similarly, we see
that  $\frac{\partial{F}}{\partial t}$ and $\frac{\partial{F}}{\partial u}$
vanish at the point $\left[a_0:a_1:a_2:\frac{3g_6(a_0,a_1,a_2)}{2g_4(a_0,a_1,a_2)}:0\right]$, so that this point belongs to $V_{o}^{sing}$.

Conversely, suppose that $[a_0:a_1:a_2: v_0 :w_0]$ belongs to $V_{o}^{sing}$. Then $w_0=0$ and~\mbox{$v_0\ne 0$}.
Furthermore, taking the partial derivative of
\eqref{eq:V1} with respect to $v$, we obtain
$$
g_4(a_0,a_1,a_2)=3v_0^2.
$$
Using \eqref{eq:V1} again, we get
$$
g_6(a_0,a_1,a_2)=-v_0^3+v_0g_4(a_0,a_1,a_2)=2v_0^3.
$$
Then
\[\begin{split}\frac{\partial{G}}{\partial s}(a_0,a_1,a_2)&=12g_4(a_0,a_1,a_2)^2\frac{\partial g_4}{\partial s}(a_0,a_1,a_2)-54g_6(a_0,a_1,a_2)\frac{\partial g_6}{\partial s}(a_0,a_1,a_2)\\
&= 108v_0^4\frac{\partial g_4}{\partial s}(a_0,a_1,a_2)-108v_0^3\frac{\partial g_6}{\partial s}(a_0,a_1,a_2)=-108v_0^3 \frac{\partial F}{\partial s} (a_0,a_1,a_2,v_0,w_0)=0.
\end{split}\]
Similarly we obtain
\[\frac{\partial{G}}{\partial t}(a_0,a_1,a_2)=\frac{\partial{G}}{\partial u}(a_0,a_1,a_2)=0.\]
This implies that $[a_0:a_1:a_2]$ belongs to $\check{C}_{o}^{sing}$.

Suppose that $[a_0:a_1:a_2]$ is in $\check{C}_v^{sing}$. Then
\[\frac{\partial{g_6}}{\partial s}(a_0,a_1,a_2)= \frac{\partial{g_6}}{\partial t}(a_0,a_1,a_2)=\frac{\partial{g_6}}{\partial u}(a_0,a_1,a_2)=0.\]
Moreover, we have
$g_6(a_0,a_1,a_2)=0$, and thus also
$g_4(a_0,a_1,a_2)=0$.
These imply
\[\frac{\partial{F}}{\partial w}\left(a_0,a_1,a_2,0,0\right)=\frac{\partial{F}}{\partial v}\left(a_0,a_1,a_2,0,0\right)=0,\]
\[\frac{\partial{F}}{\partial s}\left(a_0,a_1,a_2,0,0\right)=\frac{\partial{F}}{\partial t}\left(a_0,a_1,a_2,0,0\right)=\frac{\partial{F}}{\partial u}\left(a_0,a_1,a_2,0,0\right)=0.\]
Therefore, $[a_0:a_1:a_2:0:0]$ lies in $V_{v}^{sing}$.

Conversely, suppose that $[a_0:a_1:a_2:0:0]$ belongs to $V_v^{sing}$. Then
\[0=\frac{\partial{F}}{\partial v}\left(a_0,a_1,a_2,0,0\right)=-g_4(a_0,a_1,a_2). \]
In particular, this gives $g_6(a_0,a_1,a_2)=0$. We have
\[0=\frac{\partial{F}}{\partial s}\left(a_0,a_1,a_2,0,0\right) =\frac{\partial{g_6}}{\partial s}(a_0,a_1,a_2),\]
\[0=\frac{\partial{F}}{\partial t}\left(a_0,a_1,a_2,0,0\right) =\frac{\partial{g_6}}{\partial t}(a_0,a_1,a_2),\]
\[0=\frac{\partial{F}}{\partial u}\left(a_0,a_1,a_2,0,0\right) =\frac{\partial{g_6}}{\partial u}(a_0,a_1,a_2).\]
Therefore, we conclude that the point $[a_0:a_1:a_2]$ belongs to $\check{C}_v^{sing}$.
\end{proof}

\begin{theorem}\label{theorem:28-node}
The del Pezzo 3-fold $V$ has exactly $28$ singular points.
\end{theorem}
\begin{proof}
Let $\delta_o$ be the number of the ordinary bitangent lines,
i.e., lines tangent to $C$ at two distinct points,  of the smooth
quartic curve $C$, and let
$\delta_s$ be the number of the hyper-inflection points of the smooth quartic curve $C$.  The tangent line to $C$ at a hyper-inflection point is a bitangent
line to $C$ that is tangent to $C$ at a single point with multiplicity $4$. Let $\iota$ be the number of ordinary inflection points of $C$. Then we can derive
\[
\delta_o+\delta_s=28,\quad
\iota+2\delta_s=24
\]
from the classical Pl\"ucker formulae (for instance, see \cite[\S2.4]{GH}).
In particular, the projectively  dual curve~$\check{C}$ has exactly
\[\delta_o+\delta_s+\iota=52-2\delta_s\]
singular points.

The ordinary bitangent lines define $\delta_o$
ordinary double points on $\check{C}$. Such singular points cannot
lie on $C_6$. If so, then they would also lie on $C_4$ by Proposition~\ref{proposition:discriminant}, and would
not be ordinary double points of $\check{C}$ by Lemma~\ref{lemma:C4-meets-C6}.

Meanwhile, by Lemma~\ref{lemma:C4-meets-C6}
the curves $C_4$ and $C_6$ meet \begin{itemize}
\item either transversally
\item or in such a way that $C_6$ has a double point at $P$, and
 the local intersection index of $C_4$ and~$C_6$ at~$P$ is~$2$.
 \end{itemize}
Furthermore,
intersection points of the former type yield  ordinary cusps, and ones of the latter type produce triple points analytically  isomorphic to  $s^3=t^4$.
The number of intersection points of the former type is $\iota$ and  the number of intersection points of the latter  type is $\delta_s$.
We may then conclude that
\[ \#\left| \check{C}_o^{sing}\right |=\delta_o, \quad \#\left| \check{C}_v^{sing}\right |=\delta_s.\]
Since $\delta_o+\delta_s=28$,  Lemma~\ref{lemma:correspondence} immediately implies the statement.
\end{proof}

The following assertion can be deduced from Theorem~\ref{theorem:28-node} and Corollary~\ref{corollary:28-node}.
However, we present a computational proof here to  show how the singularities of $\check{C}$ and the singularities of~$V$ interact with each other.

\begin{proposition}\label{proposition:node}
The singular points of the del Pezzo 3-fold $V$ are all ordinary double points.
\end{proposition}
\begin{proof}
Let $P$ be a point in  $V_{v}^{sing}$. After a suitable coordinate changes, we may assume
that
$$
P=[0:0:1:0:0].
$$
By Lemmas~\ref{lemma:C4-meets-C6} and~\ref{lemma:correspondence}, the curve
$C_4$ is smooth at~$[0:0:1]$, the curve $C_6$ has a double point at~$[0:0:1]$,
 and the local intersection
index of $C_4$ and $C_6$ at~\mbox{$[0:0:1]$} equals to~$2$.
We may also assume that the
tangent line to $C_4$ at~$[0:0:1]$ is defined by $s=0$
and the Zariski tangent cone
of $C_6$ at $[0:0:1]$ is defined by $t(t+\alpha s)=0$, where $\alpha$ is a (possibly zero) constant.
Put
\[
\begin{array}{ll}
f(s,t,v,w)=F(s,t,1,v,w), & g(s,t)=G(s,t,1),\\
\tilde{g}_4(s,t,v,w)=g_4(s,t,1),& \tilde{g}_6(s,t,v,w)=g_6(s,t,1).
\end{array}
\]
We regard $s,t,v,w$ as local coordinates around the point $P$ that corresponds to the origin in~$\mathbb{C}^4$, which
we denote by $p$.
We then consider the Hessian of $f$ at the point $p$:
$$
\mathrm{Hess}(f)(p)=2\left(\frac{\partial \tilde{g}_4}{\partial s}(p)\right)^2\frac{\partial^2 \tilde{g}_6}{\partial t^2}(p)\ne 0.
$$
This implies that the point $P$ is an ordinary double point of $V$.

We now consider a point $R$ in $V_{o}^{sing}$. By appropriate coordinate changes, we may assume that
$$
R=[0:0:1:v_0:0],
$$
where $v_0$ is a non-zero constant. In the affine chart defined by $u\ne 0 $ in~\mbox{$\mathbb{P}(1,1,1,2,3)$}, the point~$R$ corresponds to $q=(0,0,v_0,0)$. In the affine chart defined by $u\ne0 $ in $\check{\mathbb{P}}^2$, we denote by $o$ the origin.
Since $q$ is a singular point of the hypersurface defined by $f=0$, we conclude that~$o$
is a singular point of the curve defined by $g=0$ by Lemma~\ref{lemma:correspondence}. Furthermore, since $q$ is from~$V_{o}^{sing}$, the origin $o$ is an ordinary double point of $g=0$.
From this one can obtain the following relations:
\begin{equation*}
\begin{array}{ll}
v_0=\frac{3\tilde{g}_6(q)}{2\tilde{g}_{4}(q)}, & 27\left(\tilde{g}_6(q)\right)^3=4\left(\tilde{g}_{4}(q)\right)^3,\\
3\tilde{g}_6(q)\frac{\partial \tilde{g}_4}{\partial s}(q)=2\tilde{g}_4(q)\frac{\partial \tilde{g}_6}{\partial s}(q), &
3\tilde{g}_6(q)\frac{\partial \tilde{g}_4}{\partial t}(q)=2\tilde{g}_4(q)\frac{\partial \tilde{g}_6}{\partial t}(q),\\
9\tilde{g}_6(q)\frac{\partial \tilde{g}_6}{\partial s}(q)=2\tilde{g}_4(q)^2\frac{\partial \tilde{g}_4}{\partial s}(q), &
9\tilde{g}_6(q)\frac{\partial \tilde{g}_6}{\partial t}(q)=2\tilde{g}_4(q)^2\frac{\partial \tilde{g}_4}{\partial t}(q).\\
\end{array}
\end{equation*}
Using them,
we are able to derive
\[
48\left(\tilde{g}_{4}(q)\right)^4\mathrm{Hess}(f)(q)=-\tilde{g}_{6}(q)\mathrm{Hess}(g)(o).
\]
(Since the  computation for this is messy and tedious, it is omitted.)
The point $[0:0:1]$ is an ordinary double point of $\check{C}$, so that $\mathrm{Hess}(g)(o)\ne 0$. This completes the proof.
\end{proof}

Below, we list several particular examples of $28$-nodal double Veronese cones
constructed from smooth plane quartics with interesting automorphism groups (cf.~\cite[Table~6]{DolgachevIskovskikh}).
In each case the 3-fold is given by equation~\eqref{eq:V1}.

\begin{example}
\label{example:V1-Klein}
We consider the Klein quartic for $C$. It is given by the equation
\[x^3y+y^3z+z^3x=0.\]
Its automorphism group is isomorphic to $\mathrm{PSL}_2(\mathbb{F}_7)$.
We have
$g_4(s,t,u)=s^3t+t^3u+u^3s$ and
$$
g_6(s,t,u)=\frac{1}{8}\left(3s^5u - 15s^2t^2u^2 + 3st^5 + 3tu^5\right).
$$
\end{example}

\begin{example}
\label{example:Fermat}
Let $C$ be the Fermat plane quartic curve defined by
$$
x^4+y^4+z^4=0.
$$
Its automorphism group is isomorphic to $\mumu_4^2\rtimes\mathfrak{S}_3$.
We have $g_4(s,t,u)=4(s^4+t^4+u^4)$ and~\mbox{$g_6(s,t,u)=16s^2t^2u^2$}.
\end{example}

For explicit equations of smooth plane quartics with an action of
the symmetric group $\mathfrak{S}_4$ and the corresponding
$28$-nodal double Veronese cones we refer the reader to
Example~\ref{example:V1-S4}.

\begin{example}
\label{example:Aut-trivial}
Suppose that the smooth quartic curve $C$ is given by
$$
x^4+y^4+z^4+x^3y+2x^3z=0.
$$
Then $C$ has no non-trivial automorphisms. We have
$g_4(s,t,u)=4(s^4 - st^3 - 2su^3 + t^4 + u^4)$ and
$$
g_6(s,t,u)=-16s^3t^2u - 8s^3tu^2 + 16s^2t^2u^2 - 4t^6 + 4t^5u - t^4u^2 -4t^2u^4 + 4tu^5 - u^6.
$$
\end{example}

\section{One-to-one correspondence}
\label{section:1-1}

In the present  section  we put the previous results together to establish a one-to-one
correspondence between $28$-nodal double Veronese cones and smooth plane quartics, which is asserted in
Theorem~\ref{theorem:one-to-one}. We also study automorphism groups of
$28$-nodal double Veronese cones and prove Theorem~\ref{theorem:Aut}.

In this section, for the projective plane we keep the same  two notations as  in~\S\ref{section:equations}, i.e., $\mathbb{P}^2$ with a homogeneous coordinate system $[x:y:z]$ and   its projectively dual plane $\check{\mathbb{P}}^2$ with~$[s:t:u]$.

\begin{remark}\label{remark:kappa-elliptic}
For an arbitrary del Pezzo 3-fold of degree $1$,
let $\kappa\colon V\dasharrow \check{\mathbb{P}}^2$ be the rational map given by
the half-anticanonical linear system on $V$. The half-anticanonical  linear system has a unique
base point. The 3-fold $V$ is smooth at the base point,
so that a general fiber of $\kappa$
is an elliptic curve by Bertini theorem and adjunction formula.
\end{remark}

In the remaining part of this section we use
$\kappa$ to denote the rational half-anticanonical elliptic fibration on a double
Veronese cone~$V$.

Now we construct a smooth plane quartic from a $28$-nodal double Veronese cone.

\begin{lemma}\label{lemma:V-to-quartic}
Let $V$ be a $28$-nodal double Veronese cone.
Denote by $\check{C}$ the
discriminant curve of the rational elliptic fibration $\kappa$, and let $C$ be the projectively dual curve of~$\check{C}$.
Then $C$ is a smooth plane quartic.
Furthermore, for any construction of $V$ from an Aronhold heptad as in Proposition~\ref{proposition:Prokhorov-construction},
the curve $C$ is isomorphic to the Hessian curve of the corresponding net of quadrics.
\end{lemma}

\begin{proof}
It follows from
Theorem~\ref{theorem:28-can-be-constructed}
that $V$ can be constructed out of  some Aronhold heptad in~$\mathbb{P}^3$
as in Proposition~\ref{proposition:Prokhorov-construction}.
Choose one construction like this, and let $\mathcal{L}$ be the net of quadrics defined by the Aronhold heptad.
The rational elliptic fibration~$\kappa$ is given by the  proper transforms of the members of the net $\mathcal{L}$. Using Lemma~\ref{lemma:Hessian-dual}, we see that the projectively dual curve of~$\check{C}$
is the Hessian curve of the net $\mathcal{L}$, which is a smooth quartic curve by Lemma~\ref{lemma:smooth-quartic}.
\end{proof}

Recall that to any (rational) fibration whose general fiber is an elliptic curve we can associate the $j$-function
of its fibers, which is a rational function on the base of the fibration.

\begin{remark}\label{remark:V-j}
Let $V$ be a del Pezzo 3-fold of degree $1$.
The 3-fold
$V$ can be defined in the weighted projective space
$\mathbb{P}(1,1,1,2,3)$ with weighted homogeneous coordinates $s,t,u,v$, and $w$ by an equation of the form
\begin{equation*}
w^2=v^3-h_4(s,t,u)v+h_6(s,t,u),
\end{equation*}
where $h_4$ and $h_6$ are homogeneous polynomials of degrees $4$ and $6$, respectively.
Since the singularities of $V$ are isolated, the polynomial
$4h_4(s,t,u)^3-27h_{6}(s,t,u)^2$ is not a zero polynomial.
The  rational elliptic fibration $\kappa$
is given by the projection to the coordinates
$s$, $t$, and $u$.
The corresponding $j$-function can be written as
$$
j_V(s,t,u)=1728\frac{4h_4(s,t,u)^3}{4h_4(s,t,u)^3-27h_{6}(s,t,u)^2}
$$
(see for instance \cite[\S3.1.1]{Dolgachev}).
In particular, the discriminant curve of $\kappa$ is given in $\check{\mathbb{P}}^2$ by the equation
$$
4h_4(s,t,u)^3-27h_6(s,t,u)^2=0.
$$
\end{remark}

Recall from \S\ref{section:equations} that any (smooth) quartic curve $C$ in the projective  plane $\mathbb{P}^2$
defines a $j$-function $j_C$
on the dual projective  plane $\check{\mathbb{P}}^2$ that computes the $j$-invariant of the quartuple of intersection points
of $C$ with the corresponding line. It is given by the formula~\eqref{equation:j-function2}.

\begin{lemma}\label{lemma:jV-vs-jC}
Let $V$ be a $28$-nodal double Veronese cone, and let $C$ be the smooth plane quartic curve
corresponding to $V$ via Lemma~\ref{lemma:V-to-quartic}.
Then $j_V(s,t,u)=j_C(s,t,u)$  for all $s,t,u$.
\end{lemma}

\begin{proof}
By Theorem~\ref{theorem:28-can-be-constructed}, the 3-fold $V$ is constructed out of  some net
$\mathcal{L}$ of quadrics in $\mathbb{P}^3$ as in Proposition~\ref{proposition:Prokhorov-construction}.
Let $P$ be a general point of $\check{\mathbb{P}}^2$.  The point $P$ determines a one-dimensional linear
subsystem~$\mathcal{L}_P$ of the net $\mathcal{L}$.
The fiber $E_P$ of the half-anticanonical
elliptic fibration $\kappa$ over $P$
is the proper transform of  the base locus of the pencil $\mathcal{L}_P$. The base locus is given by the intersection of two general quadrics $Q_1$ and $Q_2$ in $\mathcal{L}_P$. The pencil $\mathcal{L}_P$ corresponds to the line~$L$ in $\mathbb{P}^2$ dual to the point $P$.
The singular quadrics in $\mathcal{L}_P$ correspond
to the four intersection points of the line $L$ and the Hessian curve~$C$
of the net $\mathcal{L}$. Lemma~\ref{lemma:Q1Q2} implies that the curve $E_P$ is an elliptic curve isomorphic to the double cover of $L$ branched at the four intersection points of $L$ and $C$.
This means that $j_V(P)=j_C(P)$.
\end{proof}

We will need the following  assertion on
$28$-nodal double Veronese cones and their half-anticanonical elliptic fibrations.

\begin{lemma}
\label{lemma:V-from-elliptic-fibration}
Let $V$ be a $28$-nodal double Veronese cone.
Then $V$ is uniquely
defined by the $j$-function~$j_V$ of the
fibers of the rational elliptic fibration $\kappa$.
\end{lemma}

\begin{proof}
Suppose that there exists another $28$-nodal double Veronese cone $V'$
such that~\mbox{$j_V=j_{V'}$}.
The 3-folds $V$ and $V'$ can be defined in the weighted projective space
$\mathbb{P}(1,1,1,2,3)$ with weighted homogeneous coordinates $s,t,u,v$, and $w$ by equations
\begin{equation}\label{eq:V-first}
w^2=v^3-h_4(s,t,u)v+h_6(s,t,u)
\end{equation}
and
\begin{equation}\label{eq:V-second}
w^2=v^3-k_4(s,t,u)v+k_6(s,t,u),
\end{equation}
respectively, where $h_d$ and $k_d$ are homogeneous polynomials
of degree $d$. By Remark~\ref{remark:V-j} the $j$-functions can be written as
$$
j_V(s,t,u)=1728\frac{4h_4(s,t,u)^3}{4h_4(s,t,u)^3-27h_{6}(s,t,u)^2}
$$
and
$$
j_{V'}(s,t,u)=1728\frac{4k_4(s,t,u)^3}{4k_4(s,t,u)^3-27k_{6}(s,t,u)^2}.
$$
Recall that neither $h_4$ nor $k_4$ is a zero polynomial by Lemmas~\ref{lemma:jV-vs-jC} and~\ref{lemma:j-not-constant}.
Since
$$
j_V(s,t,u)=j_{V'}(s,t,u)
$$
for all $s,t$, $u$, we can easily conclude that
$$
k_4(s,t,u)=ah_4(s,t,u),\quad k_6(s,t,u)=bh_6(s,t,u),
$$
for some non-zero constants $a$ and $b$ with $a^3=b^2$.
Thus equation~\eqref{eq:V-second}
takes the form
$$
w^2=v^3-ah_4(s,t,u)v+bh_6(s,t,u).
$$
Put $c=\frac{b}{a}$. Then $c^2=a$ and $c^3=b$.
The automorphism of $\mathbb{P}(1,1,1,2,3)$ defined by
$$
\left[s:t:u:v:w\right]\mapsto \left[s:t:u:\frac{v}{c}:\frac{w}{c\sqrt{c}}\right]
$$
brings it to the form~\eqref{eq:V-first},
which means that $V$ and $V'$ are isomorphic.
\end{proof}

Now we are able to prove Theorem~\ref{theorem:one-to-one}.

\begin{proof}[First proof of Theorem~\ref{theorem:one-to-one}]
By Lemma~\ref{lemma:V-to-quartic}, a $28$-nodal double Veronese cone  uniquely determines a smooth plane quartic curve  that is the projectively dual curve of the discriminant curve $\check{C}$ of the half-anticanonical rational elliptic fibration~$\kappa$.

Now let $C$ be a smooth plane quartic curve. We then construct a $28$-nodal double Veronese cone~$V$ from $C$ as  in~\S\ref{section:equations} (see equation~\eqref{eq:V1}).
Using Remark~\ref{remark:V-j} and Proposition~\ref{proposition:discriminant}, we see that the projectively dual curve of the discriminant curve of the rational elliptic fibration~$\kappa$ is isomorphic to~$C$.

It remains to show that if $V'$ is a $28$-nodal double Veronese cone such that
the projectively dual curve of the discriminant curve of its half-anticanonical rational elliptic fibration is isomorphic to $C$,
then $V'$ is isomorphic to $V$. By Lemma~\ref{lemma:jV-vs-jC} the $j$-functions of $V$ and $V'$ depend only on~$C$,
so that
$$
j_V(s,t,u)=j_C(s,t,u)=j_{V'}(s,t,u)
$$
for all $s$, $t$, $u$.
According to Lemma~\ref{lemma:V-from-elliptic-fibration} this means that
$V'$ is isomorphic to~$V$.
\end{proof}

We now move our attention to automorphism groups of $28$-nodal double Veronese cones.
Recall that if $V$ is such a 3-fold, we denote by $\tau$ the Galois involution
of the double cover~\mbox{$V\to \mathbb{P}(1,1,1,2)$} given by the linear system $|2H|=|-K_V|$.
The automorphism $\tau$ is contained in the center of the group~\mbox{$\Aut(V)$}.
It is easy to see that it preserves the fibers of the rational
elliptic fibration~$\kappa$.

\begin{proof}[Proof of Theorem~\ref{theorem:Aut}]
Since the map $\kappa\colon V\dasharrow \check{\mathbb{P}}^2$
is $\Aut(V)$-equivariant, we have a group homomorphism
$$
\check{\Xi}\colon \Aut(V)\to \Aut(\check{\mathbb{P}}^2).
$$
Its kernel $\mathcal{K}\subset\Aut(V)$ acts on every fiber of
$\kappa$ preserving the unique base point $O$ of $|H|$. Let~$E$
be a general fiber of $\kappa$.
Then $E$ is a general  elliptic curve by Lemma~\ref{lemma:j-not-constant}, so that
the stabilizer of the point~$O$ in $\Aut(E)$
is isomorphic to $\mumu_2$.
Therefore, $\mathcal{K}$ is a subgroup of $\mumu_2$. On the other hand, $\mathcal{K}$
obviously contains the Galois involution
$\tau$ of the double cover given by the anticanonical linear system of $V$.
Thus we conclude that $\mathcal{K}\cong\mumu_2$.

Consider the action of the group $\check{\Xi}(\Aut(V))$ on the projective plane $\check{\mathbb{P}}^2$
which is the target of the map $\kappa$. By construction the action of $\check{\Xi}(\Aut(V))$
preserves the discriminant curve~$\check{C}$ of the rational elliptic fibration $\kappa$,
and hence its action  on the projectively dual plane of the  projective plane $\check{\mathbb{P}}^2$, which is $\mathbb{P}^2$,  preserves the curve~$C$ projectively dual to $\check{C}$.
Therefore, we obtain a group homomorphism
\begin{equation*}
\Xi\colon\Aut(V)\to\Aut(C)
\end{equation*}
whose kernel $\mathcal{K}$ is isomorphic to $\mumu_2$.
Note that at this moment we  claim neither that $\Xi$ is surjective nor that $\mathcal{K}$ splits as a direct factor in $\Aut(V)$.
However, we know that $\mathcal{K}$ is contained in the center
of the group $\Aut(V)$.

Now let $G$ be the automorphism group of the smooth plane quartic $C$.
Then $G$ is a finite group,  and there is a natural action of
$G$ on the plane $\check{\mathbb{P}}^2$ which is projectively dual
to the plane~$\mathbb{P}^2$ where the curve~$C$ sits.
By Remark~\ref{remark:harmonic} the latter action preserves the curves given
by equations~\mbox{$g_4(s,t,u)=0$} and~\mbox{$g_6(s,t,u)=0$}. This means that there exist
group characters ~\mbox{$\chi_4, \chi_6\colon G\to\mathbb{C}^*$} such that for any~\mbox{$\gamma\in G$} one has
\[
g_4\big(\gamma(s,t,u)\big)=\chi_4(\gamma)\cdot g_4(s,t,u), \quad
g_6\big(\gamma(s,t,u)\big)=\chi_6(\gamma)\cdot g_6(s,t,u).
\]
On the other hand, the elements of $G$ preserve the equivalence class of the cross-ratio
of a quadruple of intersection points of a line in $\mathbb{P}^2$ and the curve $C$; in other words, it preserves
the $j$-function~\mbox{$j_C(s,t,u)$}. By~\eqref{equation:j-function2} this gives
$$
\chi_4(\gamma)^3=\chi_6(\gamma)^2
$$
for every $\gamma\in G$.
Set
$$
\chi_2(\gamma)=\frac{\chi_6(\gamma)}{\chi_4(\gamma)}, \quad \gamma\in G.
$$
Then $\chi_2$ is a character of $G$ such that $\chi_2(\gamma)^2=\chi_4(\gamma)$ and $\chi_2(\gamma)^3=\chi_6(\gamma)$.
Since the image~\mbox{$\chi_2(G)\subset\mathbb{C}^*$} is a cyclic group, there is a well-defined
character
\[
\chi\colon G\to \mathbb{C}^*
\]
such that $\chi(\gamma)^2=\chi_2(\gamma)$ for all $\gamma\in G$.
We have
\[
\chi(\gamma)^4=\chi_4(\gamma), \quad \chi(\gamma)^6=\chi_6(\gamma).
\]

Given an element $\gamma\in G$, we write
$$
\gamma(s,t,u)=[s':t':u']
$$
and define an automorphism of $\mathbb{P}(1,1,1,2,3)$ by
$$
[s:t:u:v:w]\mapsto
[s':t':u':\chi(\gamma)^2v:\chi(\gamma)^3w].
$$
Since the  3-fold $V$ is given by \eqref{eq:V1} in $\mathbb{P}(1,1,1,2,3)$, $G$ preserves the hypersurface
$V$. This provides a group homomorphism
$$
\Xi'\colon \Aut(C)\to\Aut(V).
$$
It is easy to see that $\Xi\circ\Xi'=\mathrm{id}_{\Aut(C)}$.
This implies that $\Aut(V)\cong\mathcal{K}\rtimes\Aut(C)$.
Since $\mathcal{K}\cong\mumu_2$ is contained in the center
of $\Aut(V)$, we conclude that $\Aut(V)\cong\mathcal{K}\times\Aut(C)$.
\end{proof}

\begin{corollary}\label{corollary:Aut-V-vs-S}
Let~$V$ be a $28$-nodal double Veronese cone,
let $C$ be the smooth plane quartic curve corresponding to $V$
by Theorem~\ref{theorem:one-to-one},
and let~$S$ be the del Pezzo surface of degree~$2$ constructed as the double
cover of $\mathbb{P}^2$ branched along $C$. Then
$$
\mathrm{Aut}(V)\cong\mumu_2\times\mathrm{Aut}(C)\cong \mathrm{Aut}(S).
$$
\end{corollary}

\section{Theta characteristics}
\label{section:theta}

As we have seen in \S\ref{section:construction}, Aronhold heptads, which are parts of regular Cayley octads, play a role in constructing $28$-nodal double Veronese cones.
On the other hand, neither an Aronhold heptad nor a regular Cayley octad is uniquely associated to a
$28$-nodal double Veronese cone. In what follows, correspondences between regular Cayley octads, smooth plane quartic curves, and $28$-nodal double Veronese cones
will be investigated in more detail, and the necessary additional data that is required to
have natural one-to-one correspondences will be clarified.

So far we used Proposition~\ref{proposition:Prokhorov-construction} to construct $28$-nodal
double Veronese cones starting from Aronhold heptads. However, it turns out that the construction
essentially depends only on the choice of the regular Cayley octad.
This can be easily understood through Geiser type  involutions.

\begin{lemma}\label{lemma:V-i-j}
Let $P_1,\ldots,P_8$ be a regular Cayley octad in $\mathbb{P}^3$.
Then for all $i=1,\ldots, 8$ the $28$-nodal double Veronese cones
defined by the Aronhold heptads $\{P_1,\ldots,P_8\}\setminus\{P_i\}$ are isomorphic to each other.
\end{lemma}
\begin{proof}
For each $i=1,\ldots, 8$, let $\widehat{\mathbb{P}}^3_i$ be the blow up of $\mathbb{P}^3$ at the points $P_1,\ldots, P_8$ except the point~$P_i$. Let $\widehat{\mathbb{P}}^3_{ij}$ be the weak Fano 3-fold obtained by blowing up $\mathbb{P}^3$ at the six points of~\mbox{$\{P_1,\ldots, P_8\}\setminus\{P_i,P_j\}$} and let $W_{ij}$ be its anticanonical model.

For $1\leqslant k\ne l\leqslant 8$, denote by $L_{kl}$
the line passing through the points $P_k$ and $P_l$, and
denote by~$T_{kl}$ the twisted cubic passing through the six points of
$\{P_1,\ldots,P_8\}\setminus\{P_k,P_l\}$. Let $\widehat{L}_{kl}$ and~$\widehat{T}_{kl}$ be the proper transforms of these curves
on $\widehat{\mathbb{P}}^3_{ij}$.

Proposition~\ref{proposition:Prokhorov-construction-2} implies that
that $W_{ij}$ is the double cover of
$\mathbb{P}^3$ branched over a quartic surface and that the map $\phi_{ij}: \widehat{\mathbb{P}}^3_{ij}\to W_{ij} $ given by the anticanonical linear system is a birational morphism that contracts exactly the $15$ curves $\widehat{L}_{kl}$ for $k<l$ with $\{k,l\}\cap\{i,j\}=\varnothing$, and the curve
$\widehat{T}_{ij}$. Denote by~$\bar{L}_{kl}$ and $\bar{T}_{kl}$ the images on $W_{ij}$
of the curves $\widehat{L}_{kl}$ and $\widehat{T}_{kl}$ not contracted by~$\phi_{ij}$, respectively.

Note that the points $\bar{P}_i$ and $\bar{P}_j$ on $W_{ij}$  corresponding to $P_i$ and $P_j$ are conjugate to each other
with respect to the Galois involution $\varepsilon$ of the double cover.

The double cover $\zeta\colon W_{ij}\to\mathbb{P}^3$ is given by the half-anticanonical linear system,
so that the pull-backs of planes under $\zeta$
should be the proper transforms of the quadrics in the original $\mathbb{P}^3$ passing through
the six points of $\{P_1,\ldots,P_8\}\setminus\{P_i,P_j\}$.
The images of the curves $\zeta(\bar{T}_{ki})$ and~$\zeta(\bar{T}_{kj})$, where $k\not\in\{i,j\}$, are
lines in $\mathbb{P}^3$ (note however that the images $\zeta(\bar{T}_{kl})$ for~\mbox{$\{k,l\}\cap\{i,j\}=\varnothing$}
are conics).
Therefore, by Lemma~\ref{lemma:pencil-base-locus-ij} for $k\not\in\{i,j\}$ one has
$$
\aligned
&\varepsilon(\bar{T}_{ki})=\bar{L}_{ki}, \quad \varepsilon(\bar{L}_{ki})=\bar{T}_{ki},\\
&\varepsilon(\bar{T}_{kj})=\bar{L}_{kj}, \quad \varepsilon(\bar{L}_{kj})=\bar{T}_{kj}.
\endaligned
$$

The involution $\varepsilon\colon W_{ij}\to W_{ij}$ gives rise to a birational
map
$$
\hat{\varepsilon}\colon \widehat{\mathbb{P}}^3_i\dasharrow\widehat{\mathbb{P}}^3_j.
$$
Since $\varepsilon(\bar{P}_i)=\bar{P}_j$, we conclude that $\hat{\varepsilon}$ is actually
an isomorphism.

Denote by $\widehat{L}_{kl}^i$ and $\widehat{T}_{kl}^i$ (respectively, $\widehat{L}_{kl}^j$ and $\widehat{T}_{kl}^j$)
the proper transforms of the curves~$L_{kl}$ and~$T_{kl}$
on~$\widehat{\mathbb{P}}^3_{i}$ (respectively, $\widehat{\mathbb{P}}^3_{j}$). Then for $k\not\in\{i,j\}$ one has
$$
\aligned
&\hat{\varepsilon}(\widehat{T}_{ki}^i)=\widehat{L}_{ki}^j, \quad \hat{\varepsilon}(\widehat{L}_{ki}^i)=\widehat{T}_{ki}^j,\\
&\hat{\varepsilon}(\widehat{T}_{kj}^i)=\widehat{L}_{kj}^j, \quad \hat{\varepsilon}(\widehat{L}_{kj}^i)=\widehat{T}_{kj}^j.
\endaligned
$$
Also, we observe that for $\{k,l\}\cap\{i,j\}=\varnothing$
$$
\hat{\varepsilon}(\widehat{L}_{kl}^i)=\widehat{L}_{kl}^j, \quad \hat{\varepsilon}(\widehat{T}_{ij}^i)=\widehat{T}_{ij}^j.
$$

Therefore, the pluri-anticanonical
maps $\phi_i$ and $\phi_j$ of $\widehat{\mathbb{P}}^3_{i}$ and $\widehat{\mathbb{P}}^3_{i}$ both factor through $W_{ij}$,
and the pluri-anticanonical
models of $\widehat{\mathbb{P}}^3_{i}$ and $\widehat{\mathbb{P}}^3_{i}$,
that are obtained by the construction described in Proposition~\ref{proposition:Prokhorov-construction},
are both isomorphic to one and the same double Veronese cone $V$.
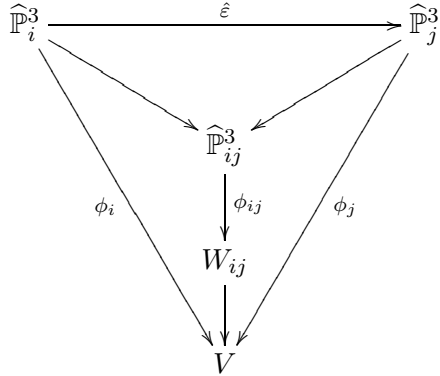
$$
\xymatrix{
\widehat{\mathbb{P}}^3_i\ar@{->}[rrd]\ar@{->}[rrddd]_{\phi_i}\ar@{->}[rrrr]^{\hat{\varepsilon}} &&&& \widehat{\mathbb{P}}^3_j\ar@{->}[lld]\ar@{->}[llddd]^{\phi_j}  \\
&& \widehat{\mathbb{P}}^3_{ij}\ar@{->}[d]^{\phi_{ij}} && \\
&&W_{ij}\ar@{->}[d]&&\\
&&V&& \\
}
$$
This completes the proof of the statement.
\end{proof}

\begin{remark}
A two-dimensional analog of Lemma~\ref{lemma:V-i-j} is the following simple observation. Let~\mbox{$R_1,\ldots,R_9$}
be points in $\mathbb{P}^2$ in general position such that they are the intersection of two cubic curves.
Then the del Pezzo surfaces of degree $1$ obtained by blowing up the eight points of~$\{R_1,\ldots,R_9\}\setminus\{R_i\}$ are isomorphic to each other for all~\mbox{$1\leqslant i\leqslant 9$}.
\end{remark}

Let $\mathcal{N}$ be the set that consists of
isomorphism classes of regular Cayley octads modulo
projective transformations.
Let $\mathcal{T}$ be the set
that consists of
the pairs $(C,\theta)$, where $C$ is a smooth plane quartic
considered up to isomorphism, and $\theta$ is an even theta characteristic
on $C$.
A given regular Cayley octad defines a  net  $\mathcal{L}$ of quadrics in~$\mathbb{P}^3$. The net $\mathcal{L}$ yields its Hessian quartic curve~\mbox{$H(\mathcal{L})$} in~$\mathbb{P}^2$,
which is smooth by Lemma~\ref{lemma:smooth-quartic}.
Meanwhile, the singular points of quadrics in the net~$\mathcal{L}$ sweep out a smooth curve of degree $6$ in~$\mathbb{P}^3$ (see \cite[Lemme~6.8]{Beauville-Prym} and \cite[Proposition~6.10]{Beauville-Prym}), which is called the \emph{Steinerian curve} of the net and is denoted by $S(\mathcal{L})$. There is an even theta characteristic $\theta(\mathcal{L})$ such that the linear system~\mbox{$|K_{H(\mathcal{L})}+\theta(\mathcal{L})|$}
defines an isomorphism of $H(\mathcal{L})$ with $S(\mathcal{L})$ ( see \cite[Proposition~6.10]{Beauville-Prym}
and \cite[Lemme~6.12]{Beauville-Prym}). With such Hessian quartic curves and their even theta characteristics, we define the  map
\[
\Theta\colon\mathcal{N}\to \mathcal {T}
\]
by assigning $\Theta(\mathcal{L})=(H(\mathcal{L}), \theta(\mathcal{L}))$.

\begin{theorem}\label{theorem:bijection-Theta}
The map $\Theta$ is bijective.
\end{theorem}
\begin{proof}
 See \cite[Proposition~6.23]{Beauville-Prym}.
\end{proof}

Suppose that the points $P_1,\ldots,P_7$ form an Aronhold heptad, and denote by $P_8$ the eighth base point of the net of quadrics $\mathcal{L}$ defined by the Aronhold heptad. Choose a pair of points  $P_i$, $P_j$ in the regular Cayley octad $P_1,\ldots,P_8$. This pair  yields a pencil in the net $\mathcal{L}$ that consists of the quadrics containing the line passing through $P_i$ and $P_j$. The  pencil contains exactly one or  two singular quadrics
(see \cite[Lemme~6.6(i)]{Beauville-Prym}).  Furthermore, the pencil defines
a line in the projective plane where the Hessian quartic
curve $H(\mathcal{L})$ sits, and this
line is bitangent to the Hessian quartic curve (see \cite[Theorem~6.3.5]{Dolgachev}).
Such a bitangent line defines an odd theta characteristic of $H(\mathcal{L})$.
It will be denoted by $\theta_{ij}(\mathcal{L})$. {It is obvious that~\mbox{$\theta_{ij}(\mathcal{L})=\theta_{kl}(\mathcal{L})$} if and only if
$\{i,j\}=\{k,l\}$. In particular, every odd theta characteristic on~$H(\mathcal{L})$ can be represented by~\mbox{$\theta_{ij}(\mathcal{L})$} for some $i$ and $j$.

For each choice of  four distinct indices $i,j,k,l$ in $\{1,\ldots,8\}$, set
\begin{equation}\label{equation:even}
\theta_{i,jkl}=\theta_{ij}(\mathcal{L})+\theta_{ik}(\mathcal{L})+\theta_{il}(\mathcal{L})-K_{H(\mathcal{L})}.
\end{equation}
Before we proceed, notice that
\begin{equation}\label{eq:3-relation}
\theta_{ij}(\mathcal{L})+\theta_{ik}(\mathcal{L})+\theta_{jk}(\mathcal{L})=K_{H(\mathcal{L})}+\theta(\mathcal{L})
\end{equation}
for each choice of three distinct indices $i,j,k$ in $\{1,\ldots,8\}$ since the left hand side is the section of the Steinerian curve $S(\mathcal{L})$ by the plane containing the points $P_i$, $P_j$, and $P_k$ in $\mathbb{P}^3$.

\begin{lemma}
\label{lemma:even-from-7-odd}
Let $r$ be  a fixed index in $\{1,\ldots,8\}$.
\begin{itemize}
\item The theta characteristic $\theta_{r, ijk}$ is even  for three distinct indices $i,j,k$ in $\{1,\ldots,8\}\setminus\{r\}$.
\item Every even theta characteristic on $H(\mathcal{L})$ except $\theta(\mathcal{L})$ can be represented as $\theta_{r, ijk}$ for
some choice of  three distinct indices $i$, $j$, $k$ in $\{1,\ldots,8\}\setminus\{r\}$.
\item One has
\begin{equation}\label{eq:even-from-7-odd}
\theta(\mathcal{L})=-3K_{H(\mathcal{L})}+\sum\limits_{i\neq r}\theta_{ri}(\mathcal{L}).
\end{equation}
\end{itemize}
\end{lemma}

\begin{proof}
For the first statement, see \cite[Theorem~6.3.6]{Dolgachev} or~\cite[Proposition~IX.4]{DolgachevOrtland}.

Since there are exactly $35$ even theta characteristics excluding $\theta(\mathcal{L})$,   in order to prove the second statement, it is enough to verify that $\theta_{r, i_1j_1k_1}\ne \theta_{r, i_2j_2k_2}$ if $\{i_1,j_1,k_1\}\ne \{ i_2,j_2,k_2\}$.
Suppose that~$\theta_{r, i_1j_1k_1}= \theta_{r, i_2j_2k_2}$. For convenience, assume that $r=8$.  Then
\[\theta_{8i_1}(\mathcal{L})+\theta_{8j_1}(\mathcal{L})+\theta_{8k_1}(\mathcal{L})+\theta_{8i_2}(\mathcal{L})+\theta_{8j_2}(\mathcal{L})+\theta_{8k_2}(\mathcal{L})-3K_{H(\mathcal{L})}
=\theta_{r, i_1j_1k_1}+\theta_{r, i_2j_2k_2}-K_{H(\mathcal{L})}=0.\]
If $\{i_1,j_1,k_1\}\cap \{ i_2,j_2,k_2\}=\varnothing$, then the above equality and~\eqref{eq:3-relation} imply
\[\theta_{i_1i_2}(\mathcal{L})+\theta_{j_1j_2}(\mathcal{L})+\theta_{k_1k_2}(\mathcal{L})=3\theta(\mathcal{L})=K_{H(\mathcal{L})}+\theta(\mathcal{L}).\]
This means that the six  points $P_{i_1}, P_{i_2}, P_{j_1}, P_{j_2}, P_{k_1}, P_{k_2}$  lie on a single plane, which is impossible
by Lemma~\ref{lemma:Aronhold-condition}.
If $\{i_1,j_1,k_1\}\cap \{ i_2,j_2,k_2\}\ne\varnothing$, then  we may assume that $i_1=i_2$. Then
\[\theta_{8j_1}(\mathcal{L})+\theta_{8k_1}(\mathcal{L})=\theta_{8j_2}(\mathcal{L})+\theta_{8k_2}(\mathcal{L}).\]
This together with~\eqref{eq:3-relation}  implies
\[\theta_{j_1k_1}(\mathcal{L})=\theta_{j_2k_2}(\mathcal{L}),\]
and hence $\{i_1,j_1,k_1\}=\{ i_2,j_2,k_2\}$.

Now let us prove the third statement. Suppose that the right hand side of~\eqref{eq:even-from-7-odd} is an odd theta characteristic~$\theta_{mn}(\mathcal{L})$.  Then $m, n\ne r$; otherwise  $\theta_{r, ijk}$ would yield the same even theta characteristic  for two different choices of  three distinct indices $i$, $j$, $k$,
which is not the case by the second statement.  For  convenience, we may assume that $r=8$, $m=2$, and $n=1$.  Then our assumption reads
$$
\theta_{21}(\mathcal{L})=-3K_{H(\mathcal{L})}+\sum\limits_{i=1}^7\theta_{8i}(\mathcal{L}).
$$
By~\eqref{eq:3-relation} this yields
\begin{equation}\label{eq:128}
-K_{H(\mathcal{L})}+\sum\limits_{i=3}^7\theta_{8i}(\mathcal{L})=\theta_{21}(\mathcal{L})+\theta_{81}(\mathcal{L})+\theta_{82}(\mathcal{L})=K_{H(\mathcal{L})}+\theta(\mathcal{L}).
\end{equation}
Note that
$$
K_{H(\mathcal{L})}+\theta(\mathcal{L})=\theta_{43}(\mathcal{L})+\theta_{83}(\mathcal{L})+\theta_{84}(\mathcal{L})
$$
by~\eqref{eq:3-relation}. Using this, we
deduce from~\eqref{eq:128} that
\[-K_{H(\mathcal{L})}+\sum\limits_{i=5}^7\theta_{8i}(\mathcal{L})=\theta_{43}(\mathcal{L}).\]
This contradicts the first statement. Therefore, the right hand side of~\eqref{eq:even-from-7-odd} is an even theta characteristic.

We now suppose that the right hand side of~\eqref{eq:even-from-7-odd} is not $\theta(\mathcal{L})$. Then, due to the second statement, we may assume that  $r=8$ and
$$
-3K_{H(\mathcal{L})}+\sum\limits_{i=1}^7\theta_{8i}(\mathcal{L})= \theta_{81}(\mathcal{L})+\theta_{82}(\mathcal{L})+\theta_{83}(\mathcal{L})-K_{H(\mathcal{L})}.
$$
However, this implies an absurd identity $\theta_{8,456}=\theta_{87}(\mathcal{L})$.  Consequently, the right hand side of~\eqref{eq:even-from-7-odd} must be $\theta(\mathcal{L})$.
\end{proof}

Let $\mathcal{L}_{\theta_{i,jkl}}$ be the net of quadrics
corresponding to the plane quartic $H(\mathcal{L})$ and
the even theta characteristic $\theta_{i,jkl}$
via the bijection from Theorem~\ref{theorem:bijection-Theta}, and let~$S_{\theta_{i,jkl}}$ be its Steinerian curve.
It turns out that these curves (and also regular Cayley octads and nets of quadrics)
are related by standard Cremona transformations of~$\mathbb{P}^3$.

\begin{proposition}\label{proposition:Cremona-Cayley-octads}
Up to projective transformation, the curve $S_{\theta_{i,jkl}}$
is obtained from $S(\mathcal{L})$ by the standard Cremona transformation centered at the points $P_i$, $P_j$, $P_k$, $P_l$,
and also by the standard Cremona transformation centered at the complementary set  of points $P_m$, $P_n$, $P_s$, $P_t$.
\end{proposition}
\begin{proof}
See \cite[Proposition~IX.4]{DolgachevOrtland}.
\end{proof}

\begin{remark}\label{remark:other-35}
By Lemma~\ref{lemma:even-from-7-odd},
the $35$ even theta characteristics on $H(\mathcal{L})$ other than $\theta(\mathcal{L})$ are all of the form~\eqref{equation:even} for a fixed index $i$ (one can take for instance $i=8$). This means that
the even theta characteristics obtained as in Proposition~\ref{proposition:Cremona-Cayley-octads}
by the $35$ Cremona transformations
associated to choices of four
points out of the seven points $P_1,\ldots,P_7$ (alternatively, by the~$35$ Cremona transformations
associated to the point $P_8$ and choices of three points out of~$P_1,\ldots,P_7$) are pairwise different.
\end{remark}

Using Proposition~\ref{proposition:Cremona-Cayley-octads}, we deduce the following.

\begin{corollary}\label{corollary:Cremona-Cayley-octads}
Up to projective transformation, the net of quadrics
$\mathcal{L}_{\theta_{i,jkl}}$ consists of the proper transforms of the quadrics from the net
$\mathcal{L}$ with respect to the standard Cremona transformation~$\varsigma$ centered at the points $P_i$, $P_j$, $P_k$, $P_l$
(or with respect to the standard Cremona transformation centered at the complementary set  of points $P_m$, $P_n$, $P_s$, $P_t$).  Furthermore, the corresponding regular Cayley octad $P_1',\ldots,P_8'$ can be constructed as follows:
the points $P_i', P_j', P_k', P_l'$ are the images of the divisors contracted by $\varsigma$, and $P_r'=\varsigma(P_r)$ for $r=m, n, s, t$.
\end{corollary}

\begin{proof}
After a suitable projective transformation
we may assume that
$$
P_i=[1:0:0:0],\quad
P_j=[0:1:0:0],\quad
P_k=[0:0:1:0],\quad
P_l=[0:0:0:1].
$$
Then the standard Cremona transformation $\varsigma$  is the selfmap of $\mathbb{P}^3$ defined by $$
\varsigma(x_0,x_1,x_2,x_3)=\left[\frac{1}{x_0}:\frac{1}{x_1}:\frac{1}{x_2}:\frac{1}{x_3}\right].
$$
Every quadric surface in $\mathcal{L}$ is defined by a quadric  homogeneous polynomial of the form
\[\sum_{\alpha\ne \beta}a_{\alpha\beta}x_\alpha x_\beta=0,\]
where $a_{\alpha\beta}=a_{\beta\alpha}.$
Its proper transform by $\varsigma$ is defined by the quadric  homogeneous polynomial
\[\sum_{\{\alpha, \beta, \gamma, \delta\}=\{0,1,2,3\}}a_{\alpha\beta}x_\gamma x_\delta=0.\]
The four coordinate points and the four points $\varsigma(P_m)$, $\varsigma(P_n)$, $\varsigma(P_s)$, $\varsigma(P_t)$  are the eight intersection points of three quadrics. Lemma~\ref{lemma:CB} then implies that the net $\mathcal{L}'$ of quadric surfaces passing though these eight points is exactly the net of the proper transforms of the quadric surfaces in $\mathcal{L}$.
Observe that the symmetric matrix corresponding to the former quadric surface and the one corresponding to the latter have the same determinant. It then follows from  Lemma~\ref{lemma:smooth-quartic}  that the four coordinate points and the four points $\varsigma(P_m)$, $\varsigma(P_n)$, $\varsigma(P_s)$, $\varsigma(P_t)$  form a regular Cayley octad.
It also immediately follows from our observation that the Steinerian curve
$S(\mathcal{L}')$ is the proper transform of the Steinerian curve
$S(\mathcal{L})$ with respect to $\varsigma$. Hence by
Proposition~\ref{proposition:Cremona-Cayley-octads} we have
$$
S(\mathcal{L}')=S_{\theta_{i,jkl}}
$$
up to projective transformation. By Theorem~\ref{theorem:bijection-Theta}
this implies that
$$
\mathcal{L}'=\mathcal{L}_{\theta_{i,jkl}},
$$
and the assertion follows.
\end{proof}

For a given smooth plane quartic curve $C$ equipped with an even theta characteristic $\theta$, its associated net $\mathcal{L}$ of quadrics determines a regular Cayley octad $P_1,\ldots, P_8$ in $\mathbb{P}^3$,
which in turn allows us to construct a unique $28$-nodal double Veronese cone $V_{C,\theta}$ (see Lemma~\ref{lemma:V-i-j}).
Let $\theta'$
be  another even theta characteristic on $C$, and let $V_{C,\theta'}$ be the corresponding
$28$-nodal double Veronese cone.

\begin{lemma}\label{lemma:V-theta-thetaprime}
The 3-folds $V_{C,\theta}$ and $V_{C,\theta'}$
are isomorphic.
\end{lemma}
\begin{proof}
By Remark~\ref{remark:other-35}, there are indices $1\leqslant i<j<k\leqslant 7$ such that
\[
\theta'=\theta_{8i}(\mathcal{L})+\theta_{8j}(\mathcal{L})+\theta_{8k}(\mathcal{L})-K_{C}.
\]
For convenience, let us say $i=5, j=6, k=7$.
Corollary~\ref{corollary:Cremona-Cayley-octads} tells us that the net of quadrics~$\mathcal{L}'$ associated
to~\mbox{$(C, \theta')$} via Theorem~\ref{theorem:bijection-Theta} can be obtained
by applying the standard Cremona transformation~$\varsigma$
centered at the points $P_1$, $P_2$, $P_3$, $P_4$ to the original net $\mathcal{L}$, up to projective equivalence.
Furthermore, the net~$\mathcal{L}'$ is defined by the regular Cayley octad $P_1',\ldots,P_8'$, where
$P_1',\ldots,P_4'$ are the images of the divisors contracted by $\varsigma$, and set
$P_i'=\varsigma(P_i)$ for $i=5,\ldots,8$. The 3-folds~\mbox{$V_{C,\theta}$} and~\mbox{$V_{C,\theta'}$} are constructed from the Aronhold heptads
$P_1,\ldots,P_7$ and $P_1',\ldots,P_7'$, respectively.

Let $\widehat{\mathbb{P}}^3$ be the blow up of $\mathbb{P}^3$ at~\mbox{$P_1,\ldots, P_7$}.  Denote by $F_i$ the exceptional divisor on $\widehat{\mathbb{P}}^3$ over the point $P_i$. Let $L_{ij}$ be the line in~$\mathbb{P}^3$ passing through the points~$P_i$ and~$P_j$ with $i< j$, and let~$\widehat{L}_{ij}$ be its proper transform on $\widehat{\mathbb{P}}^3$.
Flopping the six curves~$\widehat{L}_{ij}$ with~\mbox{$1\leqslant i<j\leqslant 4$}, we obtain a birational map~$\xi$
from~$\widehat{\mathbb{P}}^3$ to another weak Fano 3-fold $\widehat{\mathbb{P}}^{3+}$.
Thus, we obtain the following diagram:
$$
\xymatrix{
\widehat{\mathbb{P}}^3\ar@{-->}[rr]^{\xi}\ar@{->}[dr]^{\nu}&&\ \widehat{\mathbb{P}}^{3+}\ar@{->}[dl]_{\nu^+}\\%
&\bar{V}&
}
$$
Here $\nu$ and $\nu^+$ are contractions of the flopping curves of $\xi$ and $\xi^{-1}$, respectively.

Keeping in mind the decomposition of the standard Cremona transformation into blow ups and blow downs, we see
that $\widehat{\mathbb{P}}^{3+}$ can be represented as the blow up
$\pi^+\colon\widehat{\mathbb{P}}^{3+}\to\mathbb{P}^3$ at the points~\mbox{$P_1',\ldots,P_7'$}. More precisely,
let $\Pi_i$, $1\leqslant i\leqslant 4$, be the plane
passing through the three points among $P_1,\ldots,P_4$ except $P_i$.
The proper transforms of the four planes $\Pi_i$  in $\widehat{\mathbb{P}}^{3+}$ are four disjoint surfaces $\Pi_i^+$
isomorphic to $\mathbb{P}^2$. Let $F_i^+$, $5\leqslant i\leqslant 7$,
be the proper transforms in $\widehat{\mathbb{P}}^{3+}$ of the surfaces $F_i$.
Then the surfaces~$\Pi_1^+,\,\Pi_2^+,\,\Pi_3^+,\,\Pi_4^+,\,F_5^+,\,F_6^+$, and~$F_7^+$  are contracted to the seven points  $P_1',\ldots, P_7'$ on $\mathbb{P}^3$, respectively.

By Proposition~\ref{proposition:Prokhorov-construction}, the 3-folds
$V_{C,\theta}$ and $V_{C,\theta'}$ are the images of the pluri-anticanonical morphisms~$\phi$ and~$\phi^+$ of
$\widehat{\mathbb{P}}^{3}$ and $\widehat{\mathbb{P}}^{3+}$, respectively.
Since $\xi$ is a flop, both $\phi$ and $\phi^+$ factor through~$\bar{V}$,
so that
both~$V_{C,\theta}$ and~$V_{C,\theta'}$ are obtained by contracting one and the same curves on~$\bar{V}$.
Thus,~$V_{C,\theta}$ and $V_{C,\theta'}$ are isomorphic to one and the same double Veronese cone $V$.
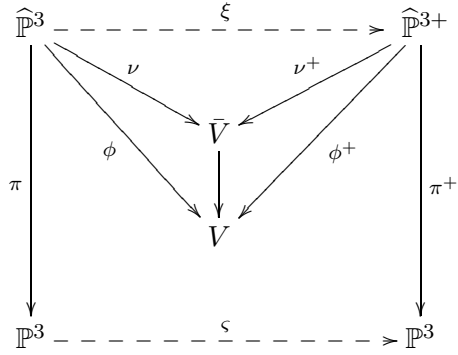
$$
\xymatrix{
\widehat{\mathbb{P}}^3\ar@{->}[ddd]_{\pi}\ar@{->}[ddrr]_{\phi}\ar@{-->}[rrrr]^{\xi}\ar@{->}[drr]^{\nu}&&&&\ \widehat{\mathbb{P}}^{3+}\ar@{->}[ddd]^{\pi^{+}}\ar@{->}[ddll]^{\phi^{+}}\ar@{->}[dll]_{\nu^+}\\%
&&\bar{V}\ar@{->}[d]&&\\
&&V&&\\
\mathbb{P}^3\ar@{-->}[rrrr]^{\varsigma}&&&&\mathbb{P}^3}
$$
This completes the proof of the lemma.
\end{proof}

Lemmas~\ref{lemma:V-i-j} and~\ref{lemma:V-theta-thetaprime} give the following

\begin{corollary}\label{corollary:all-V-same}
Let  $C$ be a smooth plane quartic curve.
Then for any choice of an even theta characteristic $\theta$ on $C$, and for
any choice of an Aronhold heptad in the regular Cayley octad corresponding to~$\theta$,
the $28$-nodal double Veronese cones constructed from these Aronhold heptads as in
Proposition~\ref{proposition:Prokhorov-construction} are isomorphic to each other.
\end{corollary}

Corollary~\ref{corollary:all-V-same}  allows us to put together the following proof.

\begin{proof}[Second proof of Theorem~\ref{theorem:one-to-one}]
Given a $28$-nodal double Veronese cone, we associate to it a smooth plane quartic curve as in Lemma~\ref{lemma:V-to-quartic}.

Let $C$ be a smooth plane quartic curve. Choose an even theta characteristic on $C$. By
Theorem~\ref{theorem:bijection-Theta}, this provides us a regular Cayley octad.
Choose an Aronhold heptad from this regular Cayley octad, and construct a $28$-nodal double Veronese cone
as in Proposition~\ref{proposition:Prokhorov-construction}.
By Corollary~\ref{corollary:all-V-same} the result does not depend on the choice of the even theta characteristic and the
Aronhold heptad, and by Theorem~\ref{theorem:28-can-be-constructed} every $28$-nodal double Veronese cone can be obtained
in this way.
Using Lemma~\ref{lemma:V-to-quartic} once again, we see that the above two constructions are mutually
inverse.
\end{proof}

Recall from \cite[\S\,IX.2]{DolgachevOrtland} that a  (unordered) set of seven distinct odd theta characteristics~\mbox{$\theta_1,\ldots,\theta_7$} on a  smooth plane quartic curve~$C$  is called an \emph{Aronhold system} if they  satisfy the condition that~${\theta_i+\theta_j+\theta_k-K_C}$ is an even theta characteristic for each choice
of three distinct indices $i,j,k$.
This is equivalent to the requirement that the six points of any three of~\mbox{$\theta_1,\ldots,\theta_7$} are not contained in a conic because
\[\left|2K_C-\theta_i-\theta_j-\theta_k\right|=\left|\theta_i+\theta_j+\theta_k-K_C\right|.\]
A given smooth quartic has  exactly~$288$ Aronhold  systems  (see \cite[Proposition~IX.3]{DolgachevOrtland}).

It follows from Lemma~\ref{lemma:even-from-7-odd} that if $P_1,\ldots,P_8$ is a regular Cayley octad, and
$\mathcal{L}$ is the corresponding net of quadrics in $\mathbb{P}^3$, then for every fixed index
$r\in\{1,\ldots,8\}$ the seven odd theta characteristics~\mbox{$\theta_{ri}(\mathcal{L})$, $i\neq r$,} form
an Aronhold system. It turns out that all Aronhold systems arise in this way.

\begin{lemma}[{cf. \cite[Proposition~6.3.11]{Dolgachev}}]
\label{lemma:all-288}
Let $C$ be a smooth plane quartic, and let $\Xi$ be an Aronhold  system on $C$. Then
there exist a unique even theta characteristic $\theta$  on $C$ and a unique point~$P_8$
in the regular Cayley octad
$P_1,\ldots,P_8$ corresponding to $(C,\theta)$ via Theorem~\ref{theorem:bijection-Theta},
such that
\begin{equation}\label{eq:Aronhold-from-Aronhold}
\Xi=\{\theta_{8i}(\mathcal{L})\mid 1\leqslant i\leqslant 7\},
\end{equation}
where $\mathcal{L}$ is the net of quadrics defined by the regular Cayley octad
$P_1,\ldots,P_8$.
\end{lemma}
\begin{proof}
For a given even theta characteristic $\theta$ on $C$, we obtain a regular Cayley octad~\mbox{$P_1,\ldots,P_8$} in~$\mathbb{P}^3$  from Theorem~\ref{theorem:bijection-Theta}. This gives us eight distinct Aronhold systems $\Xi_1,\ldots,\Xi_8$ on $C$ as in~\eqref{eq:Aronhold-from-Aronhold}.
By Lemma~\ref{lemma:even-from-7-odd} one has
$$
\theta=-3K_C+ \sum_{\vartheta\in\Xi_r }\vartheta
$$
for every $r\in\{1,\ldots,8\}$. Hence none of the Aronhold systems $\Xi_r$ can coincide with an Aronhold system
constructed in the same way starting from any other even theta characteristic $\theta'\neq \theta$
on~$C$. Since there are exactly $288$ Aronhold systems and exactly $36$ even theta characteristics on~$C$,
the above construction exhausts all the Aronhold systems on $C$, and the assertion follows.
\end{proof}

Lemmas~\ref{lemma:even-from-7-odd} and~\ref{lemma:all-288}  imply the following.

\begin{corollary}
Let $C$ be a smooth plane quartic, and let $\Xi$ be an Aronhold  system on $C$. Then
\begin{equation}\label{equation:even-theta-from-Aronhold}
\theta=-3K_C+ \sum_{\vartheta\in\Xi }\vartheta
\end{equation}
is an even theta characteristic on $C$. Moreover,
for every even theta characteristic $\theta$ on $C$ there exist exactly eight
Aronhold systems such that $\theta$ is obtained from them as in~\eqref{equation:even-theta-from-Aronhold},
and each of these eight Aronhold systems is in turn obtained from $\theta$ as in Lemma~\ref{lemma:all-288}.
\end{corollary}

To proceed we need to define a certain equivalence relation on the set of Aronhold systems.
Let~$\mathcal{T}_a$ be the set that consists of the pairs $(C, \Xi)$, where $C$ is a smooth quartic curve considered up to isomorphism, and $\Xi$ is an Aronhold system on $C$. Two members $(C_1, \Xi_1)$ and~\mbox{$(C_2, \Xi_2)$} in~$\mathcal{T}_a$ are considered to be equivalent if
\begin{itemize}
\item $C_1=C_2$;
\item one has
\[\sum_{\vartheta\in\Xi_1}\vartheta = \sum_{\vartheta\in\Xi_2}\vartheta;\]
\item there is an automorphism $\sigma\colon C_1\to C_2$ such that $\sigma^*(\Xi_2)=\Xi_1$.
\end{itemize}
In other words, the equivalence comes from the action of stabilizers of the
even theta characteristic attached to $(C,\Xi)$ by~\eqref{equation:even-theta-from-Aronhold}
in the automorphism group of $C$.
Let $\mathcal{T}_A$ be the set of the equivalence classes of members in $\mathcal{T}_a$.
Let $\mathcal{A}$ be the set of Aronhold heptads in $\mathbb{P}^3$ up to projective transformations.
 Define a map
\[\bigtriangleup\colon \mathcal{A}\to \mathcal{T}_A\]
by assigning  to a given Aronhold heptad
the Hessian curve of the corresponding net of quadrics together with the Aronhold system constructed as in~\eqref{eq:Aronhold-from-Aronhold}.

\begin{theorem}\label{theorem:diagram-one-to-one}
The map $\bigtriangleup$ is bijective.
\end{theorem}

\begin{proof}
It is enough to show that for a given smooth plane quartic curve $C$ the map $\Delta$ induces one-to-one correspondence between members of $\mathcal{A}$ with the Hessian curve $C$ and members of $\mathcal{T}_A$ with the first component $C$.

It immediately follows from  Lemma~\ref{lemma:all-288} that the map $\bigtriangleup$ is surjective.

Suppose that $\bigtriangleup (X)=\bigtriangleup(Y) $ for  Aronhold heptads $X$ and $Y$. First of all, the Hessian curves of the nets of quadrics determined by $X$ and $Y$, respectively, are one and the same smooth quartic curve~$C$. Let $\Xi_X$ and $\Xi_Y$ be the Aronhold systems determined  by $X$ and $Y$, respectively.
Then
$$
-3K_C+ \sum_{\vartheta\in\Xi_X }\vartheta=-3K_C+ \sum_{\vartheta\in\Xi_Y }\vartheta
$$
is one and the same even theta characteristic $\theta$ on $C$.
Furthermore, since $\bigtriangleup (X)=\bigtriangleup(Y) $, there is an automorphism $\sigma$ of $C$ such that $\Xi_X=\sigma^*(\Xi_Y)$. Then the automorphism $\sigma$ preserves the linear system~$|K_C+\theta |$ that is the linear system of hyperplane sections of  the Steinerian curve of the net of quadrics $\mathcal{L}$ in $\mathbb{P}^3$
corresponding to~$(C, \theta)$ via Theorem~\ref{theorem:bijection-Theta}.
This means that the automorphism $\sigma$ considered as an automorphism of the Steinerian curve
is induced from a projective transformation of~$\mathbb{P}^3$ preserving the regular Cayley octad defined by $\mathcal{L}$. Thus, we conclude that $X=Y$ in~$\mathcal{A}$.
\end{proof}

\begin{corollary}\label{corollary:Aronhold-sets}
Let $V$ be a $28$-nodal double Veronese cone, and let $C$ be the smooth plane quartic curve
corresponding to $V$ via Theorem~\ref{theorem:one-to-one}. Then
there exists a natural one-to-one correspondence
between the set of
the diagrams~\eqref{equation:Prokhorov-diagram}
considered up to projective transformations of $\mathbb{P}^3$ and  the set of members in $ \mathcal{T}_A$ with the first component in the pair isomorphic to $C$.
\end{corollary}
\begin{proof}
This  immediately follows from Theorem~\ref{theorem:diagram-one-to-one}, because by Corollary~\ref{corollary:all-V-same} there is an obvious one-to-one correspondence
between the set of
the diagrams~\eqref{equation:Prokhorov-diagram}
considered up to projective transformations of~$\mathbb{P}^3$ and  the set of members in $ \mathcal{A}$ with the Hessian curve $C$.
\end{proof}

\begin{corollary}\label{corollary:288-diagram}
For a general smooth plane quartic curve $C$, there are exactly $288$  diagrams~\eqref{equation:Prokhorov-diagram} associated to  $C$,
up to projective transformations of $\mathbb{P}^3$.
\end{corollary}

If the smooth plane quartic curve $C$ is not general, the number of diagrams~\eqref{equation:Prokhorov-diagram} associated to  $C$
up to projective transformations of $\mathbb{P}^3$
may be smaller than~$288$. Note however that this number is always at least~$36$.

\begin{example}
Let $C$ be the Klein quartic curve given in $\mathbb{P}^2$ by the equation
\[x^3y+y^3z+z^3x=0.\]
Then $\Aut(C)\cong\mathrm{PSL}_2(\mathbb{F}_7)$.
There exists a unique $\Aut(C)$-invariant even theta characteristic~$\theta_0$ on~$C$ (see \cite[Example~2]{Burns} or \cite[Example~2.8]{Dolgachev97}).
The remaining $35$ even theta characteristics split into three $\Aut(C)$-orbits of
lengths $7$, $7$, and $21$, respectively (see \cite[(8.3)]{DolgachevKanev}).
Furthermore, according to \cite[Theorem~10.1]{Jeurissen}
the set of $288$ Aronhold systems on $C$ splits into four $\Aut(C)$-orbits of lengths~$8,\,56,\,56$, and $168$,
respectively. This means that the stabilizer of every even theta characteristic~$\theta$ in~$\Aut(C)$ acts transitively
on the set of eight Aronhold systems corresponding to $\theta$ via~\eqref{equation:even-theta-from-Aronhold}.

Now let $V$ be the $28$-nodal double Veronese cone corresponding to $C$ via Theorem~\ref{theorem:one-to-one}.
Corollary~\ref{corollary:Aronhold-sets} tells us that $V$ admits exactly~$36$
diagrams~\eqref{equation:Prokhorov-diagram}
up to projective transformations of~$\mathbb{P}^3$.
One of them corresponds to the $\Aut(C)$-invariant theta characteristic $\theta_0$.
The group~\mbox{$\Aut(C)$} acts on the corresponding three-dimensional projective space
preserving the regular Cayley octad.
Note however that since the regular Cayley octad
in $\mathbb{P}^3$ is a single $\Aut(C)$-orbit (see for instance \cite[Lemma~3.2]{CheltsovShramov-PSL}),
this diagram is not equivariant with respect to the whole
group~\mbox{$\Aut(C)$}. It is equivariant with respect to the
subgroup~\mbox{$\mumu_7\rtimes\mumu_3\subset\Aut(C)$}.
\end{example}

\medskip
There is another natural
way to recover the smooth plane quartic $C$ corresponding to the $28$-nodal double Veronese cone $V$.

Let $P_1,\ldots, P_7$ be seven points in $\mathbb{P}^3$ that form an Aronhold heptad. Let $P_8$ be the eighth base point of the net $\mathcal{L}(P_1,\ldots, P_7)$. Let $\rho\colon\mathbb{P}^3\dasharrow \mathbb{P}^2$ be the linear projection centered at~$P_8$.

\begin{lemma}\label{lemma:Gale}
No three of the seven points $\rho(P_1),\ldots , \rho(P_7)$ are collinear (so that in particular no two of them coincide),
and no six points among them are contained in a conic.
\end{lemma}

\begin{proof}
If three points, say  $\rho(P_1),\rho(P_2)$, and $\rho(P_3)$, lie on a line,
then the four points  $P_1,P_2, P_3$, and~$P_8$ lie on a plane.
This implies that there is an element in~\mbox{$\mathcal{L}(P_1,\ldots, P_7)$} that contains the plane.
Otherwise the restriction of the net $\mathcal{L}(P_1,\ldots, P_7)$ to the plane would define a net of conics passing through the four points
$P_1,P_2, P_3$, and $P_8$.

We now suppose that six points, say  $\rho (P_1),\ldots, \rho (P_6)$, lie on a conic. Then there is a quadric surface
in $\mathbb{P}^3$ passing through the six points $P_1,\ldots, P_6$ together with $P_8$. It follows from Lemma~\ref{lemma:CB}
that the quadric surface should pass through the point $P_7$.
If the base point $P_8$ of the net $\mathcal{L}(P_1,\ldots, P_7)$
is a singular point of a member of the net,
then the base locus of the net~\mbox{$\mathcal{L}(P_1,\ldots, P_7)$}
could not consists of eight distinct points.
\end{proof}

\begin{remark}\label{remark:Gale-recover}
The heptad $\rho(P_1),\ldots , \rho(P_7)$ is the \emph{Gale transform}
of the heptad~\mbox{$P_1,\ldots,P_7$} (see~\cite{EP00}, and in particular the examples following \cite[Corollary~3.2]{EP00}).
Thus, the points~\mbox{$\rho(P_1),\ldots , \rho(P_7)$} uniquely define the Aronhold heptad~\mbox{$P_1,\ldots,P_7$} and
the regular Cayley octad~\mbox{$P_1,\ldots,P_8$} up to projective transformations of~$\mathbb{P}^3$.
\end{remark}

Now let $\alpha\colon S\to \mathbb{P}^2$ be the blow up
of $\mathbb{P}^2$ at the seven points $\rho(P_1),\ldots , \rho(P_7)$.
By Lemma~\ref{lemma:Gale} the surface $S$ is a smooth del Pezzo surface of degree $2$, so that its anticanonical linear system
yields the double cover $\varphi\colon S\to\mathbb{P}^2$ branched along some smooth plane quartic curve $C$.
The quartic curve $C$ coincides with the Hessian curve of the net
of quadrics $\mathcal{L}(P_1,\ldots, P_7)$
(see~\mbox{\cite[Proposition~6.3.11]{Dolgachev}}
or~\mbox{\cite[Proposition~IX.2]{DolgachevOrtland}}).
The above maps are diagramed as follows:
\begin{equation}
\label{equation:non-commutative}
\xymatrix{
V\ar@{-->}[d]_{\kappa}&\widehat{\mathbb{P}}^3\ar@{->}[d]^{\pi}\ar@{->}[l]_{\phi}&&\ S\ar@{->}[dl]_{\alpha}\ar@{->}[d]^{\varphi}\\%
\mathbb{P}^2&\mathbb{P}^3\ar@{-->}[r]_{\rho}&\mathbb{P}^2&\mathbb{P}^2}
\end{equation}
Here $\phi$ is a small resolution of all singular points of the 3-fold $V$,
the morphism $\pi$ is the blow up of~$\mathbb{P}^3$ at the seven distinct points $P_1,\ldots,P_7$,  and the rational map $\kappa$ is given by the half-anticanonical linear system of $V$.

\begin{remark}\label{remark:Aronhold}
The image of the exceptional curve of $\alpha$ over the point $\rho(P_i)$ by the double cover morphism $\varphi$ is the bitangent line of $C$ corresponding to the odd theta characteristic~$\theta_{8i}(\mathcal{L})$
(see~\mbox{\cite[Proposition~IX.2]{DolgachevOrtland}}).
These seven odd theta characteristic form an  Aronhold system.
Moreover, this Aronhold system is the same
as the one given by~\eqref{eq:Aronhold-from-Aronhold}.
\end{remark}

\begin{remark}
Following the diagram \eqref{equation:non-commutative}, one may easily figure out the last statement in Lemma~\ref{lemma:Q1Q2}. Indeed, a fiber of the rational map $\kappa$ is given by the base locus of a pencil contained in the net~$\mathcal{L}(P_1,\ldots, P_7)$. In case of a smooth fiber, the fiber is the proper transform of a smooth intersection  $E$ of  two quadrics in the net. The intersection $E$ is a curve of degree $4$ passing through eight points $P_1,\ldots, P_8$ which are the base points of the net. Thus the projection $\rho$ sends~$E$ to a cubic curve passing through the seven points $\rho(P_1),\ldots, \rho(P_7)$. The proper transform of this cubic curve on the del Pezzo surface $S$ is a smooth member of the anticanonical linear system of $S$. Since the double cover $\varphi$ is given by $|-K_S|$, the latter  proper transform is the pull-back of a line on $\mathbb{P}^2$ by $\varphi$.  It is the double cover of the line branched at the distinct four points at which the line and the quartic curve $C$ intersect. The point of the dual projective plane corresponding to this line is exactly the point over which the original fiber lies.
\end{remark}

\section{$\mathfrak{S}_4$-symmetric $28$-nodal double Veronese cones}
\label{section:S4-Veronese}

In this section
we study the $\mathfrak{S}_4$-equivariant birational geometry of
double Veronese cones introduced in Example~\ref{example:V1-S4}.

Let $C$ and $V$ be  the plane quartic curve and the corresponding $28$-nodal double Veronese cone  in Example~\ref{example:V1-S4}.
We consider the automorphisms of $\mathbb{P}(1,1,1,2,3)$
\[\begin{split}
&\tau\colon [s:t:u:v:w]\mapsto [s:t:u:v:-w],\\
&\sigma_1\colon [s:t:u:v:w]\mapsto [-s:-t:u:v:w],\\
&\sigma_2\colon [s:t:u:v:w]\mapsto [u:s:t:v:w],\\
&\sigma_3\colon [s:t:u:v:w]\mapsto [t:s:u:v:-w].
\end{split}
\]
Since they keep $V$ invariant, they may be regarded as automorphisms of $V$.
Then the involution~$\tau$ is the Galois involution of the double cover of the Veronese cone.
Moreover, by Theorem~\ref{theorem:Aut}, we have
$$
\mathrm{Aut}(V)\cong\mumu_2\times\mathrm{Aut}(C),
$$
where the subgroup $\mumu_2$ is generated by $\tau$.

Let $\mathfrak{G}$ be the subgroup in $\mathrm{Aut}(V)$ that is generated by $\sigma_1$, $\sigma_2$, and $\sigma_3$,
and let $\mathfrak{G}^\prime$ be the subgroup in $\mathrm{Aut}(V)$ generated by $\sigma_1$, $\sigma_2$, and $\tau\circ\sigma_3$.
Then $\mathfrak{G}\cong\mathfrak{G}^\prime\cong\mathfrak{S}_4$, and both of them are projected isomorphically to the same subgroup in $\mathrm{Aut}(C)$.
Meanwhile, the subgroups $\mathfrak{G}$ and $\mathfrak{G}^\prime$ are not conjugate in $\mathrm{Aut}(V)$.

To facilitate the computations, introduce an auxiliary variable
\begin{equation}\label{eq:bar-v-vs-v}
\bar{v}=v-\mu(s^2+t^2+u^2),
\end{equation}
where $\mu=\frac{2\lambda}{3}$.
The defining equation of the 3-fold $V$  may then be rewritten as
\begin{multline}\label{eq:V-bar-v}
w^2=\bar{v}\Big(\bar{v}^2+3\mu\bar{v}(s^2+t^2+u^2)+3\mu^2(s^2+t^2+u^2)^2-g_4(s,t,u)\Big)+\\
+\mu^3(s^2+t^2+u^2)^3-\mu(s^2+t^2+u^2)g_4(s,t,u)+g_6(s,t,u).
\end{multline}
One has
$$
\mu^3(s^2+t^2+u^2)^3-\mu(s^2+t^2+u^2)g_4(s,t,u)+g_6(s,t,u)=4(\lambda-2)^2(\lambda+1)s^2t^2u^2.
$$
Set $\gamma=2(\lambda-2)\sqrt{\lambda+1}$. We see from~\eqref{eq:V-bar-v}
that the 3-fold $V$ contains the surfaces $\Pi_{\pm}$ given by
equations
$$
\bar{v}=w\mp\gamma stu=0,
$$
that is, by equations
\begin{equation}\label{eq:Pi-pm-eqs}
\left\{\aligned
&v=\mu(s^2+t^2+u^2),\\
&w=\pm\gamma stu.
\endaligned
\right.
\end{equation}
Note also that both $\Pi_+$ and $\Pi_-$ are $\mathfrak{G}^\prime$-invariant,
and both of them are not $\mathbb{Q}$-Cartier divisors.
In particular, we see that $\mathrm{rk}\,\mathrm{Cl}(V)^{\mathfrak{G}^\prime}\ne 1$.

Let us introduce a new coordinate
$$
r=\frac{w-\gamma stu}{\bar{v}}.
$$
On the 3-fold $V$, one has
$$
r=\frac{\bar{v}^2+3\mu\bar{v}(s^2+t^2+u^2)+3\mu^2(s^2+t^2+u^2)^2-g_4(s,t,u)}{w+\gamma stu}.
$$
This gives us a birational map of $V$ to the complete intersection in $\mathbb{P}(1,1,1,1,2,3)$ given by
\begin{equation}\label{eq:w-exclusion}
\left\{\aligned
&r\bar{v}=w-\gamma stu,\\
&r(w+\gamma stu)=\bar{v}^2+3\mu\bar{v}(s^2+t^2+u^2)+3\mu^2(s^2+t^2+u^2)^2-g_4(s,t,u).\\
\endaligned
\right.
\end{equation}
Now excluding the variable $w$ using the first equation
in~\eqref{eq:w-exclusion} and expressing $\bar{v}$ in terms of
$s$, $t$, $u$, and~$v$ by~\eqref{eq:bar-v-vs-v},
we obtain a quartic hypersurface in~\mbox{$\mathbb{P}(1,1,1,1,2)$}
with homogeneous coordinates $s$, $t$, $u$, $r$, and $v$
that is given by
$$
r(r\bar{v}+2\gamma stu)=\bar{v}^2+3\mu\bar{v}(s^2+t^2+u^2)+3\mu^2(s^2+t^2+u^2)^2-g_4(s,t,u)
$$
or equivalently by
\begin{equation}\label{eq:double-cover-from-V1}
v^2+v\Big(\mu(s^2+t^2+u^2)-r^2\Big)+\mu^2(s^2+t^2+u^2)^2-g_4(s,t,u)+\mu r^2(s^2+t^2+u^2)-2\gamma rstu=0.
\end{equation}
We denote this 3-fold by $W$.

We have constructed a birational map $V\dasharrow W$ that fits into the following $\mathfrak{G}^\prime$-equivariant diagram
\begin{equation}
\label{equation:unprojection}
\xymatrix{
\widehat{V}\ar@{->}[rr]\ar@{->}[d]&&V\ar@{-->}[lld]\\%
W&&}
\end{equation}
Here the horizontal arrow is the blow up of the
surface $\Pi_+$, and the vertical arrow
is the blow down of the proper transform of the surface $\Pi_-$ to the point
$$
[s:t:u:r:v]=[0:0:0:1:0],
$$
which is a smooth point of the 3-fold $W$.

The 3-fold $W$ is the double cover of $\mathbb{P}^3$ branched over a quartic surface
$$
\Big(\mu(s^2+t^2+u^2)-r^2\Big)^2-4\Big(\mu^2(s^2+t^2+u^2)^2-g_4(s,t,u)+\mu(s^2+t^2+u^2)r^2-2\gamma rstu\Big)=0,
$$
where we consider $s$, $t$, $u$ and $r$ as homogeneous coordinates on $\mathbb{P}^3$.
This quartic surface has exactly~$16$ nodes, which means that it is a Kummer surface, so that $W$ would also have exactly~$16$ nodes.

Let $S$ be the surface in $W$ that is cut out by $r=0$.
Then $S$ is $\mathfrak{G}^\prime$-invariant.
It is the double cover of $\mathbb{P}^2$  branched over the quartic curve
$$
4g_4(s,t,u)-3\mu^2(s^2+t^2+u^2)^2=0,
$$
where we consider $s$, $t$ and $u$ as homogeneous coordinates on $\mathbb{P}^2$.
Keeping in mind that $\mu=\frac{2\lambda}{3}$, we see that this curve is given by
\begin{equation}\label{eq:branch-curve-section}
s^4+t^4+u^4+\lambda(t^2u^2+s^2u^2+s^2t^2)=0.
\end{equation}
Therefore, it is isomorphic to our original curve $C$. In particular, the surface $S$ is a smooth del Pezzo surface.

\begin{lemma}
\label{lemma:S4-reps}
Let $\mathbb{U}=\mathrm{Pic}(S)\otimes\mathbb{C}$.
One has an isomorphism of $\mathfrak{G}^{\prime}$-representations
\begin{equation}
\label{eq:W3W3prime}
\mathbb{U}\cong\mathbb{I}\oplus\mathbb{I}^\prime\oplus\mathbb{W}_3\oplus\mathbb{W}_3^\prime,
\end{equation}
where $\mathbb{I}$ is the trivial representation, $\mathbb{I}^\prime$ is the sign representation,
and $\mathbb{W}_3$ and $\mathbb{W}_3^\prime$ are non-isomorphic irreducible three-dimensional representations.
\end{lemma}

\begin{proof}
Recall from~\eqref{eq:double-cover-from-V1} that the surface $S$ is given in $\mathbb{P}(1,1,1,2)$
with weighted homogeneous coordinates $s,t,u$, and $v$ by equation
\begin{equation*}
v^2+\frac{2\lambda}{3}v(s^2+t^2+u^2)+\frac{4\lambda^2}{9}(s^2+t^2+u^2)^2-g_4(s,t,u)=0,
\end{equation*}
where
\[g_4(s,t,u)=\frac{\lambda^2+12}{3}(s^4+t^4+u^4)+\frac{2(\lambda^2+6\lambda)}{3}(t^2u^2+s^2u^2+t^2s^2).\]
As before, the parameter $\lambda$ varies in $\mathbb{C}\setminus\{\pm 2, -1\}$.

Let $\mathrm{v}$ be a new variable defined as
\[
\mathrm{v}=\frac{v}{2}+\frac{\lambda}{6}(s^2+t^2+u^2).
\]
One can verify that, after this change of coordinates, $S$ is given by the following equation in~\mbox{$\mathbb{P}(1,1,1,2)$} with weighted homogeneous coordinates $s,t,u$, and $\mathrm{v}$, cf.~\eqref{eq:branch-curve-section}:
\begin{equation*}
\mathrm{v}^2=s^4+t^4+u^4+\lambda(t^2u^2+s^2u^2+t^2s^2).
\end{equation*}

Note that $\mathfrak{G}^{\prime}$ is the subgroup in $\Aut(S)$ generated by the transformations
\[\begin{split}
&[s:t:u:\mathrm{v}]\mapsto [-s:-t:u:\mathrm{v}],\\
&[s:t:u:\mathrm{v}]\mapsto [u:s:t:\mathrm{v}],\\
&[s:t:u:\mathrm{v}]\mapsto [t:s:u:\mathrm{v}].
\end{split}\]

Let $\mathfrak{G}^{\prime\prime}$
be the subgroup in $\Aut(S)$ generated by the transformations
\[\begin{split}
&[s:t:u:\mathrm{v}]\mapsto [-s:-t:u:\mathrm{v}],\\
&[s:t:u:\mathrm{v}]\mapsto [u:s:t:\mathrm{v}],\\
&[s:t:u:\mathrm{v}]\mapsto [t:s:u:-\mathrm{v}].
\end{split}\]
Then $\mathfrak{G}^{\prime\prime}\cong\mathfrak{G}^{\prime}$, but these subgroups are not conjugate in $\Aut(S)$.
They intersect by the even elements, and every odd element of $\mathfrak{G}^{\prime}$ differs
from the corresponding element of  $\mathfrak{G}^{\prime\prime}$ by the Galois involution of the anticanonical double cover $S\to\mathbb{P}^2$.
This involution acts as a multiplication by~$-1$ on the orthogonal complement
to $K_S$ with respect to the intersection form in~\mbox{$\mathrm{Pic}(S)\otimes\mathbb{C}$}.
Thus, to prove~\eqref{eq:W3W3prime}, it is enough to show that
$$
\mathbb{U}\cong\mathbb{I}^{\oplus 2}\oplus\mathbb{W}_3\oplus\mathbb{W}_3^\prime
$$
as a representation of the group $\mathfrak{G}^{\prime\prime}$.

For ease of notation, take
$\alpha,\beta\in\mathbb{C}$ so that
\[
\alpha^2=\frac{-\lambda+2\sqrt{-\lambda-1}}{\lambda+2}
\quad\text{and}\quad
\beta=\frac{\lambda}{2}(1+\alpha^2).
\]

Consider the following twelve $(-1)$-curves on the surface $S$:
\[\begin{array}{ll}
 L_1 : s+\alpha t=\mathrm{v}-\beta t^2-u^2=0; & L_2 : s-\alpha t=\mathrm{v}-\beta t^2-u^2=0;\\
L_3 : t+\alpha s=\mathrm{v}+\beta s^2+u^2=0; & L_4 : t-\alpha s=\mathrm{v}+\beta s^2+u^2=0; \\
L_5 : t+\alpha u=\mathrm{v}-\beta u^2-s^2=0; & L_6 : t-\alpha u=\mathrm{v}-\beta u^2-s^2=0; \\
L_{7} : s+\alpha u=\mathrm{v}+\beta  u^2+t^2=0;& L_{8}: s-\alpha u=\mathrm{v}+\beta u^2+t^2=0.\\
L_9 : u+\alpha t=\mathrm{v}+\beta t^2+s^2=0; & L_{10} : u-\alpha t=\mathrm{v}+\beta t^2+s^2=0; \\
L_{11} : u+\alpha s=\mathrm{v}-\beta  s^2-t^2=0;& L_{12} : u-\alpha s=\mathrm{v}-\beta s^2-t^2=0.
\end{array}\]
They form a single $\mathfrak{G}^{\prime\prime}$-orbit.
Moreover, these $12$ curves split into a disjoint union of the following pairs of intersecting lines:
$$
L_1\cup L_2, \quad L_3\cup L_4, \quad L_5\cup L_6, \quad L_7\cup L_{8}, \quad L_9\cup L_{10}, \quad L_{11}\cup L_{12}.
$$
Hence, there exists a $\mathfrak{G}^{\prime\prime}$-equivariant conic bundle
$\zeta\colon S\to\mathbb{P}^1$ such that these pairs are exactly its singular fibers.
This implies that $\mathrm{Pic}(S)^{\mathfrak{G}^{\prime\prime}}$ is generated by $K_S$ and a fiber of this conic bundle.
In particular, the surface $S$ is not $\mathfrak{G}^{\prime\prime}$-minimal.
With \cite[Table~7]{DolgachevIskovskikh}, we conclude that $\mathfrak{G}^{\prime\prime}$
is the \emph{odd lift} of $\mathfrak{S}_4$ to $\mathrm{Aut}(S)$
(see \cite[\S6.6]{DolgachevIskovskikh} for the terminology).

Therefore, the subgroup $\mathfrak{G}^\prime$ is the \emph{even lift} of $\mathfrak{S}_4$ to $\mathrm{Aut}(S)$.
Hence, by \cite[Table~7]{DolgachevIskovskikh}, the surface~$S$ is $\mathfrak{G}^{\prime}$-minimal,
so that $\mathbb{U}$ contains a unique trivial $\mathfrak{G}^{\prime}$-representation (generated by $K_S$).
This in turn implies that $\mathbb{U}$ does not contain
summands isomorphic to the sign representation of the group~$\mathfrak{G}^{\prime\prime}$.

As $\mathfrak{G}^{\prime\prime}$-representation,
$\mathbb{U}$ splits as
\begin{equation}\label{eq:U-preliminary-splitting}
\mathbb{U}\cong\mathbb{I}^{\oplus 2}\oplus\overline{\mathbb{U}},
\end{equation}
where $\overline{\mathbb{U}}$ is some six-dimensional representation of $\mathfrak{G}^{\prime\prime}$
that does not contain one-dimensional subrepresentations.
We conclude that $\overline{\mathbb{U}}$ splits into a sum of two- and three-dimensional irreducible~$\mathfrak{G}^{\prime\prime}$-representations.
If there is a two-dimensional summand in $\overline{\mathbb{U}}$, then it must be a sum of three two-dimensional summands.
In this case the action of $\mathfrak{G}^{\prime\prime}$
on $\mathrm{Pic}(S)$ would be not faithful, which contradicts \cite[Corollary~8.2.40]{Dolgachev}.
Therefore, we see that $\overline{\mathbb{U}}$ is a sum of two three-dimensional irreducible representations of $\mathfrak{G}^{\prime\prime}$. Note that in the splitting~\eqref{eq:U-preliminary-splitting}
the summand~$\mathbb{I}^{\oplus 2}$ is generated by $K_S$ and a fiber of the conic bundle $\zeta$,
while $L_1$, $L_3$, $L_5$, $L_7$, $L_9$, and $L_{11}$ form a basis in~$\overline{\mathbb{U}}$.

Let $\sigma$ be the element of the group $\mathfrak{G}^{\prime\prime}$ that acts by
$$
[s:t:u:\mathrm{v}]\mapsto[t:s:u:-\mathrm{v}].
$$
Then
$$
\sigma(L_1)=L_3, \quad \sigma(L_3)=L_1, \quad \sigma(L_5)=L_{7}, \quad \sigma(L_7)=L_5, \quad \sigma(L_9)=L_{11},
\quad \sigma(L_{11})=L_9.
$$
Thus, the trace of $\sigma$ in $\overline{\mathbb{U}}$ equals $0$.
The element $\sigma$ corresponds to a transposition in $\mathfrak{G}^{\prime\prime}\cong\mathfrak{S}_4$.
This means that $\overline{\mathbb{U}}\cong\mathbb{W}_3\oplus\mathbb{W}_3^\prime$ as $\mathfrak{G}^{\prime\prime}$-representation.
Therefore, the isomorphism~\eqref{eq:W3W3prime} of $\mathfrak{G}^{\prime}$-representations holds.
\end{proof}

The commutative diagram~\eqref{equation:unprojection} gives a surjective $\mathfrak{G}^{\prime}$-module homomorphism of
$\mathrm{Cl}(V)$ onto~\mbox{$\mathrm{Cl}(W)$} with one-dimensional kernel generated by the surface $\Pi_-$.
On the other hand, restricting divisors in~$\mathrm{Cl}(W)$ to $S$, we obtain a $\mathfrak{G}^{\prime}$-module
 homomorphism $$
\mathrm{Cl}(W)\to\mathrm{Pic}(S),
$$
which is injective (cf. \cite[Corollary~3.9.3]{Prokhorov}).
Note that $\mathrm{Cl}(W)$ is of rank $7$, while $\mathrm{Pic}(S)$ is of rank~$8$.

As $\mathfrak{G}^{\prime}$-representations,
we have
\begin{equation*}
\mathrm{Cl}(V)\otimes\mathbb{C}\cong\mathbb{I}^{\oplus 2}\oplus\mathbb{V},
\end{equation*}
where $\mathbb{I}$ is a trivial $\mathfrak{G}^{\prime}$-representation, and $\mathbb{V}$ is some six-dimensional $\mathfrak{G}^{\prime}$-representation.
Note that the summand~\mbox{$\mathbb{I}^{\oplus 2}$} is generated by $K_V$ and $\Pi_-$.
We see that $\mathbb{V}$ is a subrepresentation of~\mbox{$\mathrm{Pic}(S)\otimes\mathbb{C}$},
so that~$\mathbb{V}$ is a sum of two irreducible
three-dimensional $\mathfrak{G}^{\prime}$-representations~$\mathbb{W}_3$ and~$\mathbb{W}_3^\prime$ by Lemma~\ref{lemma:S4-reps}.
As a by-product, we obtain

\begin{corollary}
\label{corollary:class-group-V-G-prime}
There is an isomorphism of $\mathfrak{G}^{\prime}$-representations
\begin{equation*}
\Cl(V)\otimes\mathbb{C}\cong
\mathbb{I}^{\oplus 2}\oplus\mathbb{W}_3\oplus\mathbb{W}_3^\prime.
\end{equation*}
In particular,
$\mathrm{rk}\,\mathrm{Cl}(V)^{\mathfrak{G}^\prime}=2$
and $\mathrm{rk}\,\mathrm{Cl}(W)^{\mathfrak{G}^\prime}=1$.
\end{corollary}

Recall that the subgroups $\mathfrak{G}$ and $\mathfrak{G}^{\prime}$ are both isomorphic to $\mathfrak{S}_4$.
They intersect by the even elements, and every odd element of $\mathfrak{G}$ differs
from the corresponding element of  $\mathfrak{G}^{\prime}$ by the Galois involution $\tau$. This involution acts as a multiplication by $-1$ on the orthogonal complement, with respect to
the intersection form \eqref{equation:E7},  to $K_V$ in $\mathrm{Cl}(V)\otimes\mathbb{C}$.
Thus, Corollary~\ref{corollary:class-group-V-G-prime}
implies that~\mbox{$\mathrm{Cl}(V)\otimes\mathbb{C}$} contains
a unique trivial subrepresentation of the group $\mathfrak{G}$.

\begin{corollary}
\label{corollary:class-group-V-G}
One has $\mathrm{rk}\,\mathrm{Cl}(V)^{\mathfrak{G}}=1$.
\end{corollary}

\section{Birational rigidity}
\label{section:rigidity}

In this section we study $G$-birationally
rigidity of nodal double Veronese cones.

Let $V$ be a nodal del Pezzo 3-fold of degree $1$
(not necessarily with 28 nodes).
Then $V$ can be given in $\mathbb{P}(1,1,1,2,3)$ with weighted homogeneous coordinates $s$, $t$, $u$, $v$, and $w$
(of weights $1$, $1$, $1$, $2$, and~$3$, respectively)
by equation
$$
w^2=v^3+vh_4(s,t,u)+h_6(s,t,u)
$$
for some homogeneous polynomials $h_4(s,t,u)$ and $h_6(s,t,u)$ of degree $4$ and $6$, respectively.
Then
$$
-K_V\sim 2H,
$$
where $H$ is an ample Cartier divisor on $V$ such that $H^3=1$.
The linear system $|H|$ has one base point, which we denote by $O$.
Recall that the 3-fold $V$ is smooth at the point~$O$.

The linear system $|2H|$ is free from base points and gives a double cover $V\to\mathbb{P}(1,1,1,2)$.
As before, denote by $\tau$ the Galois involution of this double cover.

Let $\kappa\colon V\dasharrow\mathbb{P}^2$ be the projection given by
$$
[s:t:u:v:w]\mapsto[s:t:u].
$$
Then $\kappa$ is a rational map given by the linear system $|H|$.
In particular, $\kappa$ is $\mathrm{Aut}(V)$-equivariant. Furthermore,
there exists an $\mathrm{Aut}(V)$-equivariant commutative diagram
$$
\xymatrix{
&\widetilde{V}\ar@{->}[ld]_{\psi}\ar@{->}[rd]^{\tilde{\kappa}}&\\%
V\ar@{-->}[rr]_{\kappa}&&\mathbb{P}^2}
$$
where $\psi$ is the blow up at the point $O$,
and $\tilde{\kappa}$ is a morphism whose general fiber is an elliptic curve.
Note that every fiber of the morphism $\tilde{\kappa}$ is irreducible.
Moreover, we have an exact sequence of groups
$$
\xymatrix{1\ar@{->}[r]&\Gamma\ar@{->}[r]&\mathrm{Aut}(V)\ar@{->}[r]&\mathrm{Aut}(\mathbb{P}^2),}
$$
where $\Gamma$ is a finite subgroup in $\mathrm{Aut}(V)$ that contains $\tau$.
If $h_4(x,y,z)$ is not a zero polynomial, then~$\Gamma$ is generated by $\tau$.
If $h_4(x,y,z)$ is a zero polynomial, then $\Gamma$ is a cyclic group of order $6$
that is generated by $\tau$ and a map given by
$$
[s:t:u:v:w]\mapsto[s:t:u:\epsilon_3 v:w],
$$
where $\epsilon_3$ is a primitive cube root
of unity (cf. \cite[\S6.7]{DolgachevIskovskikh}).

Let $G$ be a finite subgroup of $\mathrm{Aut}(V)$.
Denote its image in $\mathrm{Aut}(\mathbb{P}^2)$ by $\overline{G}$.
Suppose that
$$
\mathrm{rk}\,\mathrm{Cl}\big(V\big)^G=1,
$$
so that $V$ is a $G$-Mori fibre space (see \cite[Definition~1.1.5]{CheltsovShramov-V5}).
If $\mathbb{P}^2$ contains a $\overline{G}$-fixed point,
then there exists a $G$-equivariant commutative diagram as follows:
$$
\xymatrix{
&\check{V}\ar@{->}[ld]_{\alpha}\ar@{->}[rrrdd]^{\upsilon}&\\%
V\ar@{-->}[d]_{\kappa}&&&&\\
\mathbb{P}^2\ar@{-->}[rrrr]&&&&\mathbb{P}^1}
$$
Here $\mathbb{P}^2\dasharrow\mathbb{P}^1$ is the projection from the $\overline{G}$-fixed point,
$\alpha$ is the maximal extraction whose center is the fiber of $\kappa$ over the $\overline{G}$-fixed point,
and $\upsilon$ is a fibration into del Pezzo surfaces of degree $1$.
Thus, in this case, the 3-fold $V$ cannot be $G$-birationally rigid.

If $V$ is smooth and $\mathbb{P}^2$ contains  no $\overline{G}$-fixed points,
then $V$ is $G$-birationally super-rigid.
This follows from \cite{Grinenko2003} and \cite{Grinenko2004} (cf. \cite[Remark~1.19]{CheltsovShramov2009}).
It is conjectured that the following two conditions are equivalent
for a nodal double Veronese cone $V$ with an action of
a finite group $G$ such that~\mbox{$\mathrm{rk}\,\mathrm{Cl}\big(V\big)^G=1$}:
\begin{itemize}
\item $\mathbb{P}^2$ does not contain $\overline{G}$-fixed points,
\item the 3-fold $V$ is $G$-birationally rigid.
\end{itemize}
At present, its proof is out of reach.
Instead, we give a very simple proof of a weaker result,
which implies Theorem~\ref{theorem:rigid}.
 To state the result, note that a surface $\Pi$ in the 3-fold $V$ is said to be a \emph{plane} if~$\Pi\cong\mathbb{P}^2$ and $H^2\cdot\Pi=1$.
For instance, by \cite[Theorem~7.2]{Prokhorov}
every $28$-nodal double Veronese cone contains exactly $126$ planes.

\begin{theorem}
\label{theorem:V1-G-rigid}
Let $V$ be a nodal
double Veronese cone  with an action of
a finite group $G$ such that $\mathrm{rk}\,\mathrm{Cl}\big(V\big)^G=1$.
Suppose that the following three conditions are satisfied:
\begin{enumerate}
\item $\mathbb{P}^2$ does not contain $\overline{G}$-fixed points;

\item the singular locus $\mathrm{Sing}(V)$ does not contain $G$-orbits of length $3$;

\item for every $G$-irreducible curve $D$ in $V$ such that $H\cdot D$ is either $2$ or $3$, and $\kappa(D)$ is a curve in $\mathbb{P}^2$ of degree $H\cdot D$,
there is a plane in $V$ that contains~$D$.
\end{enumerate}
Then $V$ is $G$-birationally super-rigid.
\end{theorem}

\begin{corollary}
\label{corollary:V1-G-rigid}
Let $V$  and $G$ be as in Theorem~\ref{theorem:V1-G-rigid}. Suppose that $\mathbb{P}^2$ does not contain $\overline{G}$-fixed points.
If  $\mathbb{P}^2$ does not contain  $\overline{G}$-invariant curves of degree $2$ or $3$,
then $V$ is $G$-birationally super-rigid.
\end{corollary}

\begin{remark}
\label{remark:Z-3-singular-points}
We cannot drop condition (2) in Theorem~\ref{theorem:V1-G-rigid}.
Suppose that $\mathbb{P}^2$ does not contain $\overline{G}$-fixed points,
but $\mathrm{Sing}(V)$ contains a $G$-orbit of length $3$.
Let $\beta\colon \breve{V}\to V$ be the blow up at this $G$-orbit.
Then it follows from
\cite{JahnkeRadloff} and
\cite[Theorem~1.5]{CheltsovShramovPrzyjalkowski}
that there exists a crepant $G$-birational morphism~$\nu\colon\breve{V}\to X$
such that $X$ is a double cover of $\mathbb{P}^3$ branched
in an irreducible sextic surface.
Let $\iota$ be the involution in $\mathrm{Aut}(X)$ of this double cover,
and put
$$
\rho=\beta\circ\nu^{-1}\circ\iota\circ\nu\circ\beta^{-1}.
$$
Then $\rho\in\mathrm{Bir}^{G}(V)$.
Moreover, if $\nu$ is small, then $\rho$ is not biregular.
\end{remark}

Before proving Theorem~\ref{theorem:V1-G-rigid}, let us use it to prove Theorem~\ref{theorem:rigid}.

\begin{proof}[Proof of Theorem~\ref{theorem:rigid}]
Suppose that the $28$-nodal double Veronese cone $V$ is $G$-birationally rigid.
Then the group $\overline{G}$ does not have
fixed points on $\mathbb{P}^2$, and thus also has no
fixed points on its projectively dual plane. Let $C$ be the smooth plane quartic curve
that is constructed as
the projectively dual curve of the discriminant curve $\check{C}$ of the half-anticanonical rational elliptic fibration
$\kappa\colon V\dasharrow \mathbb{P}^2$ (see Theorem~\ref{theorem:one-to-one} or Lemma~\ref{lemma:V-to-quartic}).
Then the group $\overline{G}$ acts faithfully on~$C$.
We see from Lemma~\ref{lemma:must-contain-S4}
that $C$ can be given
by equation~\eqref{eq:S4-quartic-general}, the group~$\overline{G}$ must contain a subgroup isomorphic to~$\mathfrak{S}_4$,
and~$\mathbb{P}^2$ (as well as its projective dual) can be identified
with a projectivization of an irreducible three-dimensional
representation of $\mathfrak{S}_4$.
Using once again Theorem~\ref{theorem:one-to-one}
and formula~\eqref{eq:V1}, we conclude that $V$ is a double Veronese cone from
Example~\ref{example:V1-S4}.
Furthermore, by Theorem~\ref{theorem:Aut} the group $G$ contains a subgroup isomorphic to~$\mathfrak{S}_4$.
Now it follows from Corollary~\ref{corollary:class-group-V-G-prime} that $G$ contains a subgroup
conjugate to~$\mathfrak{G}$.

Now we suppose that $V$ is a double Veronese cone
from Example~\ref{example:V1-S4}, and $G=\mathfrak{G}$.
Then $\mathbb{P}^2$ does not contain $\overline{G}$-fixed points,
and $\mathrm{rk}\,\mathrm{Cl}(V)^G=1$ by Corollary~\ref{corollary:class-group-V-G}.
To complete the proof, we have to show that $V$ is $G$-birationally super-rigid.
Suppose that this is not a case.
Let us seek for a contradiction.

Note that $\overline{G}\cong\mathfrak{S}_4$, and $\mathbb{P}^2$ contains a unique $\overline{G}$-orbit of length $3$.
This orbit consists of the points~$[0:0:1]$,  $[0:1:0]$ and  $[1:0:0]$.
Taking partial derivatives of~\eqref{eq:V1} and using the expressions for $g_4$ and $g_6$
provided in Example~\ref{example:V1-S4},
we see that~\mbox{$\mathrm{Sing}(V)$}
does not contain points that are mapped by $\kappa$ to $[0:0:1]$,  $[0:1:0]$ or  $[1:0:0]$.
In particular, we see that the singular locus $\mathrm{Sing}(V)$ does not contain $\overline{G}$-orbits of length $3$.

Using Theorem~\ref{theorem:V1-G-rigid}, we conclude that $V$ contains a $G$-irreducible curve $D$ on $V$ such
that~\mbox{$H\cdot D$} is either $2$ or $3$,
the curve $D$ is not contained in any plane in $V$,
and $\kappa(D)$ is a curve in $\mathbb{P}^2$ of degree~$H\cdot D$.

Observe that there exist a
unique $\overline{G}$-invariant conic
and a unique $\overline{G}$-invariant
cubic curve in $\mathbb{P}^2$.
The former curve is given by $s^2+t^2+u^2=0$,
and the latter curve is given by $stu=0$.

Suppose that $H\cdot D=3$.
Then $\kappa(D)$ is given by $stu=0$, so that $D$ consists of three irreducible components that are mapped isomorphically by $\kappa$
to the lines $s=0$, $t=0$ and $u=0$ in $\mathbb{P}^2$.

Let $D_s$ be the irreducible component of the curve $D$ that
is mapped by $\kappa$ to the line~\mbox{$s=0$}. Then the stabilizer $G_s$ of $D_s$ in $G\cong \mathfrak{S}_4$
is isomorphic to the dihedral group of order $8$. However,
its action on $D_s$ is not faithful, but the action of its
quotient isomorphic to $\mumu_2\times\mumu_2$ is faithful. Therefore,
a general $G_s$-orbit in $D_s$ has length $4$. Consider the pencil
$\mathcal{R}$ of $G$-invariant
surfaces in~$V$ generated by the surfaces
$v=0$ and $s^2+t^2+u^2=0$. Since a surface from $\mathcal{R}$
has intersection number $2$ with $D_s$,
we conclude that there is a surface $R$ in $\mathcal{R}$
that contains $D_s$. Since~$R$ is $G$-invariant, it contains
the whole curve $D$.

The surface $R$ is given by equation
$$
v=\mu(s^2+t^2+u^2)
$$
for some $\mu\in\mathbb{C}$. Thus, the curve $D$ is contained in the subset in $V$ that is cut out by
$$
\left\{\aligned
&v=\mu(s^2+t^2+u^2),\\
&stu=0.\\
\endaligned
\right.
$$

If $\mu=\frac{2\lambda}{3}$, then
the equation $v=\mu(s^2+t^2+u^2)$ cuts out two planes in $V$, which are given by
$$
\left\{\aligned
&v=\mu(s^2+t^2+u^2),\\
&w=\gamma stu,\\
\endaligned
\right.
$$
where $\gamma=\pm 2(\lambda-2)\sqrt{\lambda+1}$ (see~\eqref{eq:Pi-pm-eqs}).
Thus, in this case, the curve $D$ is contained in a plane, which is impossible by assumption.
Hence, we conclude that $\mu\ne\frac{2\lambda}{3}$.

We see that the irreducible component $D_s$
is contained in the subset in $V$ that is cut out by
$$
\left\{\aligned
&v=\mu(t^2+u^2),\\
&s=0.\\
\endaligned
\right.
$$
If this subset were irreducible and reduced, then it would coincide with $D_s$, which would contradict~${H\cdot D_s=1}$.
Hence, the above subset is either reducible or non-reduced.
Algebraically, this simply means that the polynomial
$$
P(t,u)=\mu^3(t^2+u^2)^3-\mu g_4(0,t,u)(t^2+u^2)+g_6(0,t,u)
$$
must be a complete square (this includes the possibility
for the above polynomial to be zero, which corresponds to the
non-reduced case).
The polynomial $P(t,u)$ simplifies as
\[\left(\mu^3-\mu\frac{\lambda^2+12}{3}-\frac{2\lambda(\lambda-6)(\lambda+6)}{27}\right)(t^2+u^2)^3
-4(\lambda-2)\left(\mu-\frac{2}{3}\lambda\right)(t^2+u^2)t^2u^2.\]
This shows that the polynomial $P(t,u)$ cannot be a zero polynomial because $\lambda\ne 2$ and $\mu\ne \frac{2\lambda}{3}$.
Suppose that~$P(t,u)$ is a complete square. Then it must be of the form
\[
P(t,u)=(a_3t^3+a_2t^2u+a_1tu^2+a_0u^3)^2,
\]
where $a_i$'s are constants depending on $\lambda$ and $\mu$. We can see directly from the simplified~$P(t,u)$ above
 that $P(t,u)$ is invariant under switching variables $t$ and $u$ and that  it contains $t$  and $u$ with only even exponents. We therefore obtain
$$
a_0^2=a_3^2, \quad a_1^2+2a_0a_2=a_2^2+2a_1a_3, \quad a_0a_1=a_2a_3=0,  \quad a_0a_3+a_1a_2=0.
$$
However, these yield $a_0=a_1=a_2=a_3=0$, which is impossible. Consequently, one has~\mbox{$H\cdot D\ne 3$}.

We see that $H\cdot D=2$.
Then $\kappa(D)$ is given by equation $s^2+t^2+u^2=0$ in $\mathbb{P}^2$.
In particular, one has~\mbox{$D\cong \mathbb{P}^1$}.
The action of $G\cong\mathfrak{S}_4$ on $D$ is faithful,
so that any $G$-orbit in $D$ has length at least~$6$. Consider the pencil
$\mathcal{R}$ as above.
Since a surface from $\mathcal{R}$
has intersection number $4$ with~$D_s$,
we conclude that every surface in $\mathcal{R}$
contains $D$.
In particular, $D$ is contained in a surface given by equation
$$
v=\frac{2\lambda}{3}(s^2+t^2+u^2).
$$
As above, we conclude that $D$ is contained in a plane,
which gives a contradiction.
This completes the proof of Theorem~\ref{theorem:rigid}.
\end{proof}

Now we are equipped to prove Theorem~\ref{theorem:V1-G-rigid}.
Recall that $V$ is a nodal
double Veronese cone with an action of
a finite group $G$ such that $\mathrm{rk}\,\mathrm{Cl}\big(V\big)^G=1$,
and assumptions (1)--(3) of Theorem~\ref{theorem:V1-G-rigid} hold.
Namely, we
suppose that $\mathbb{P}^2$ does not contain $\overline{G}$-fixed points;
that the singular locus $\mathrm{Sing}(V)$
does not contain $G$-orbits of length $3$;
 that for any $G$-irreducible curve $D$ in $V$ such that~\mbox{$H\cdot D$} is either $2$ or $3$,
and $\kappa(D)$ is a curve in $\mathbb{P}^2$ of degree $H\cdot D$,
there is a plane in $V$ containing the curve $D$.

\begin{remark}\label{remark:overline-G-orbits}
Since there are no $\overline{G}$-fixed points on
$\mathbb{P}^2$, there are also no $\overline{G}$-invariant lines.
This implies that there are no $\overline{G}$-orbits
contained in a line. In other words, every
$\overline{G}$-orbit contains three non-collinear points,
and in particular no $\overline{G}$-orbits
of length $2$.
Furthermore, by Lemma~\ref{lemma:4-general-points}
every
$\overline{G}$-orbit contains four points such that
no three of them are collinear.
\end{remark}

To~prove Theorem~\ref{theorem:V1-G-rigid}, we have to show that
the 3-fold~$V$ is $G$-birationally super-rigid.
Suppose that it is not.
Then there exists a $G$-invariant mobile linear system $\mathcal{M}$ on the 3-fold~$V$ such that
$$
\mathcal{M}\sim nH,
$$
and the log pair $(V,\frac{2}{n}\mathcal{M})$ is not canonical
(see for instance \cite[Theorem~3.3.1]{CheltsovShramov-V5}).

Let $Z$ be a $G$-center of non-canonical
singularities of the log pair $(V,\frac{2}{n}\mathcal{M})$.
Recall that $O$ denotes the (unique) base point of the linear
system $|H|$ on $V$, and $O$ is a smooth point of $V$.

\begin{lemma}
\label{lemma:Z-not-O}
The $G$-center  $Z$ cannot be the point $O$.
\end{lemma}

\begin{proof}
Let $M$ be a general surface in $\mathcal{M}$,
and let $E$ be a general fiber of the map $\kappa$.
Then~\mbox{$E\not\subset M$}, so that
$$
n=M\cdot E\geqslant\mathrm{mult}_{O}\big(E\big)\cdot\mathrm{mult}_{O}\big(M\big)\geqslant\mathrm{mult}_{O}\big(M\big).
$$
On the other hand, the Zariski tangent space $T_{O,V}$ is an irreducible representation of the group~$G$,
because $\mathbb{P}^2$ does not contain $\overline{G}$-fixed points.
Thus, using Lemma~\ref{lemma:mult-2}, we conclude that $O$
is not a center of non-canonical singularities
of the log pair $(V,\frac{2}{n}\mathcal{M})$.
\end{proof}

\begin{lemma}
\label{lemma:Z-not-smooth-point}
If $Z$ is $0$-dimensional, then  $Z\subset\mathrm{Sing}(V)$.
\end{lemma}

\begin{proof}
Suppose that the $G$-orbit $Z$ consists of smooth points of the 3-fold $V$.
Let $M_1$ and~$M_2$ be general surfaces in the linear system $\mathcal{M}$.
Then
\begin{equation}
\label{equation:Z-not-smooth-point}
\mathrm{mult}_{P}\big(M_1\cdot M_2\big)>n^2
\end{equation}
for every $P\in Z$ by \cite[Corollary~3.4]{Corti2000}.
Let us show that this leads to a contradiction.

By Lemma~\ref{lemma:Z-not-O}, we have $Z\ne O$, so that $\kappa(Z)$ is a
well-defined
$\overline{G}$-orbit.
Moreover, $\kappa(Z)$ consists of at least $3$ points,
and  contains three non-collinear points
by Remark~\ref{remark:overline-G-orbits}.
Denote them by $P_1$, $P_2$ and~$P_3$.

Let $E_1$, $E_2$ and $E_3$ be fibers of the map $\kappa$ over the points $P_1$, $P_2$ and $P_3$, respectively.
Then
$$
M_1\cdot M_2=m\big(E_1+E_2+E_3\big)+\Delta,
$$
where $m$ is a non-negative integer, and $\Delta$ is an effective one-cycle whose support does not contain the curves $E_1$, $E_2$ and $E_3$.
Then $m\leqslant\frac{n^2}{3}$, since
$$
n^2=H\cdot M_1\cdot M_2=H\cdot\Big(m\big(E_1+E_2+E_3\big)+\Delta\Big)=3m+H\cdot\Delta\geqslant 3m.
$$

Let $O_1$, $O_2$ and $O_3$ be points in $Z$ that are mapped to $P_1$, $P_2$ and $P_3$, respectively.
Let~\mbox{$\delta=\mathrm{mult}_{O_i}\big(E_i\big)$}. Then either $\delta=1$ or $\delta=2$.
Moreover, it follows from \eqref{equation:Z-not-smooth-point} that
$$
\mathrm{mult}_{O_i}\big(\Delta\big)>n^2-m\delta
$$
for each point $O_i$.

Let $\mathcal{B}$ be a linear subsystem in $|2H|$ consisting of all
surfaces that contain the curves $E_1,\,E_2$ and~$E_3$.
Then $\mathcal{B}$ does not have other base curves,
since $P_1$, $P_2$ and $P_3$ are not collinear.
Thus, for a general surface $B$ in $\mathcal{B}$, we have
$$
2n^2-6m=B\cdot\Delta\geqslant\sum_{i=1}^{3}\mathrm{mult}_{O_i}\big(\Delta\big)>3\Big(n^2-m\delta\Big).
$$
This implies
$$
-4m\geqslant m\delta-6m>n^2,
$$
which is a contradiction.
\end{proof}

\begin{lemma}
\label{lemma:Z-not-singular-point}
The center $Z$ is $1$-dimensional.
\end{lemma}

\begin{proof}
Suppose that $\dim Z=0$, so that $Z$ is a $G$-orbit. Then $Z\subset\mathrm{Sing}(V)$ by Lemma~\ref{lemma:Z-not-smooth-point}.
Let~$r$ be the length of the $\overline{G}$-orbit $\kappa(Z)$.
Note that $r\geqslant 3$ by Remark~\ref{remark:overline-G-orbits}.
Denote the points of $\kappa(Z)$ by $P_1,\ldots,P_r$.
Let $E_i$ be the fiber of the map $\kappa$ over the point $P_i$.
Then $E_i$ is an irreducible curve because~\mbox{$E_i\cdot H=1$}.
Since the arithmetic genus of $E_i$ equals $1$,
we see that $E_i$ has at most one
singular point, and thus
contains at most one
singular point of the 3-fold $V$.
Hence, since~\mbox{$E_i\cap Z\ne\varnothing$}, each curve~$E_i$ contains exactly one singular point of $V$,
which must be a singular point of the curve~$E_i$.
In particular, we see that
the length of the $G$-orbit $Z$ also equals~$r$.

Let $O_i$ be the singular point of the curve $E_i$.
Then $Z$ consists of the points $O_1,\ldots,O_r$.
By assumption, we have $r\geqslant 4$.

The $\overline{G}$-orbit $\kappa(Z)$ contains four points in $\mathbb{P}^2$ such that no three of
them are collinear by Remark~\ref{remark:overline-G-orbits}.
Without loss of generality, we may assume that these points are $P_1$, $P_2$, $P_3$, and~$P_4$.
Let $\mathcal{B}$ be a linear subsystem in $|2H|$ consisting of all
surfaces that contain the curves $E_1$, $E_2$, $E_3$ and $E_4$.
Then $\mathcal{B}$ does not contain other base curves.

Let~$\beta\colon \overline{V}\to V$ be the blow up of the points $O_1$, $O_2$, $O_3$ and $O_4$,
and let $F_1$, $F_2$, $F_3$, $F_4$ be the $\beta$-exceptional surfaces that are mapped to the points $O_1$, $O_2$, $O_3$, $O_4$, respectively.
Denote by~$\overline{\mathcal{B}}$ the proper transforms on $\overline{V}$ of the linear system $\mathcal{B}$.
Then
$$
-K_{\overline{V}}\sim\overline{\mathcal{B}}\sim\beta^*\big(2H\big)-\big(F_1+F_2+F_3+F_4\big),
$$
which implies that the divisor $-K_{\overline{V}}$ is nef, because it intersects the proper transform of each curve $E_i$ trivially.

Denote by $\overline{\mathcal{M}}$ the proper transform of the linear system $\mathcal{M}$ on the 3-fold $\overline{V}$.
Then
$$
\overline{\mathcal{M}}\sim_{\mathbb{Q}}\beta^*\big(nH\big)-m\big(F_1+F_2+F_3+F_4\big)
$$
for some non-negative rational number $m$.
Let $\overline{M}_1$ and $\overline{M}_2$ be general surfaces in $\overline{\mathcal{M}}$.
Then
$$
0\leqslant\overline{B}\cdot\overline{M}_1\cdot\overline{M}_2=2n^2-8m^2,
$$
which gives $m\leqslant\frac{n}{2}$. This is impossible by \cite[Theorem~3.10]{Corti2000}.
\end{proof}

Thus, Lemma~\ref{lemma:Z-not-singular-point}  concludes that $Z$ is a $G$-irreducible curve.
Denote by $Z_1,\ldots,Z_r$ the irreducible components of $Z$.
Thus, if $r=1$, then $Z=Z_1$ is an irreducible curve.
In this case, $\kappa(Z)$ is neither a point nor a line, because $\mathbb{P}^2$ does not contain $\overline{G}$-fixed points.
In particular, we see that $H\cdot Z\ne 1$.

Since $Z$ is a $G$-center of non-canonical singularities of the log pair $(V,\frac{2}{n}\mathcal{M})$, we have
$$
\mathrm{mult}_{Z_i}\big(\mathcal{M}\big)>\frac{n}{2}.
$$
Let $M_1$ and $M_2$ be general surfaces in $\mathcal{M}$.
Then
$$
n^2=H\cdot M_1\cdot M_2\geqslant \sum_{i=1}^{r}\mathrm{mult}_{Z_i}\big(M_1\cdot M_2\big)>\frac{n^2}{4}rH\cdot Z_i,
$$
so that we have the following four possibilities:
\begin{itemize}
\item[($\mathrm{A}$)] $r=1$, $H\cdot Z=2$, and $\kappa(Z)$ is a smooth conic in $\mathbb{P}^2$;
\item[($\mathrm{B}$)] $r=1$, $H\cdot Z=3$, and $\kappa(Z)$ is a smooth cubic in $\mathbb{P}^2$;
\item[($\mathrm{C}$)] $r=3$, $H\cdot Z=3$, and $\kappa(Z)$ is a union of three lines;
\item[($\mathrm{D}$)] $r=3$, $H\cdot Z=3$,
and $\kappa(Z)$ is a $\overline{G}$-orbit of length $3$.
\end{itemize}

We claim that the case ($\mathrm{D}$) is impossible.
Indeed, if $\kappa(Z)$ is a $\overline{G}$-orbit of length $3$,
then the linear system~$|H|$ contains a surface $H_{12}$ passing through
$Z_1$ and $Z_2$. Let $M$ be a general element of
the linear system $\mathcal{M}$. Then
$$
n=H\cdot M\cdot H_{12}\geqslant \mathrm{mult}_{Z_1}\big(M\big)+
\mathrm{mult}_{Z_2}\big(M\big)>n,
$$
which is absurd.

Thus, we see that $H\cdot Z$ is either $2$ or $3$ and that $\kappa(Z)$ is a curve of degree $H\cdot Z$.
Therefore, by assumption, the 3-fold $V$ has a plane $\Pi$ that contains the curve $Z$.
Let $L$ be a general line in this plane, so that $H\cdot L=1$,
the intersection $L\cap Z$ consists of exactly $H\cdot Z$ points,
and $L$ is not contained in the base locus of the linear system $\mathcal{M}$.
Then, for a general surface $M\in\mathcal{M}$, we have
$$
n=L\cdot M\geqslant\mathrm{mult}_Z\big(M\big)\cdot\big| L\cap Z\big|\geqslant2\mathrm{mult}_Z\big(M\big)\geqslant2\mathrm{mult}_{Z}\big(\mathcal{M}\big)>n,
$$
which is a contradiction. This completes the proof of Theorem~\ref{theorem:V1-G-rigid}.

\section{Questions and problems}
\label{section:question}

In this section we discuss several open questions concerning
$28$-nodal double Veronese cones.

\medskip
\textbf{Del Pezzo surfaces of degree $2$.}
Let~$V$ be a $28$-nodal double Veronese cone,
let $C$ be the  plane quartic curve corresponding to $V$
by Theorem~\ref{theorem:one-to-one},
and let~$S$ be the del Pezzo surface of degree $2$ constructed as the double
cover of $\mathbb{P}^2$ branched along $C$.
We know from Corollary~\ref{corollary:Aut-V-vs-S}
that
$$
\mathrm{Aut}(V)\cong\mumu_2\times\mathrm{Aut}(C)\cong\Aut(S),
$$
although there is no obvious choice for a natural isomorphism
between $\Aut(V)$ and $\Aut(S)$.
Also, one has
$$
\Cl(V)\cong\mathbb{Z}^8\cong\Pic(S).
$$
Recall from \cite[Corollary~7.1.4]{Prokhorov}
that the group $\Aut(V)$ acts faithfully on $\Cl(V)$,
while by~\mbox{\cite[Corollary~8.2.40]{Dolgachev}} the group
$\Aut(S)$ acts faithfully on~\mbox{$\Pic(S)$}.
It would be interesting to know
if there exists a natural identification of the corresponding
representations. More precisely, we ask the following:

\begin{question}
\label{question:V1-C4}
Does there exist an isomorphism $\mathrm{Aut}(V)\cong\mathrm{Aut}(S)$
under which $\mathrm{Cl}(V)\otimes\mathbb{C}$
is isomorphic to $\mathrm{Pic}(S)\otimes\mathbb{C}$ as representations
of this group?
\end{question}

Recall that both $\Cl(V)$ and $\Pic(S)$ have a naturally
defined intersection form.
The canonical class~$K_V$ is invariant with respect to the action of
the group $\Aut(V)$ on $\Cl(V)$, and the canonical class $K_S$ is
invariant with respect to the action of $\Aut(S)$ on $\Pic(S)$.
Moreover, in some cases these are the only trivial subrepresentations
in $\mathrm{Cl}(V)\otimes\mathbb{C}$ and
$\mathrm{Pic}(S)\otimes\mathbb{C}$; for instance, this is the case
when $\Aut(C)\cong\mathrm{PSL}_2(\mathbb{F}_7)$.
Therefore, if there exists a natural identification
of automorphism groups as required in Question~\ref{question:V1-C4},
then for any $28$-nodal Veronese double cone $V$ and the corresponding
del Pezzo surface $S$ there must exist an identification
of the seven-dimensional representations
of the group~\mbox{$\Aut(V)\cong\Aut(S)$} in the orthogonal
complement~\mbox{$\mathrm{Cl}(V)^{\perp K_V}\otimes\mathbb{C}$} to~$K_V$ in
$\mathrm{Cl}(V)\otimes\mathbb{C}$ and in
the orthogonal complement $\mathrm{Pic}(S)^{\perp K_S}\otimes\mathbb{C}$
to $K_S$ in~\mbox{$\mathrm{Pic}(S)\otimes\mathbb{C}$}.

Recall that both lattices $\mathrm{Cl}(V)^{\perp K_V}$ and
$\mathrm{Pic}(S)^{\perp K_S}$ are isomorphic to the lattice
$\mathbf{E}_7$ (see~\mbox{\cite[Theorem~1.7]{Prokhorov}}
and \cite[Remark~1.8]{Prokhorov} for the former, and \cite[\S8.2.6]{DolgachevIskovskikh}
for the latter).
We point out that despite this fact in general
one cannot hope for an equivariant
identification of the lattices~\mbox{$\mathrm{Cl}(V)^{\perp K_V}$} and
$\mathrm{Pic}(S)^{\perp K_S}$ that takes into account the intersection
forms, even if the answer to  Question~\ref{question:V1-C4}
is positive.

Indeed, let $C$ be a smooth plane quartic curve with an action of the symmetric group
$\mathfrak{S}_4$ as in Example~\ref{example:V1-S4}.
Let us use the notation of~\S\ref{section:S4-Veronese}.
It follows from Corollary~\ref{corollary:class-group-V-G-prime}
and Lemma~\ref{lemma:S4-reps} that
each of the $\mathfrak{G}^\prime$
representations $\mathrm{Cl}(V)\otimes\mathbb{C}$
and $\mathrm{Pic}(S)\otimes\mathbb{C}$ contains a unique
two-dimensional subrepresentation.
This implies that the lattice $\Cl(V)$ contains a unique
$\mathfrak{G}^\prime$-invariant primitive sublattice $\Lambda_V$ of rank~$2$,
and the lattice $\Pic(S)$ contains a unique
$\mathfrak{G}^\prime$-invariant primitive sublattice~$\Lambda_S$ of rank~$2$.
However, the sublattice $\Lambda_V$ is generated by the half-anticanonical
divisor $H$ and the plane $\Pi_{-}$, so that the intersection
form on $\Lambda_V$ is given by the matrix
$$
\left(
\begin{array}{cc}
H^2 & H\cdot\Pi_{-}\\
H\cdot\Pi_{-} & \Pi_{-}^2
\end{array}
\right)=
\left(
\begin{array}{cc}
1 & 1\\
1 & -1
\end{array}
\right).
$$
On the other hand, it follows from the proof of
Lemma~\ref{lemma:S4-reps} that
the sublattice $\Lambda_S$ is generated by~$-K_S$
and a fiber $F$ of a conic bundle, so that the intersection
form on $\Lambda_S$ is given by the matrix
$$
\left(
\begin{array}{cc}
K_S^2 & -K_S\cdot F\\
-K_S\cdot F & F^2
\end{array}
\right)=
\left(
\begin{array}{cc}
2 & 2\\
2 & 0
\end{array}
\right).
$$
Thus, we see that the lattices $\Lambda_V$
and $\Lambda_S$ are not isomorphic.

One situation when we can easily establish the isomorphism required
in Question~\ref{question:V1-C4} is as follows.
Let $G$ be a subgroup in $\mathrm{Aut}(\mathbb{P}^3)$ such that there exists
a $G$-invariant Aronhold heptad~\mbox{$P_1,\ldots,P_7$}.
Then the diagram~\eqref{equation:non-commutative}
is $G$-equivariant, which implies that
$\mathrm{Cl}(V)\otimes\mathbb{C}$ is isomorphic
to~\mbox{$\mathrm{Pic}(S)\otimes\mathbb{C}$}
as $G$-representations.

\medskip
\textbf{Elliptic fibration.}
The following construction was pointed out to us by Alexander
Kuznetsov.
Let~$C$ be a smooth plane quartic, and let $S$ be the corresponding
del Pezzo surface of degree~$2$ with the anticanonical morphism
$\varphi\colon S\to\mathbb{P}^2$. Let $\check{\mathbb{P}}^2$ be the projectively dual plane of $\mathbb{P}^2$.
Let $p_1$ and $p_2$ be the projections
of the flag variety
$$
\mathcal{F}=\mathrm{Fl}(1,2,3)\subset\mathbb{P}^2\times\check{\mathbb{P}}^2
$$
to the first and the second factor of $\mathbb{P}^2\times\check{\mathbb{P}}^2$,
respectively.
Consider the fiber product
$$
\xymatrix{
 & Y\ar@{->}[ld]\ar@{->}[rd]^{\gamma} & & \\
S\ar@{->}[rd]_{\varphi} & &
\mathcal{F}\ar@{->}[ld]^{p_1}\ar@{->}[rd]_{p_2} & \\
& \mathbb{P}^2 & & \check{\mathbb{P}}^2
}
$$
Note that there is a natural action of the group
$\Aut(S)\cong\mumu_2\times\Aut(C)$ on $Y$.
The morphism
$$
\kappa_Y=p_2\circ\gamma\colon Y\to\check{\mathbb{P}}^2
$$
is an elliptic fibration
whose discriminant curve is projectively dual to~$C$.

Now let $V$ be the $28$-nodal double Veronese cone constructed from the plane quartic~$C$ by Theorem~\ref{theorem:one-to-one}.
Then the half-anticanonical map $\kappa\colon V\dasharrow\check{\mathbb{P}}^2$ is
a rational elliptic fibration. Let~$\tilde{V}$ be the blow up of
$V$ at the unique base point of $\kappa$, and let $\tilde{\kappa}\colon \tilde{V}\to\check{\mathbb{P}}^2$
be the corresponding regular elliptic fibration. Then the discriminant curve of $\tilde{\kappa}$
is projectively dual to~$C$.
Moreover, it follows from Lemma~\ref{lemma:Q1Q2} that over every point in $\check{\mathbb{P}}^2$ the fibers
of the elliptic fibrations $\tilde{\kappa}$
and~$\kappa_Y$ are isomorphic to each other (cf. Lemma~\ref{lemma:V-from-elliptic-fibration}).

Furthermore, it follows from Lemma~\ref{lemma:Q1Q2} that (at least over the complement
to the discriminant curve) the elliptic fibration $\kappa_Y$ is the relative $\mathrm{Pic}^2$ of the elliptic
fibration~$\tilde{\kappa}$.
Indeed, let $E$ be a smooth fiber of $\tilde{\kappa}$, and let $P_1$ and $P_2$ be two points on $E$.
Recall that $E$ is a proper transform of the base locus~$E_{\mathcal{Q}}$ of some pencil $\mathcal{Q}$ of quadrics in $\mathbb{P}^3$.
The two points $P_1, P_2\in E_{\mathcal{Q}}$ define a line~$L$ passing through them.
Thus they define a unique quadric $Q$ in $\mathcal{Q}$ containing the line $L$, and also
the family of lines on $\mathcal{Q}$ which contains $L$. Note that the fiber of
the elliptic fibration~$\kappa_Y$ over the point $\tilde{\kappa}(E)$ can be identified with the double cover
of $\mathcal{Q}\cong\mathbb{P}^1$ branched in the locus of singular quadrics of $\mathcal{Q}$
or, in other words, with the Hilbert scheme of lines
contained in the quadrics of $\mathcal{Q}$. Therefore, the pair of points~$(P_1,P_2)$
defines a point on the fiber of $\kappa_Y$ over~$\tilde{\kappa}(E)$. This point
depends only on the class of the pair~$(P_1,P_2)$ in $\mathrm{Sym}^2(E)$, since every other pair from
the same class determines the same quadric~$Q$ and a line from the same family of lines on $Q$ where $L$ lies.
We conclude that the fiber of~$\kappa_Y$ over $\tilde{\kappa}(E)$ is identified with~$\mathrm{Pic}^2(E)$.
It remains to notice that this construction can be performed in the family, and thus the assertion
about the fibrations follows.

One of the consequences of the above construction is as follows.
Since the elliptic fibration~$\tilde{\kappa}$ is obtained by blowing up
$V$ at a single point, it comes with a distinguished section (which corresponds to the
eighth base point of the net of quadrics in the construction of $V$ from an Aronhold heptad).
Therefore, there exist a natural birational map between $V$ and~$Y$.

\begin{question}
Is there a more explicit geometric relation between the 3-fold $Y$ and the
$28$-nodal double Veronese cone $V$?
\end{question}

\medskip
\textbf{Intersections of quadrics and cubics.}
The following construction was also explained to us by Alexander
Kuznetsov. Let $P_1,\ldots,P_8$ be a regular Cayley octad, and let $\mathcal{L}$ be the corresponding net of quadrics in $\mathbb{P}^3$. Then the lines contained in the quadrics of $\mathcal{L}$ are parameterized by a three-dimensional variety
$Z$ (cf.~\cite[\S1.2]{Reid}). Then $Z$ can be regarded as a subvariety of the Grassmannian~$\mathrm{Gr}(2,4)$.
Let~$\mathcal{U}^\vee$ be the universal quotient bundle on~$\mathrm{Gr}(2,4)$ (or, which is the same, the dual
of the tautological subbundle).
Furthermore, $Z$ can be described as the degeneracy locus
of the morphism of vector bundles
$$
\mathcal{O}_{\mathrm{Gr}(2,4)}\to \mathrm{S}^2\mathcal{U}^\vee
$$
corresponding to~$\mathcal{L}$. Therefore, the 3-fold $Z$ is an intersection of a quadric
and a cubic in~$\mathbb{P}^5$. Note that $Z$ has $28$ singular points corresponding to the lines
passing through pairs of points $P_i$ and $P_j$. In particular, its intermediate Jacobian is trivial.

\begin{question}
Is there a natural geometric relation between the 3-fold $Z$ and the
$28$-nodal double Veronese cone $V$ constructed from the net of quadrics~$\mathcal{L}$?
Is the 3-fold $Z$ rational?
\end{question}

\begin{question}
Does the 3-fold $Z$ depend on the choice of the net of quadrics~$\mathcal{L}$, or only on the
corresponding plane quartic curve?
\end{question}

\medskip
\textbf{Other 3-folds constructed from plane quartics.}
There are other interesting classes of 3-folds related to
plane quartic curves. For instance, a construction due to S.\,Mukai assigns
to a general plane quartic
a smooth Fano 3-fold of genus $12$ (and anticanonical degree $22$),
see~\cite{Mukai92}, \cite{Mukai04}, and~\cite{RanestadSchreyer} for details.

\begin{question}
Is there a natural geometric relation between the
$28$-nodal double Veronese cone~$V$ and the Fano 3-fold $V_{22}$
of genus 12 constructed from the same smooth plane quartic~$C$?
\end{question}

We point out that both of the above varieties are rational, so
there always exists a birational map between them. However,
one cannot choose such a map to be $\Aut(C)$-equivariant. Indeed,
if~\mbox{$\Aut(C)\cong\mathrm{PSL}_2(\mathbb{F}_7)$},
then both $V$ and $V_{22}$
are $\mathrm{PSL}_2(\mathbb{F}_7)$-birationally super-rigid
by Theorem~\ref{theorem:rigid} and~\cite[Theorem~1.10]{CheltsovShramov-PSL}.

\medskip
\textbf{Degenerations.}
The following question was asked by Vyacheslav Shokurov and Yuri Prokhorov.

\begin{question}
Can one extend the one-to-one correspondence between smooth plane quartics and
$28$-nodal double Veronese cones given by Theorem~\ref{theorem:one-to-one}
to the (mildly) singular case
(like, double Veronese cones with $c\mathbf{A}_1$-singularities, or
double Veronese cones with $26$ nodes and one nice $c\mathbf{A}_2$-singularity,
or plane quartics with a single node)?
Furthermore, can one include in this one-to-one correspondence
the plane quartic that is a union of a cuspidal cubic and its tangent line in the cusp
(in other words, the quartic that gives rise to a del Pezzo surface of degree $2$ with an $\mathbf{E}_7$-singularity)?
\end{question}


\bigskip

\appendix
\section{Computer-aided calculations}
\label{section:appendix}

This appendix explains,  from a computational point of view, 
how~\eqref{equation:j-function2} is obtained. 
In particular, a way to extract explicit formulae 
of  the covariants $g_4(s,t,u)$ and $g_6(s,t,u)$ from a given 
plane quartic~\eqref{equation:quartic} is described, 
which enables us to produce concrete examples of $28$-nodal double 
Veronese cones from smooth quartic curves as in Example~\ref{example:V1-S4}.
We point out that there is a method to write down 
the polynomial $g_4(s,t,u)$ using a smaller amount of 
explicit computations. It is 
based on the expression for the cubic invariant in the coefficients of the
plane quartic given in~\cite[\S293]{Salmon} (see 
the end of~\cite[\S1]{Ottaviani}).

Since  the denominator and the numerator of the $j$-function 
in \eqref{equation:j-function} are symmetric polynomials 
in~\mbox{$x_1, x_2, x_3, x_4$}, the $j$-function 
in~\eqref{equation:j-function} may be regarded as a rational function 
in $b_0, b_1, b_2, b_3, b_4$ of~\eqref{equation:quartic2}. 
Using the simple MAGMA (\cite{Magma}) function

\begin{minipage}[m]{.95\linewidth}
\vspace{0.5cm}

\hrule width 15cm

\vspace{0.1cm}

\noindent Computing the numerator (\texttt{N}) and  the denominator (\texttt{D}) in terms of symmetric functions

\vspace{0.1cm}

\hrule width 15cm

\begin{alltt}
Q<x1,x2,x3,x4> := PolynomialRing(RationalField(), 4);
Num:=(x1-x2)^2*(x4-x3)^2-(x1-x3)*(x4-x2)*(x4-x1)*(x3-x2);
Den:=(x1-x2)^2*(x1-x3)^2*(x1-x4)^2*(x2-x3)^2*(x2-x4)^2*(x3-x4)^2;
Q<b1, b2, b3, b4> := PolynomialRing(RationalField(), 4);
I, N := IsSymmetric(Num, Q);
I, D := IsSymmetric(Den, Q);
D;  
N;
\end{alltt}
\hrule width 15cm

\end{minipage}
\vspace{0.5cm}

\noindent we are able to obtain the denominator and the numerator of the $j$-function (respectively \texttt{D}  and  \texttt{N}), which after homogenization with respect to $b_0$ result in

\[\aligned & (x_1-x_2)^2(x_4-x_3)^2-(x_1-x_3)(x_4-x_2)(x_4-x_1)(x_3-x_2)=\frac{1}{b_0^2}\Big(-3b_1b_3 +12 b_0b_4+b_2^2\Big);\\
& (x_1-x_2)^2(x_1-x_3)^2(x_1-x_4)^2(x_2-x_3)^2(x_2-x_4)^2(x_3-x_4)^2 =
 \frac{1}{b_0^6}\Big(-27b_1^4b_4^2 + 18b_1^3b_2b_3b_4 - \\ &- 4b_1^3b_3^3 - 4b_1^2b_2^3b_4 + b_1^2b_2^2b_3^2
    + 144b_0b_1^2b_2b_4^2 - 6b_0b_1^2b_3^2b_4 - 80b_0b_1b_2^2b_3b_4     + 18b_0b_1b_2b_3^3 -\\&-
    192b_0^2b_1b_3b_4^2 +16b_0b_2^4b_4 - 4b_0b_2^3b_3^2 - 128b_0^2b_2^2b_4^2 + 144b_0^2b_2b_3^2b_4 -
    27b_0^2b_3^4 + 256b_0^3b_4^3\Big).\\ \endaligned
\]
Therefore, regarded as a rational function in $b_0, b_1, b_2, b_3, b_4$, the $j$-function in \eqref{equation:j-function} can be expressed as follows:
\[\begin{split}j(b_0, b_1, b_2, b_3, b_4)&=\frac{2^8(-3b_1b_3 +12 b_0b_4+b_2^2)^3}{\left\{
\aligned
&-27b_1^4b_4^2 + 18b_1^3b_2b_3b_4 - 4b_1^3b_3^3 - 4b_1^2b_2^3b_4 + b_1^2b_2^2b_3^2
    + 144b_0b_1^2b_2b_4^2 -\\
    &- 6b_0b_1^2b_3^2b_4 - 80b_0b_1b_2^2b_3b_4 + 18b_0b_1b_2b_3^3 -
    192b_0^2b_1b_3b_4^2 +\\&+ 16b_0b_2^4b_4 - 4b_0b_2^3b_3^2 - 128b_0^2b_2^2b_4^2 + 144b_0^2b_2b_3^2b_4 -
    27b_0^2b_3^4 + 256b_0^3b_4^3\\
\endaligned\right\}}\\
&=1728\frac{4h_2(b_0, b_1, b_2, b_3, b_4)^3}{4h_2(b_0, b_1, b_2, b_3, b_4)^3-27h_{3}(b_0, b_1, b_2, b_3, b_4)^2},\\
\end{split}\]
where
\[h_2(b_0, b_1, b_2, b_3, b_4)=\frac{1}{3}\Big(-3b_1b_3 +12 b_0b_4+b_2^2\Big);\]
\[ h_{3}(b_0, b_1, b_2, b_3, b_4)=\frac{1}{27}\Big(72b_0b_2b_4 - 27b_0b_3^2 - 27b_1^2b_4 + 9b_1b_2b_3 - 2b_2^3\Big).\]
Since $b_0,\ldots, b_4$ are homogeneous polynomials of 
degree $4$ in $s$, $t$, and~$u$, 
we can obtain expressions for~\mbox{$h_2(b_0, b_1, b_2, b_3, b_4)$} and
$h_{3}(b_0, b_1, b_2, b_3, b_4)$ as homogeneous polynomials of degrees $8$ and $12$, respectively, in $s, t, u$. 
Starting with a quartic equation
\[
\sum_{i+j+k=4} a_{ijk}x^iy^jz^k=0, 
\]
where $a_{ijk}$ are independent variables,
we can calculate $b_0,\ldots, b_4$, using the identities 
given after \eqref{equation:quartic2}, as in the code below. 

\begin{minipage}[m]{.95\linewidth}
\vspace{0.5cm}

\hrule width 11.7cm

\vspace{0.1cm}

\noindent Calculating $h_2$ and $h_3$ in terms of $s,t,u$ with general coefficients $a_{ijk}$

\vspace{0.1cm}

\hrule width 11.7cm

\begin{alltt}
P<a400,a310,a301,a220,a202,a211,a130,a103,a121,a112,a040,a031,a022,a013,a004,
s,t,u>:=PolynomialRing(Rationals(),18);
b0:=a004*s^4 - a103*s^3*u + a202*s^2*u^2 - a301*s*u^3 + a400*u^4;
b1:=4*a004*s^3*t - a013*s^3*u - 3*a103*s^2*t*u + a112*s^2*u^2 + 2*a202*s*t*u^2 
-a211*s*u^3 - a301*t*u^3 + a310*u^4;
b2:=6*a004*s^2*t^2 - 3*a013*s^2*t*u + a022*s^2*u^2 - 3*a103*s*t^2*u 
+ 2*a112*s*t*u^2 - a121*s*u^3 + a202*t^2*u^2 - a211*t*u^3 + a220*u^4;
b3:=4*a004*s*t^3 - 3*a013*s*t^2*u + 2*a022*s*t*u^2 - a031*s*u^3 - a103*t^3*u 
+ a112*t^2*u^2 - a121*t*u^3 + a130*u^4;
b4:=a004*t^4 - a013*t^3*u + a022*t^2*u^2 - a031*t*u^3 + a040*u^4;
h2:=-3*b1*b3 + b2^2 + 12*b4*b0;
h3:=8/3*b0*b2*b4-b0*b3^2-b1^2*b4+1/3*b1*b2*b3-2/27*b2^3;
\end{alltt}
\vspace{0.1cm}

\hrule width 11.7cm
\end{minipage}
\vspace{0.5cm}

\noindent It can be checked, for example using the command 
\texttt{Factorisation(h2)}, 
that $h_2$ and $h_3$ are divisible by $u^4$ and~$u^6$, respectively. Therefore, we are able to put
  \[h_2 (b_0, b_1, b_2, b_3, b_4)=u^4g_4(s,t,u)\quad\text{and}\quad
 h_3 (b_0, b_1, b_2, b_3, b_4)=u^6g_6(s,t,u),\]
 where $g_4(s,t,u)$ and $g_6(s,t,u)$ are 
 homogeneous polynomials of degrees $4$ and $6$, respectively, 
 in~\mbox{$s, t, u$}. 
Consequently, the rational function $j(b_0, b_1, b_2, b_3, b_4)$ can be presented as a rational function $j_C$
in~$\check{\mathbb{P}}^2$ with variables \mbox{$s, t, u$} via
 \[j_C(s,t,u)= 1728\frac{4g_4(s,t,u)^3}{4g_4(s,t,u)^3-27g_{6}(s,t,u)^2}\cdot\]



\begin{thebibliography}{dJSBVdV90}

\bibitem[Ahm17]{Ahmadinezhad}
H.\ Ahmadinezhad, \emph{On conjugacy classes of the {K}lein simple group in
  {C}remona group}, Glasg. Math. J. \textbf{59} (2017), no.~2, 395--400.

\bibitem[Avi16]{Avilov-IntTwoQuadrics}
A.\ Avilov, \emph{Automorphisms of threefolds that can be represented as an
  intersection of two quadrics}, Mat. Sb. \textbf{207} (2016), no.~3, 3--18.

\bibitem[Avi18a]{Avilov-Segre}
\bysame, \emph{Automorphisms of singular three-dimensional cubic
  hypersurfaces}, Eur. J. Math. \textbf{4} (2018), no.~3, 761--777.

\bibitem[Avi18b]{Avilov-RealSegre}
\bysame, \emph{On forms of the {S}egre cubic}, 2018, {\tt arXiv:1812.11363}.

\bibitem[Avi19]{Avilov-15points}
\bysame, \emph{Biregular and birational geometry of double coverings of a
  projective space with ramification in a quartic with 15 ordinary double
  points}, Izv. Ross. Akad. Nauk Ser. Mat. \textbf{83} (2019), no.~3, 5--14.

\bibitem[BCP97]{Magma}
W.\ Bosma, J.\ Cannon, and C.\ Playoust, \emph{The {M}agma algebra system. {I}.
  {T}he user language}, J. Symbolic Comput. \textbf{24} (1997), no.~3-4,
  235--265, Computational algebra and number theory (London, 1993).

\bibitem[Bea77]{Beauville-Prym}
A.\ Beauville, \emph{Vari\'{e}t\'{e}s de {P}rym et jacobiennes
  interm\'{e}diaires}, Ann. Sci. \'{E}cole Norm. Sup. (4) \textbf{10} (1977),
  no.~3, 309--391.

\bibitem[Bea00]{Beauville-determinantal}
\bysame, \emph{Determinantal hypersurfaces}, Michigan Math. J. \textbf{48}
  (2000), 39--64.

\bibitem[Bur83]{Burns}
D.\ Burns, \emph{On the geometry of elliptic modular surfaces and
  representations of finite groups}, Algebraic geometry ({A}nn {A}rbor,
  {M}ich., 1981), Lecture Notes in Math., vol. 1008, Springer, Berlin, 1983,
  pp.~1--29.

\bibitem[Cor00]{Corti2000}
A.\ Corti, \emph{Singularities of linear systems and {$3$}-fold birational
  geometry}, Explicit birational geometry of 3-folds, London Math. Soc. Lecture
  Note Ser., vol. 281, Cambridge Univ. Press, Cambridge, 2000, pp.~259--312.

\bibitem[CPS05]{CheltsovShramovPrzyjalkowski}
I.\ Cheltsov, V.\ Przyjalkowski, and C.\ Shramov, \emph{Hyperelliptic and
  trigonal {F}ano threefolds}, Izv. Ross. Akad. Nauk Ser. Mat. \textbf{69}
  (2005), no.~2, 145--204.

\bibitem[CPS19]{CheltsovPrzyjalkowskiShramov-BarBur}
\bysame, \emph{Burkhardt quartic, {B}arth sextic, and the icosahedron}, Int.
  Math. Res. Not. IMRN (2019), no.~12, 3683--3703.

\bibitem[CS09]{CheltsovShramov2009}
I.\ Cheltsov and C.\ Shramov, \emph{Extremal metrics on del {P}ezzo
  threefolds}, Tr. Mat. Inst. Steklova \textbf{264} (2009),  37--51.

\bibitem[CS12]{CheltsovShramov-PSL}
\bysame, \emph{Three embeddings of the {K}lein simple group into the {C}remona
  group of rank three}, Transform. Groups \textbf{17} (2012), no.~2, 303--350.

\bibitem[CS14]{CheltsovShramov-A6}
\bysame, \emph{Five embeddings of one simple group}, Trans. Amer. Math. Soc.
  \textbf{366} (2014), no.~3, 1289--1331.

\bibitem[CS16]{CheltsovShramov-V5}
\bysame, \emph{Cremona groups and the icosahedron}, Monographs and Research
  Notes in Mathematics, CRC Press, Boca Raton, FL, 2016.

\bibitem[dDM19]{Lucas}
L.\ das Dores and M.\ Mauri, \emph{G-birational superrigidity of {D}el {P}ezzo
  surfaces of degree 2 and 3}, Eur. J. Math. \textbf{5} (2019), no.~3,
  798--827.

\bibitem[DI09]{DolgachevIskovskikh}
I.\ Dolgachev and V.\ Iskovskikh, \emph{Finite subgroups of the plane {C}remona
  group}, Algebra, arithmetic, and geometry: in honor of {Y}u. {I}. {M}anin.
  {V}ol. {I}, Progr. Math., vol. 269, Birkh\"{a}user Boston, Inc., Boston, MA,
  2009, pp.~443--548.

\bibitem[Dix02]{Dixon}
A.\ Dixon, \emph{Note on the reduction of a ternary quantic to a symmetrical
  determinant}, Proc.\ Camb.\ Phil.\ Soc.\ \textbf{11} (1902), 350--351.

\bibitem[dJSBVdV90]{JSBV}
A.\ de~Jong, N.\ Shepherd-Barron, and A.\ Van~de Ven, \emph{On the {B}urkhardt
  quartic}, Math. Ann. \textbf{286} (1990), no.~1--3, 309--328.

\bibitem[DK93]{DolgachevKanev}
I.\ Dolgachev and V.\ Kanev, \emph{Polar covariants of plane cubics and
  quartics}, Adv. Math. \textbf{98} (1993), no.~2, 216--301.

\bibitem[DO88]{DolgachevOrtland}
I.\ Dolgachev and D.\ Ortland, \emph{Point sets in projective spaces and theta
  functions}, Ast\'{e}risque (1988), no.~165, 210 pp. (1989).

\bibitem[Dol99]{Dolgachev97}
I.\ Dolgachev, \emph{Invariant stable bundles over modular curves {$X(p)$}},
  Recent progress in algebra ({T}aejon/{S}eoul, 1997), Contemp. Math., vol.
  224, Amer. Math. Soc., Providence, RI, 1999, pp.~65--99.

\bibitem[Dol12]{Dolgachev}
\bysame, \emph{Classical algebraic geometry. A modern view}. Cambridge University Press,
  Cambridge, 2012. 

\bibitem[Dol16]{Dolgachev-Segre}
\bysame, \emph{Corrado {S}egre and nodal cubic threefolds}, From classical to
  modern algebraic geometry, Trends Hist. Sci., Birkh\"{a}user/Springer, Cham,
  2016, pp.~429--450.

\bibitem[Edg37a]{Edge1}
W.\ Edge, \emph{Notes on a {N}et of {Q}uadric {S}urfaces: {I}. {T}he {C}remona
  {T}ransformation}, Proc. London Math. Soc. (2) \textbf{43} (1937), no.~4,
  302--315.

\bibitem[Edg37b]{Edge2}
\bysame, \emph{{Notes on a Net of Quadric Surfaces: II. Anharmonic
  Covariants}}, J. London Math. Soc. (1) \textbf{12} (1937), no.~4, 276--280.

\bibitem[Edg38]{Edge3}
\bysame, \emph{Notes on a {N}et of {Q}uadric {S}urfaces: {III}. {T}he {S}croll
  of {T}risecants of the {J}acobian {C}urve}, Proc. London Math. Soc. (2)
  \textbf{44} (1938), no.~6, 466--480.

\bibitem[Edg41]{Edge4}
\bysame, \emph{Notes on a net of quadric surfaces. {IV}. {C}ombinantal
  covariants of low order}, Proc. London Math. Soc. (2) \textbf{47} (1941),
  123--141.

\bibitem[Edg42]{Edge5}
\bysame, \emph{Notes on a net of quadric surfaces. {V}. {T}he pentahedral net},
  Proc. London Math. Soc. (2) \textbf{47} (1942), 455--480.

\bibitem[EGH96]{EGH}
D.\ Eisenbud, M.\ Green, and J.\ Harris, \emph{Cayley--{B}acharach theorems and
  conjectures}, Bull. Amer. Math. Soc. (N.S.) \textbf{33} (1996), no.~3,
  295--324.

\bibitem[EP00]{EP00}
D.\ Eisenbud and S.\ Popescu, \emph{The projective geometry of the {G}ale
  transform}, J. Algebra \textbf{230} (2000), no.~1, 127--173.

\bibitem[EVdV81]{EisenbudVanDeVen}
D.\ Eisenbud and A.\ Van~de Ven, \emph{On the normal bundles of smooth rational
  space curves}, Math. Ann. \textbf{256} (1981), no.~4, 453--463.

\bibitem[Fin87]{Finkelnberg}
H.\ Finkelnberg, \emph{Small resolutions of the {S}egre cubic}, Nederl. Akad.
  Wetensch. Indag. Math. \textbf{49} (1987), no.~3, 261--277.

\bibitem[GH78]{GH}
P.\ Griffiths and J.\ Harris, \emph{Principles of algebraic geometry},
  Wiley-Interscience [John Wiley \& Sons], New York, 1978. 

\bibitem[Gri03]{Grinenko2003}
M.\ Grinenko, \emph{On the double cone over the {V}eronese surface}, Izv. Ross.
  Akad. Nauk Ser. Mat. \textbf{67} (2003), no.~3, 5--22.

\bibitem[Gri04]{Grinenko2004}
\bysame, \emph{Mori structures on a {F}ano threefold of index 2 and degree 1},
  Tr. Mat. Inst. Steklova \textbf{246} (2004), no.~Algebr. Geom. Metody, Svyazi
  i Prilozh., 116--141.

\bibitem[Hen76]{Henn}
P.\ Henn, \emph{Die {A}utomorphismengruppen der algebraischen
  {F}unktionenk{\"o}rper vom {G}eschlecht 3}, Inaugural-dissertation,
  Heidelberg, 1976.

\bibitem[Hes55]{Hesse}
O.\ Hesse, \emph{\"{U}ber {D}eterminanten und ihre {A}nwendung in der
  {G}eometrie, insbesondere auf {C}urven vierter {O}rdnung}, J. Reine Angew.
  Math. \textbf{49} (1855), 243--264.

\bibitem[IP99]{Isk-Prokh}
V.\ Iskovskikh and Yu.\ Prokhorov, \emph{Fano varieties}, Algebraic geometry,
  {V}, Encyclopaedia Math. Sci., vol.~47, Springer, Berlin, 1999, pp.~1--247.

\bibitem[JR06]{JahnkeRadloff}
P.\ Jahnke and I.\ Radloff, \emph{Gorenstein {F}ano threefolds with base points
  in the anticanonical system}, Compos. Math. \textbf{142} (2006), no.~2,
  422--432.

\bibitem[JvOS94]{Jeurissen}
R.\ Jeurissen, C.\ van Os, and J.\ Steenbrink, \emph{The configuration of
  bitangents of the {K}lein curve}, Discrete Math. \textbf{132} (1994),
  no.~1--3, 83--96.

\bibitem[Kaw01]{Kawakita}
M.\ Kawakita, \emph{Divisorial contractions in dimension three which contract
  divisors to smooth points}, Invent. Math. \textbf{145} (2001), no.~1,
  105--119.

\bibitem[Keu97]{Keum}
J.~H.\ Keum, \emph{Automorphisms of {J}acobian {K}ummer surfaces}, Compositio
  Math. \textbf{107} (1997), no.~3, 269--288.

\bibitem[KM98]{KollarMori98}
J.\ Koll\'{a}r and S.\ Mori, \emph{Birational geometry of algebraic varieties},
  Cambridge Tracts in Mathematics, vol. 134, Cambridge University Press,
  Cambridge, 1998, With the collaboration of C. H. Clemens and A. Corti,
  Translated from the 1998 Japanese original.

\bibitem[KSC04]{CoKoSm03}
J.\ Koll\'{a}r, K.\ Smith, and A.\ Corti, \emph{Rational and nearly rational
  varieties}, Cambridge Studies in Advanced Mathematics, vol.~92, Cambridge
  University Press, Cambridge, 2004.

  \bibitem[KK79]{KuribayashiKomiya}
 A.\ Kuribayashi and K.\ Komiya, \emph{On {W}eierstrass points and automorphisms
  of curves of genus three}, Algebraic geometry ({P}roc. {S}ummer {M}eeting,
  {U}niv. {C}openhagen, {C}openhagen, 1978), Lecture Notes in Math., vol. 732,
  Springer, Berlin, 1979, pp.~253--299.


\bibitem[Muk92]{Mukai92}
S.\ Mukai, \emph{Fano {$3$}-folds}, Complex projective geometry ({T}rieste,
  1989/{B}ergen, 1989), London Math. Soc. Lecture Note Ser., vol. 179,
  Cambridge Univ. Press, Cambridge, 1992, pp.~255--263.

\bibitem[Muk04]{Mukai04}
\bysame, \emph{Plane quartics and {F}ano threefolds of genus twelve}, The
  {F}ano {C}onference, Univ. Torino, Turin, 2004, pp.~563--572.

\bibitem[Ott13]{Ottaviani} 
G.\ Ottaviani, \emph{A computational approach to L\"uroth quartics}, 
Rend. Circ. Mat. Palermo (2) \textbf{62} (2013), no.~1, 165--177. 

\bibitem[Pro13]{Prokhorov}
Yu.\ Prokhorov, \emph{{$G$}-{F}ano threefolds, {I}}, Adv. Geom. \textbf{13}
  (2013), no.~3, 389--418.

\bibitem[Rei72]{Reid}
M.\ Reid, \emph{The complete intersection of two or more quadrics}, Ph.D.\
  dissertation, Cambridge University (1972).

\bibitem[RS00]{RanestadSchreyer}
K.\ Ranestad and F.-O.\ Schreyer, \emph{Varieties of sums of powers}, J. Reine
  Angew. Math. \textbf{525} (2000), 147--181.

\bibitem[Sal52]{Salmon}  
G.\ Salmon, \emph{A treatise on the higher plane curves: 
intended as a sequel to ``A treatise on conic sections.''} 
Dublin: Hodges \& Smith, 1852.

\bibitem[Tyu75]{Tyurin}
A.\ Tyurin, \emph{On intersections of quadrics}, Russian Mathematical Surveys
  \textbf{30} (1975), no.~6, 51--105.
  
\bibitem[W57]{Weil}
A.\ Weil, \emph{Beweis des Torellischen Satzes}, Nachr. Akad. Wiss. G\"ottingen. Math.-Phys. Kl. IIa. \textbf{1957} (1957), 
 33--53.

\bibitem[Wal78]{Wall1}
C.\ Wall, \emph{Nets of quadrics, and theta-characteristics of singular
  curves}, Philos. Trans. Roy. Soc. London Ser. A \textbf{289} (1978),
  no.~1357, 229--269.

\bibitem[Wal81]{Wall2}
\bysame, \emph{Singularities of nets of quadrics}, Compositio Math. \textbf{42}
  (1980/81), no.~2, 187--212.

\end{thebibliography}
\end{document}